\documentclass[11pt]{amsart} 
\usepackage[utf8]{inputenc}
\usepackage{amsmath,amsthm}
\usepackage{xypic}
\usepackage{enumerate}
\usepackage{lscape}
\usepackage{tikz-cd}
\usepackage{tikz}
\usepackage{color} 
\definecolor{darkred}{rgb}{1,0,0} 
\definecolor{darkgreen}{rgb}{0,1,0}
\definecolor{darkblue}{rgb}{0,0,1}
\usepackage{hyperref}
\hypersetup{colorlinks,
linkcolor=darkblue,
filecolor=darkblue,
urlcolor=darkred,
citecolor=darkgreen}

\usepackage{amscd}
\usepackage{amsfonts,amssymb,latexsym} 
\usepackage[all]{xy} 
\xyoption{arc}
\CompileMatrices

\newcommand{\udots}{\mathinner{\mskip1mu\raise1pt\vbox{\kern7pt\hbox{.}}
\mskip2mu\raise4pt\hbox{.}\mskip2mu\raise7pt\hbox{.}\mskip1mu}}

\newcommand{\SC}{{\mathcal{C}}}
\newcommand{\SD}{{\mathcal{D}}}
\newcommand{\SE}{{\mathcal{E}}}
\newcommand{\SF}{{\mathcal{F}}}

\newcommand{\SH}{{\mathcal{H}}}
\newcommand{\SI}{{\mathcal{I}}}

\newcommand{\SL}{{\mathcal{L}}}
\newcommand{\SM}{{\mathcal{M}}}

\newcommand{\SO}{{\mathcal{O}}}
\newcommand{\SP}{{\mathcal{P}}}

\newcommand{\ST}{{\mathcal{T}}}
\newcommand{\SU}{{\mathcal{U}}}

\newcommand{\SX}{{\mathcal{X}}}

\newcommand{\LL}{\mathbb{L}}

\newcommand{\ZZ}{\mathbb{Z}}

\newcommand{\CC}{\mathbb{C}}
\newcommand{\RR}{\mathbb{R}}

\newcommand{\HH}{\mathbb{H}}

\newcommand{\Ext}{\operatorname{Ext}}
\newcommand{\Gr}{\operatorname{Gr}}
\newcommand{\Spec}{\operatorname{Spec}}

\newcommand{\codim}{\operatorname{codim}}

\newcommand{\Hom}{\operatorname{Hom}}
\newcommand{\Aut}{\operatorname{Aut}}

\newcommand{\Sym}{\operatorname{Sym}}
\newcommand{\id}{\operatorname{Id}}
\newcommand{\im}{\operatorname{Im}}

\newcommand{\rk}{\operatorname{rk}}
\newcommand{\End}{\operatorname{End}}
\newcommand{\tr}{\operatorname{tr}}
\newcommand{\GL}{\operatorname{GL}}
\newcommand{\U}{\operatorname{U}}
\newcommand{\Id}{\operatorname{Id}}
\newcommand{\Jac}{\operatorname{Jac}}

\newcommand{\SSL}{\operatorname{SL}}

\newcommand{\VHS}{\operatorname{VHS}}
\newcommand{\HC}{\operatorname{HC}}
\newcommand{\coeff}{\mathop{\mathrm{coeff}}}
\newcommand{\VarC}{\mathcal{V}ar_{\CC}}
\newcommand{\Diff}{\operatorname{Diff}}
\newcommand{\Dol}{\operatorname{Dol}}
\newcommand{\Hod}{\operatorname{Hod}}

\newcommand{\fa}{\mathfrak{a}}

\newcommand{\op}{\operatorname}

\newtheorem{proposition}{Proposition}[section]
\newtheorem{theorem}[proposition]{Theorem}
\newtheorem{definition}[proposition]{Definition}

\newtheorem{lemma}[proposition]{Lemma}
\newtheorem{conjecture}[proposition]{Conjecture}
\newtheorem{corollary}[proposition]{Corollary}
\newtheorem{remark}[proposition]{Remark}

\newtheorem{example}[proposition]{Example}

\numberwithin{equation}{section}

\newcommand{\smtrx}[1]{\left (\begin{smallmatrix}#1\end{smallmatrix}\right)}

\definecolor{fillgrey}{rgb}{0.94, 0.94, 0.95}
\definecolor{arrowred}{rgb}{0.95, 0.95, 0.96}

\title[Lie algebroid connections, twisted Higgs bundles and motives of moduli spaces]{Lie algebroid connections, twisted Higgs bundles and motives of moduli spaces}
\author[D. Alfaya]{David Alfaya}

\address{D. Alfaya, 
\newline\indent
Department of Applied Mathematics and Institute for Research in Technology, ICAI School of Engineering, Comillas Pontifical University, C/Alberto Aguilera 25, 28015 Madrid, Spain}

\email{dalfaya@comillas.edu}

\author[A. Oliveira]{Andr\'e Oliveira}
\address{A. Oliveira, 
\newline\indent 
Centro de Matem\'atica da Universidade do Porto, Faculdade de Ci\^encias da Universidade do Porto, Rua do Campo Alegre s/n, 4169-007 Porto, Portugal \newline\indent \textsl{and}\newline\indent Departamento de Matem\'atica, Universidade de Tr\'as-os-Montes e Alto Douro, UTAD, 
Quinta dos Prados, 5000-911 Vila Real, Portugal}
\email{andre.oliveira@fc.up.pt,\ agoliv@utad.pt}


\begin{document}

\begin{abstract}
Let $\SL=(L,[\cdot\,,\cdot],\delta)$ be an algebraic Lie algebroid over a smooth projective curve $X$ of genus $g\geq 2$ such that $L$ is a line bundle whose degree is less than $2-2g$. Let $r$ and $d$ be coprime numbers. We prove that the motivic class of the moduli space of $\SL$-connections of rank $r$ and degree $d$ over $X$ does not depend on the Lie algebroid structure $[\cdot\,,\cdot]$ and $\delta$ of $\SL$ and neither on the line bundle $L$ itself, but only on the degree of $L$ (and of course on $r$, $d$ and $X$). In particular it is equal to the motivic class of the moduli space of $K_X(D)$-twisted Higgs bundles of rank $r$ and degree $d$, for $D$ any effective divisor with the appropriate degree. As a consequence, similar results (actually slightly stronger) are obtained for the corresponding $E$-polynomials.
Some applications of these results are then deduced.
\end{abstract}

\maketitle

\tableofcontents

\section{Introduction}

In this paper we consider moduli spaces of flat Lie algebroid connections on a compact hyperbolic Riemann surface $X$. Lie algebroid connections are closely related to Simpson's notion of $\Lambda$-modules \cite{Si94,Tor11, Tor12} and constitute a simultaneous generalization of several classes of geometric structures which are broadly used in differential geometry, algebraic geometry and mathematical physics \cite{CM04,LM12}, such as Higgs bundles \cite{HitchinHiggs, SimpsonHiggs, Si94}, twisted Higgs bundles \cite{Ni}, flat connections \cite{SimpsonHiggs} and logarithmic or meromorphic connections \cite{De,Ni93,Boalch02,BS13}. By approaching the study of all these different types of moduli spaces from the perspective of a single unifying framework, we can transfer some of the geometric properties and analysis techniques from the better known moduli spaces of Higgs bundles to other types of moduli spaces whose geometries are less understood (like moduli spaces of possibly irregular meromorphic connections). For a general choice of a Lie algebroid, the geometry of moduli spaces of Lie algebroid connections was mostly unknown, but here we obtain properties of their motives, including explicit formulas in low ranks, when the Lie algebroid is of rank one and degree strictly less than the Euler characteristic of $X$. In particular, we prove a motivic correspondence between moduli spaces of twisted Higgs bundles and of flat Lie algebroid connections. In fact, we prove that the motive of the moduli space of Lie algebroid connections does not depend on the algebroid structure nor on its underlying bundle, but just on its degree. These results then yield counterpart identifications of their $E$-polynomials (in particular, of their Betti numbers) as well as their Hodge structures. Along the way, we compute their dimension, prove smoothness and irreducibility. In addition, our techniques also allow us to compute various higher homotopy groups of these moduli spaces, and to verify computationally a conjecture by Mozgovoy on the motives of moduli spaces of twisted Higgs bundles \cite{Moz12} for small values of the invariants (genus, rank and degree of the twisting bundle). As applications of this analysis, we also obtain new information about the geometry of the moduli spaces of meromorphic connections and of the moduli spaces of $\U(2,2)$-Higgs bundles or $\U(3,3)$-Higgs bundles with maximal Toledo invariant, which are important in higher Teichmüller theory.

\subsection{On the moduli spaces of Lie algebroid connections}

A Lie algebroid $\SL=(V,[\cdot\,,\cdot],\delta)$ is a vector bundle $V\to X$ endowed with a $\CC$-linear Lie bracket $[\cdot\,,\cdot]:V\otimes_\CC V\to V$ and an anchor map $\delta:L\to T_X$ to the tangent bundle $T_X$, which are compatible with each other, in the sense that the bracket satisfies the Jacobi identity and a ``twisted'' version of Leibniz rule involving the anchor; see Definition \ref{def:Liealgebroid}. Given such a gadget, one defines an $\SL$-connection as a pair $(E,\nabla)$, where $E\to X$ is a vector bundle and $\nabla:E\to E\otimes V^*$ is a map which generalises the notion of a connection on $E$ in the sense that $\nabla$ satisfies a version of Leibniz rule in which the canonical differential is replaced by the differential of the Chevalley-Eilenberg-de Rham complex of the Lie algebroid $\SL$; cf.~Definition \ref{def:L-connect}. The first example of a Lie algebroid is the tangent bundle $T_X$ with the Lie bracket of vector fields and the identity anchor. Denote this Lie algebroid by $\mathcal{T}_X$. Then a $\mathcal{T}_X$-connection $(E,\nabla)$ is just a connection on $E$. 

Here we consider algebraic (i.e.\ holomorphic) Lie algebroids $\SL$ on $X$, algebraic $\SL$-connections and the corresponding moduli space $\SM_{\SL}(r,d)$ of flat $\SL$-connections on $X$ of rank $r$ and degree $d$ (the rank and the degree are the ones of the underlying bundle $E$). As we mentioned, for some choices of $\SL$, like the canonical Lie algebroid structures on the tangent bundle $\SL=\mathcal{T}_X$ or $T_X(-D)=T_X\otimes\mathcal{O}_X(-D)$ for an effective and reduced divisor $D$ on $X$, these moduli are classical spaces (moduli spaces of flat connections and integrable logarithmic connections respectively) which have been extensively studied; cf.~\cite{Si94,Ni93}. For a general $\SL$, the spaces $\SM_{\SL}(r,d)$ have been constructed by analytic methods in \cite{Kr08,Kr10} and they have also been identified with certain algebraic moduli spaces of $\Lambda$-modules in \cite{Si94,Tor11,Tor12}, proving that they are quasi-projective varieties. However, even though Lie algebroids are quite classical objects in geometry, very little is known a priori on the geometry and topology of $\SM_{\SL}(r,d)$. The following theorem, containing some of the results proven in this paper, constitutes a first step towards filling this gap.

\begin{theorem}\label{thm:geo.top-moduli-Lconnections}
 Let $X$ be a smooth projective complex curve of genus $g\geq 2$. Let $\SL=(L,[\cdot\,,\cdot],\delta)$ be an algebraic Lie algebroid on $X$ such that $L$ is a line bundle of degree $\deg(L)<2-2g$. Suppose $r,d$ are coprime integers. Then $\SM_{\SL}(r,d)$ is a non-empty, smooth and connected quasi-projective variety such that its:
\begin{enumerate}
\item complex dimension is $1-r^2\deg(L)$;
\item homotopy groups, up to order $2(g-1)(r-1)-2$ and for $(r,g)\neq (2,2)$, are given as follows:
\begin{enumerate}
\item $\pi_1(\SM_{\SL}(r,d))\cong H_1(X,\ZZ)\cong\ZZ^{2g}$;
\item $\pi_2(\SM_{\SL}(r,d))\cong \ZZ$;
\item $\pi_k(\SM_{\SL}(r,d))\cong \pi_{k-1}(\mathcal{G}(\mathbb{E}))$, for every $k=3,\ldots,2(g-1)(r-1)-2$. Here $\mathbb{E}$ is the unique (up to diffeomorphism) $C^\infty$ Hermitian vector bundle on $X$ of rank $r$ and degree $d$, and $\mathcal{G}(\mathbb{E})$ is the corresponding unitary gauge group.
\end{enumerate}
\item motivic class $[\SM_{\SL}(r,d)]$ in the Grothendieck ring of varieties depends on $\deg(L)$ but not on the remaining structure of the Lie algebroid $\SL$. The same is true for its Voevodsky and Chow motives. Moreover, for $r=2,3$, $[\SM_{\SL}(r,d)]$ is explicitly given by the formulas presented in Corollary \ref{cor:K-class moduli rk2,3}.
\item Hodge structure is pure and its $E$-polynomial $E(\SM_{\SL}(r,d))$ depends on $r$ and $\deg(L)$ but not on the remaining structure of the Lie algebroid $\SL$ nor on the degree $d$. Moreover, for $r=2,3$, $E(\SM_{\SL}(r,d))$ is explicitly given by the formulas presented in Corollary \ref{cor:E-poly_r=2,3}.
\end{enumerate}
\end{theorem}

Our main motivation for this paper was to study the Grothendieck motives of $\SM_\SL(r,d)$ and items (3) and (4) of the previous theorem may be seen as the main results of the paper.

The above listed topological and geometrical properties are obtained as we now roughly describe (contextualised details are provided below). We show that there is a family of moduli spaces, parameterised by $\CC$, such that the moduli space given by any $z\in\CC^*$ is isomorphic to $\SM_{\SL}(r,d)$, while the one given by $0\in\CC$ is isomorphic to the twisted Dolbeault moduli space $\SM^{\Dol}_{L^{-1}}(r,d)$ of $L^{-1}$-twisted Higgs bundles of rank $r$ and degree $d$ over $X$ (here $L^{-1}$ is the dual of $L$). We prove that this family is actually a moduli space --- the $\SL$-Hodge moduli space --- which is a smooth semiprojective variety. Using this fact we then prove that we can deduce properties of its fibres over a non-zero complex number, thus of $\SM_{\SL}(r,d)$, from properties over the zero fibre, i.e.\ $\SM^{\Dol}_{L^{-1}}(r,d)$. Since this latter space has been extensively studied, there are plenty of algebraic tools to study its geometry and topology, in particular its motives, and we apply them here. We now provide more details on these techniques.  

\subsection{The classical case: moduli spaces of flat connections and of Higgs bundles}

Let $\mathcal{T}_X$ denote the canonical Lie algebroid on the tangent bundle $T_X$ of the compact Riemann surface $X$. In \cite{Si94}, Simpson provided a geometric invariant theory construction of the moduli space of rank $r$ and degree $0$ holomorphic flat connections $(E,\nabla)$, also known as local systems, on $X$. Denote it by $\SM^{\op{dR}}(r,0)$ reflecting the fact that Simpson named it as the de Rham moduli space. Note that in terms of Lie algebroid connections, and using the notation employed in the previous section, $\SM^{\op{dR}}(r,0)=\SM_{\mathcal{T}_X}(r,0)$.

In the same paper, Simpson also constructed what he called the Dolbeault moduli space, namely the moduli of rank $r$ and degree $0$ Higgs bundles over $X$. These are pairs $(E,\varphi)$ where $E$ is an algebraic rank $r$ and degree $0$ vector bundle and $\varphi:E\to E\otimes K_X$ an algebraic bundle map --- the Higgs field --- where $K_X=T_X^*$ is the holomorphic cotangent bundle, i.e.\ the canonical line bundle of $X$. Denote this moduli as $\SM^{\Dol}_{K_X}(r,0)$ or simply as $\SM^{\Dol}(r,0)$.

As proved in \cite{Si94}, the moduli spaces of Higgs bundles and of connections, $\SM^{\Dol}(r,0)$ and $\SM^{\op{dR}}(r,0)$, are, respectively, the fibre over $0$ and $1$ of an algebraic family over $\CC$, such that any other fibre is also isomorphic to $\SM^{\op{dR}}(r,0)$. This family, called the Hodge moduli space, is obtained by taking the moduli space $\SM^{\Hod}_\lambda(r,0)$ of $\lambda$-connections $(E,\nabla,\lambda)$ of rank $r$ and degree $0$, where $\lambda\in\CC$. A $\lambda$-connection is an object very similar to a connection but where, in Leibniz rule, the summand with the differential gets multiplied by the scalar  $\lambda$. More precisely, for every local section $s$ and every local regular complex function $f$ the $\lambda$-connection $\nabla:E\to E\otimes \Omega^1_X$ satisfies
$$\nabla(fs)=f\nabla(s)+\lambda s df$$
Thus $\SM^{\Hod}_0(r,0)\cong\SM^{\Dol}(r,0)$ and $\SM^{\Hod}_1(r,0)\cong\SM^{\op{dR}}(r,0)$. By letting $\lambda$ vary in $\CC$, Simpson then constructed a variety $\SM^{\Hod}(r,0)$ endowed with the projection $\SM^{\Hod}(r,0)\to\CC$ mapping to $\lambda$, yielding the mentioned family.

It turns out that $\SM^{\Dol}(r,0)$ and $\SM^{\op{dR}}(r,0)$ are singular moduli spaces. One can work with smooth spaces by taking instead $d$ coprime with $r$ and the corresponding moduli spaces as follows. For Higgs bundles, let $\SM^{\Dol}(r,d)$ be the moduli of degree $d$ and rank $r$ Higgs bundles. For connections, fix any point $x\in X$ and consider the moduli space $\SM^{\op{dR}}(r,d)$ of logarithmic connections, where poles at $x\in X$ are allowed, but with fixed residue depending on $d$ \cite{De,Ni93}. 
A complete analogue of the picture of the preceding paragraph holds as well. Hence there is a Hodge moduli space $\SM^{\Hod}(r,d)$ parameterising logarithmic $\lambda$-connections certain given fixed residues which has a natural mapping to $\CC$ so that the fibre over $0$ is $\SM^{\Dol}(r,d)$ and any other fibre is isomorphic to $\SM^{\op{dR}}(r,d)$. 

For any $d$ (possibly $0$), the family $\SM^{\Hod}(r,d)\to\CC$ admits certain preferred sections --- the real twistor lines --- which realise the famous non-abelian Hodge correspondence \cite{HitchinHiggs,SimpsonHiggs,Si94,Si95}, providing a homeomorphism between $\SM^{\Dol}(r,d)$ and $\SM^{\op{dR}}(r,d)$.

However, when $r$ and $d$ are coprime, Hausel and Thaddeus proved in \cite[Theorem 6.2]{HT03} that $\SM^{\Dol}(r,d)$ and $\SM^{\op{dR}}(r,d)$ actually share deep geometrical invariants other than just topological ones. Namely, both spaces share the same $E$-polynomial and their Hodge structures are pure\footnote{This was actually done in \cite{HT03} for the fixed determinant (and traceless) versions of these moduli spaces, but the argument also works in the non-fixed determinant case.}. An important feature of their proof is the use of the fact that, as proved in \cite{Si94, Si97}, $\SM^{\Hod}(r,d)$ is endowed with a natural $\CC^*$-action with proper fixed point sets and such that every orbit has a limit when $\CC^*\ni z\to 0$. This means by definition that the Hodge moduli is a semiprojective variety \cite{HRV15}.

More recently, it was observed that the semiprojectivity of  $\SM^{\Hod}(r,d)$ implies the existence of another fundamental correspondence between $\SM^{\Dol}(r,d)$ and $\SM^{\op{dR}}(r,d)$. Namely, in \cite{HL21}, Hoskins and Lehaleur established what they called a ``motivic non-abelian Hodge correspondence'' by proving an equality of the Voevodsky motives of  $\SM^{\Dol}(r,d)$ and $\SM^{\op{dR}}(r,d)$. In fact, by considering $d=0$ and the stacky version of these moduli spaces, a similar result was proved before in \cite{FSS17}, for motivic classes in the Grothendieck ring of stacks, but through a completely different technique, namely point counting. This was recently generalised to the parabolic setting in \cite{FSSsigma}.

\subsection{This paper's contribution}

In the present paper, we generalise the results mentioned in the previous section to the case of any rank one Lie algebroid of degree less than that of $T_X$ and we show that the motives and E-polynomials of these moduli spaces actually exhibit broader invariance properties with respect to the choice of the Lie algebroid used in the construction of the moduli space.

In order to understand better the types of moduli spaces of Lie algebroid connections which arise from rank one Lie algebroids, we obtain a full classification of isomorphism classes of Lie algebroid structures on line bundles and derive non-emptiness results on the corresponding moduli spaces of Lie algebroid connections (cf.~Proposition \ref{prop:moduliNonEmpty}, Theorem \ref{thm:classificationLieAlgebroidsRank1}, Corollary \ref{cor:classificationLieAlgebroidsRank1} and Corollary \ref{cor:classificationLieAlgebroidsRank1Meromorphic}).
\begin{theorem}
\label{thm:classificationLieAlgebroidsRank1-intro}
Let $X$ be a smooth complex projective curve and let $L$ be a line bundle on $X$. Then there is a one-to-one correspondence between isomorphism classes of Lie algebroid structures $\SL=(L,[\cdot\,,\cdot],\delta)$ on $L$ and orbits of the scaling $\CC^*$-action on $H^0(L^{-1}\otimes T_X)$ sending
$$\lambda \cdot \delta = \lambda\delta \quad \forall \lambda \in \CC^* \,\, \forall \delta\in H^0(L^{-1}\otimes T_X).$$
Moreover:
\begin{itemize}
\item If $\deg(L)>2-2g$, or $\deg(L)=2-2g$ but $L\ne T_X$, then the only Lie algebroid structure admitted by $L$ is the trivial one $\SL=(L,0,0)$. Thus, the only existing moduli spaces of $\SL$-connections are moduli spaces of $L^{-1}$-twisted Higgs bundles.
\item $L=T_X$ only admits two nonisomorphic Lie algebroid structures: the trivial one $(T_X,0,0)$ and the canonical one given by the Lie bracket of vector fields $\ST_X=(T_X,[\cdot\,,\cdot]_{\op{Lie}},\id)$. Thus, the only existing non-empty moduli spaces of $\SL$-connections are the classical moduli spaces $\SM^{\Dol}(r,d)$ of Higgs bundles or $\SM^{\op{dR}}(r,0)$ of flat connections.  
\item If $\deg(L)<2-2g$ and $\SL=(L,[\cdot\,,\cdot],\delta)$ is any Lie algebroid structure on $L$, then the moduli space $\SM_\SL(r,d)$ of semistable flat $\SL$-connections is always non-empty. A similar non-emptiness result is obtained if $\SL$ is any intransitive Lie algebroid on a vector bundle $V$ of arbitrary rank.
\end{itemize}
As a consequence, if $L$ is a line bundle then for each Lie algebroid structure $\SL=(L,[\cdot,\cdot],\delta)$ on $L$, the moduli space of $\SL$-connections is isomorphic to either the moduli space of:
\begin{itemize}
\item $L^{-1}$-twisted Higgs bundles,
\item flat connections,
\item integrable logarithmic connections with poles over a fixed reduced effective divisor $D$, or
\item integrable meromorphic connections with poles prescribed by a fixed non-reduced effective divisor $D$.
\end{itemize}
In the last two items, $D$ lies in the linear system associated to the line bundle $L^{-1}\otimes T_X$.
\end{theorem}

We will now give a general overview on the proof of Theorem \ref{thm:geo.top-moduli-Lconnections}, together with the main results leading to it. But before we make a remark  regarding also the preceding theorem.

\begin{remark}
It is important to stress here that the classification Theorem \ref{thm:classificationLieAlgebroidsRank1-intro} plays no role in the proof of Theorem \ref{thm:geo.top-moduli-Lconnections}. Their proofs are completely independent. In other words, Theorem \ref{thm:geo.top-moduli-Lconnections} is proved purely in the setup of Lie algebroid connections for any `abstract' rank one Lie algebroid $\SL$ under the conditions stated in loc.~cit.  Actually, some of the results we prove along the way to Theorem \ref{thm:geo.top-moduli-Lconnections} hold for higher rank Lie algebroids, while others, although are not proved for higher rank, are expected to be valid in this more general setup as well (see for instance Remark \ref{rmk:higherrk}).
\end{remark}

The main idea behind the proof of Theorem \ref{thm:geo.top-moduli-Lconnections}, and in particular of the motivic invariance properties stated there, can then be summarized as follows. In principle, the moduli space $\SM_\SL(r,d)$ is expected to vary when the Lie algebroid $\SL$ changes. Nevertheless, we prove that many topological properties and motivic invariants of the moduli space actually remain constant with respect to the choice of $\SL$. If we look at a trivial Lie algebroid structure $\SL=(L,0,0)$, so that $\SM_\SL(r,d)$ corresponds to a moduli space of twisted Higgs bundles, we could distinguish two complementary ways in which $\SL$ could change. On one hand, one could keep the underlying bundle $L$ fixed and modify the Lie algebroid structure on it to be a nontrivial one. According to Theorem \ref{thm:classificationLieAlgebroidsRank1-intro}, this corresponds to moving across the space of $\CC^*$-orbits of the vector space $H^0(L^{-1}\otimes T_X)$ and going through different kinds of moduli spaces of logarithmic and meromorphic connections with poles over different divisors on $X$, which can instead be interpreted as generalized connections sharing the same fixed twisting bundle $L^{-1}$. This corresponds on moving the moduli spaces over each $H^0(L^{-1}\otimes T_X)$ and $H^0(L'^{-1}\otimes T_X)$ in Figure \ref{fig:mainidea} below. On the other hand, one could change $L$ into a different line bundle $L'$ of the same degree $d_L$, so moving in the corresponding `Jacobian' of degree $d_L$ line bundles on  $\Jac^{d_L}(X)$ of $X$, but now keeping the trivial Lie algebroid structure. This means considering just moduli spaces of twisted Higgs bundles, but for different twisting bundles. We prove, through different techniques, that both types of modifications to the Lie algebroid structure leave the motives and other properties of the moduli spaces unchanged, thus proving that they remain invariant for all the moduli spaces $\SM_\SL(r,d)$ independently on $\SL$.

\begin{figure}[h!]
\begin{center}
\begin{tikzpicture}[scale=.9,domain=0:4] 
  \draw[thick] (-5.5, -4.5) -- (-5.5, 4.5);
  \node[left] at (-5.5,0) {${\scriptstyle \mathrm{Jac}^{d_L}(X)}$};
  \fill (-5.5,2.5) circle (2pt);
  \fill (-5.5,-2.5) circle (2pt);
  \node[left] at (-5.5,2.5) {${\scriptscriptstyle L}$};
  \node[left] at (-5.5,-2.5) {${\scriptscriptstyle L'}$};
  \coordinate (A1) at (-4,-3.5);
  \coordinate (A2) at (4,-3.5);
  \coordinate (A3) at (6,-1.5);
  \coordinate (A4) at (-2,-1.5);
  \draw[thick,fill=fillgrey] (A1)--(A2)--(A3)--(A4)--cycle;
  \draw[domain=-1.5:3.5,blue,variable=\x] plot ({\x}, {-1/3*\x-13/6});
      \node at (1+.25,-2.5+.35) {${\scriptscriptstyle \SM^{\Dol}_{(L')^{-1}}(r,d)}$};
       \node[below,blue] at (1,-2.5-.05) {${\scriptscriptstyle 0}$};
      \node at (0-0.85,-13/6-0.1) {${\scriptscriptstyle \SM_{\SL'}(r,d)}$};
            \node[blue,above,rotate=-19] at (2.5+.2,-3-.15) {${\scriptscriptstyle  \SM_{\SL'}^{\Hod}(r,d)}$};
      \node[left,blue] at (-1.5,-5/3-.1) {${\scriptscriptstyle \mathbb{C}}$};
  
   \fill[black] (1,-2.5) circle (2pt);
  \fill[black] (0,-13/6) circle (2pt);
  
    \node at (6,-3) {${\scriptscriptstyle H^0((L')^{-1}\otimes T_X)}$};
    
     \draw[<-,dashed] (-5.4,-2.5)--(-3,-2.5);

  \coordinate (B1) at (-4,1.5);
  \coordinate (B2) at (4,1.5);
  \coordinate (B3) at (6,3.5);
  \coordinate (B4) at (-2,3.5);
  \draw[thick,fill=fillgrey] (B1)--(B2)--(B3)--(B4)--cycle;

  \draw[domain=-3:5,blue,variable=\x] plot ({\x}, {0.2*\x-2.7+5});
      \node at (1-.15,-2.5-.4+5) {${\scriptscriptstyle \SM^{\Dol}_{L^{-1}}(r,d)}$};
      \node[above,blue] at (1,-2.5+.1+5) {${\scriptscriptstyle 0}$};
      \node at (3+.5,-2.1+5-.25) {${\scriptscriptstyle \SM_\SL(r,d)}$};
        \node[blue,above,rotate=10] at (-1.5,-3+5-.05) {${\scriptscriptstyle  \SM_{\SL}^{\Hod}(r,d)}$};
         \node[left,blue] at (-3,-33/10+.1+5) {${\scriptscriptstyle \mathbb{C}}$};

        \fill[black] (1,-2.5+5) circle (2pt);
  \fill[black] (3,-2.1+5) circle (2pt);

    \node at (6,-3+5) {${\scriptscriptstyle H^0(L^{-1}\otimes T_X)}$};
    
         \draw[<-,dashed] (-5.4,-2.5+5)--(-3,-2.5+5);
         
                  \draw[<->,dashed]  (1,-1.75)--(1,1.75);
                    \node[right] at (1,0) {\scriptsize  \begin{tabular}{c}Bialynicki-Birula \\stratification analysis\end{tabular}};

\end{tikzpicture}
\end{center}
\caption{For any pair of line bundles $L$ and $L'$ of the same degree $d_L$, all moduli spaces of Lie algebroid connections parameterized by both  spaces $H^0(L^{-1}\otimes T_X)$ and $H^0((L')^{-1}\otimes T_X)$ share the same motives (and other properties).}
\label{fig:mainidea}
\end{figure}

We address first the invariance with respect to the Lie algebroid structure by considering what we call $\SL$-Hodge moduli spaces. These are a generalization of the classical Hodge moduli space which provides families of moduli spaces of $\SL$-connections which degenerate into a moduli space of twisted Higgs bundles in a sufficiently regular way.

More precisely, out of a given Lie algebroid $\SL=(V,[\cdot\,,\cdot ],\delta)$ on $X$, one can consider flat $(\lambda,\SL)$-connections $(E,\nabla_\SL,\lambda)$ for each $\lambda\in\CC$, these being the analogues of the above mentioned $\lambda$-connections corresponding to $\SL=\ST_X$. Then a flat $(0,\SL)$-connection is just a $V^*$-twisted Higgs bundle (with an integrability condition on the Higgs field) on $X$; note that it is the same as a Lie algebroid connection by taking the trivial Lie algebroid structure on $V$. On the other hand, a flat $(1,\SL)$-connection is a flat $\SL$-connection. Then the generalised version of the above Hodge moduli space is the $\SL$-Hodge moduli space $\SM^{\Hod}_\SL(r,d)$ of flat $(\lambda,\SL)$-connections on $X$, where a $(\lambda,\SL)$-connection is the same thing as a Lie algebroid connection for the Lie algebroid structure obtained from $\SL$ by multiplying the Lie bracket and the anchor by $\lambda$. 

These moduli spaces are constructed as a particular case of a general construction by Simpson \cite{Si94}, arising from moduli spaces $\SM_{\Lambda}(r,d)$ of $\Lambda$-modules (see Definition \ref{def:Lambda-module}), where $\Lambda$ is a sheaf of rings of differential operators on $X$. As proved by Tortella in \cite{Tor11,Tor12}, given a Lie algebroid $\SL$ on $X$, there is an equivalence of categories between integrable $\SL$-connections and $\Lambda_\SL$-modules, where $\Lambda_\SL$ a certain sheaf of rings of differential operators uniquely corresponding to $\SL$. Then this equivalence yields an isomorphism between the $\Lambda_\SL$-modules moduli space $\SM_{\Lambda_\SL}(r,d)$ and the moduli  $\SM_{\SL}(r,d)$ of flat $\SL$-connections. 

From $\Lambda_\SL$ one then constructs a sheaf of rings of differential operators $\Lambda_\SL^{\op{red}}$ on $X\times \CC$ relative to $\CC$ and obtain a moduli space $\SM_{\Lambda_\SL^{\op{red}}}(r,d)\to \CC$ of $\Lambda_\SL^{\op{red}}$-modules which can be identified with the desired $\SL$-Hodge moduli space $\SM_{\SL}^{\Hod}(r,d)$. This moduli is equipped with a natural map $\SM_{\SL}^{\Hod}(r,d)\to\CC$, $(E,\nabla_\SL,\lambda)\mapsto\lambda$, whose fibre over zero is hence the moduli space of $V^*$-twisted Higgs bundles of rank $r$ and degree $d$, and every other fibre is isomorphic to the fibre over $1$, namely to the moduli of integrable $\SL$-connections of rank $r$ and degree $d$. Then we prove following (see Theorem \ref{thm:semiprojectiveHodge}).

\begin{theorem}\label{thm:smothsemiproj}
Let $\SL=(L,[\cdot\,,\cdot],\delta)$ be a Lie algebroid such that $L$ is a line bundle with $\deg(L)<2-2g$, where $g\geq 2$ is the genus of $X$. If $(r,d)=1$, then the $\SL$-Hodge moduli space $\SM_{\SL}^{\Hod}(r,d)\cong \SM_{\Lambda_\SL^{\op{red}}}(r,d)$ is a smooth semiprojective variety. Moreover the projection  $\SM_{\SL}^{\Hod}(r,d)\to\CC$ is a $\CC^*$-equivariant surjective submersion.
\end{theorem}

Part of the proof of this theorem includes the computation of the dimension of $\SM_\SL(r,d)$, establishing (1) of Theorem \ref{thm:geo.top-moduli-Lconnections}. We also stress that the two equivalent interpretations of the ``same object'' --- $\SL$-connections and $\Lambda_{\SL}$-modules --- are actually required in this proof. For instance, in order to prove smoothness (which is the most intricate part to prove), we make use of the deformation theory for integrable $\SL$-connections, developed in \cite{Tor12}, since deformation theory for $\Lambda$-modules is not yet well-understood in the required generality.  On the other hand, the proof of the existence of limits of a natural $\CC^*$-action on the $\SL$-Hodge moduli (which is a condition for being semiprojective) is carried out by explicitly using $\Lambda_\SL$-modules rather than $\SL$-connections.

The geometric and topological properties listed in Theorem \ref{thm:geo.top-moduli-Lconnections} are obtained by exploiting the setting arising from Theorem 
 \ref{thm:smothsemiproj}. For instance, we prove the following general result by making use of Bialynicki-Birula stratifications \cite{{BB73}} (cf.~Theorem \ref{thm:semiprojectiveMotive}) and their relations with motives.

\begin{theorem}
\label{thm:semiprojectiveMotive-int}
Let $Y$ be a smooth semiprojective variety together with a surjective $\CC^*$-equivariant submersion $\pi: Y \to \CC$. Consider the motivic classes of $\pi^{-1}(0)$ and $\pi^{-1}(1)$ in the completion of the Grothendieck ring of varieties $\hat{K}(\VarC)$. Then such classes coincide.
\end{theorem}
 
Theorem \ref{thm:semiprojectiveMotive-int} allows us to shift the study of the motives of $\SM_\SL(r,d)$ to that of the zero fibre of the $\SL$-Hodge moduli, namely to the moduli space $\SM^{\Dol}_{L^{-1}}(r,d)$ of $L^{-1}$-twisted Higgs bundles. Then, working directly on $\SM^{\Dol}_{L^{-1}}(r,d)$ and studying the Bialynicki-Birula decomposition of the moduli space, we prove that the motives of these twisted Higgs moduli spaces do not depend on the line bundle $L$ but rather on its degree (see Figure \ref{fig:mainidea} above). Combining both results, we prove item (1) of the theorem below (cf.~Theorem \ref{thm:equalMotive}). The other items are proved within the same spirit (cf.~Theorem \ref{thm:equalMotive-Chow-Voevodsky}).

\begin{theorem}\label{thm:motivicclasses}
 Let $X$ be a smooth projective complex curve of genus $g\geq 2$. Let $\SL=(L,[\cdot\,,\cdot ],\delta)$ and $\mathcal{L'}=(L',[\cdot\,,\cdot ]',\delta')$ be any two Lie algebroids on $X$, such that $L$ and $L'$ are line bundles with  $\deg(L)=\deg(L')<2-2g$. Suppose $(r,d)=1$. Then
$$\SI(\SM_{\SL}(r,d)) = \SI(\SM_{\SL'}(r,d)),$$ 
where $\SI(X)$ denotes one of the following:
\begin{enumerate}
\item The motivic class $[X]\in \hat{K}(\VarC)$;
\item The Voevodsky motive $M(X)\in \op{DM}^{\op{eff}}(\CC,R)$ for any ring $R$. In this case, moreover, the motives are pure;
\item The Chow motive $h(X)\in \op{Chow}^{\op{eff}}(\CC,R)$ for any ring $R$;
\item The Chow ring $\op{CH}^\bullet(X,R)$ for any ring $R$.
\end{enumerate}
Moreover, the mixed Hodge structures of the moduli spaces are pure and, if $d'$ is any integer coprime with $r$, then 
$$E(\SM_{\SL}(r,d)) = E(\SM_{\SL'}(r,d')).$$
Finally, if $L=L'=T_X(-D)$ for some effective divisor $D$, then there is an actual isomorphism of pure mixed Hodge structures
$$H^\bullet(\SM_{\SL}(r,d)) \cong  H^\bullet(\SM_{\SL'}(r,d')).$$
\end{theorem}

The above theorem is the main result of the paper. It is a more precise and expanded version of parts of items (3) and (4) of Theorem \ref{thm:geo.top-moduli-Lconnections}. Its motivic statements constitute a generalization of the ``motivic non-abelian Hodge correspondence'' of \cite{HL21} to any Lie algebroid $\SL$ on the given conditions.

According to Theorem \ref{thm:classificationLieAlgebroidsRank1-intro}, the above theorem applies precisely to the moduli spaces of logarithmic connections (without fixed residues on the poles) and meromorphic connections (without fixing the corresponding Stokes data), by taking $\ST_X(-D)\subset \ST_X$, for $D$ an is an effective reduced (logarithmic case) or non-reduced (meromorphic case) divisor on $X$. Hence, we have the following direct application of the above theorem (see Corollary \ref{cor:motivlog=irreg}), which then deals with both situations (logarithmic and meromorphic) at the same time.

\begin{corollary}
Let $X$ be as above. Then all the motivic classes, Chow ring, $E$-polynomial and Hodge structure mentioned in Theorem \ref{thm:motivicclasses}, of any moduli space of \emph{meromorphic} connections on $X$, of coprime rank and degree, equal that of any moduli space of \emph{logarithmic} connections $X$ of the same rank and degree, with singular divisor of the same degree.
\end{corollary}

The remaining parts of the statements (3) and (4) of  Theorem \ref{thm:geo.top-moduli-Lconnections} are also obtained using Theorem \ref{thm:motivicclasses}. Indeed, we establish explicit formulas for the motivic classes and for the $E$-polynomials of the moduli spaces of $L^{-1}$-twisted Higgs bundles for $r=2,3$ and so, by Theorem \ref{thm:motivicclasses}, the same formulas also hold for  $\SM_{\SL}(r,d)$; cf.~Corollaries \ref{cor:K-class moduli rk2,3} and \ref{cor:E-poly_r=2,3}. We also note that Corollary \ref{cor:K-class moduli rk2,3} should prove Conjecture 3 of \cite{Moz12} (giving a general formula for the motive in any rank; see Conjecture \ref{conj:motivicADHM}) for $r=2,3$. We verified, by computer methods, that this is indeed the case for both $r=2,3$, $g\leq 11$ and $-18-2g\leq\deg(L)\leq 1-2g$ (see Section \ref{sec:Mozgovoy} and \cite{A22}). These formulas also allow us to compute the motives and $E$-polynomials of the maximal components of the moduli spaces of $\U(2,2)$- and $U(3,3)$-Higgs bundles for degrees coprime with the rank (see Theorem \ref{thm:UnnMotive}).

The homotopy groups computation of (2) of Theorem \ref{thm:geo.top-moduli-Lconnections} is also obtained in a similar fashion. We use Theorem \ref{thm:smothsemiproj} together with the Bialynicki-Birula stratifications of both the $\SL$-Hodge moduli $\SM_{\SL}^{\Hod}(r,d)$ and the $L^{-1}$-twisted Higgs bundles moduli $\SM^{\Dol}_{L^{-1}}(r,d)$ to show that there is a big open subset  of the $\SL$-connections moduli $\SM_\SL(r,d)$ which retracts onto the moduli space of vector bundles and whose homotopy groups are the ones listed in Theorem \ref{thm:geo.top-moduli-Lconnections}. The codimension computation of such open set yields the result; see Theorem \ref{thm:homotopygrps}.

Finally, we also give some correspondent results in the fixed determinant setting.

\subsection{Organization}

Here is a brief description of the contents of the paper.

In Section \ref{section:Higgs} we give a quick introduction to $V$-twisted Higgs bundles.

In Section \ref{section:LambdaModules} we introduce Lie algebroids, sheaves of rings of differential operators, $\SL$-connections and $\Lambda$-modules. Following \cite{Tor11,Tor12}, we give a short overview of the equivalence of categories between integrable $\SL$-connections and $\Lambda_\SL$-modules, where $\Lambda_\SL$ is a sheaf of rings of differential operators associated to $\SL$. This is done over any base variety. Then we introduce the moduli spaces of $\Lambda_\SL$-modules / integrable $\SL$-connections over a curve $X$, we prove some non-emptiness results about them, and we introduce the $\SL$-Hodge moduli spaces. Finally, as the work is mainly focused on rank 1 Lie algebroids, an explicit classification of Lie algebroid structures on line bundles and a subsequent description of the corresponding moduli spaces of Lie algebroid connections are obtained.

Sections \ref{sec:smoothn-semiproj-L-Hodge} to \ref{section:HiggsPolynomials} are the core of the paper. In Section \ref{sec:smoothn-semiproj-L-Hodge} we prove the first main result of the paper, namely the fact that the $\SL$-Hodge moduli space is smooth and semiprojective, under the given conditions. In Section \ref{section:motives} we introduce the Grothendieck ring of varieties and motives. Then we prove a general result on motives of smooth semiprojective varieties and apply it to show that the motivic class of the moduli spaces of Lie algebroid connections equals that of the corresponding twisted Higgs bundle moduli space. We also show that similar results hold for the $E$-polynomials and for the Hodge structures. In Section \ref{section:HiggsPolynomials} we show that the motivic class of the twisted Higgs bundles moduli is independent of the twisting (in the sense that it only depends on its degree) and hence the motive of the moduli spaces of Lie algebroid connections only depend on the degree of the line bundle underlying the Lie algebroid. Similar (actually, stronger) results are also proven for $E$-polynomials and Hodge structures.

In Section \ref{section:mainResults}, we combine the results from the previous sections and a study of the geometry of the Bialynicki-Birula stratification of the $\SL$-Hodge moduli space to complete the proof of the main motivic invariance results of the paper announced in Theorem \ref{thm:motivicclasses}, including versions of the main theorems for Voevodsky motives, Chow motives and Chow rings, as well as the results on the topology of the moduli spaces of Lie algebroid connections described in Theorem \ref{thm:geo.top-moduli-Lconnections}. As an application, we present some conclusions on the topology and motivic classes of the moduli spaces of meromorphic and logarithmic connections.

In Section \ref{section:motiveComputations} we further apply the previous geometric analysis to derive explicit formulas for the motives and E-polynomials of the moduli spaces in ranks $2$ and $3$. Using computer-assisted techniques, we utilize these formulas to verify Mozgovoy's conjecture on the motive of twisted Higgs bundles in low genus, rank and degree of the twisting bundle. As an application, the formulas also yield the motives and E-polynomials of the maximal components of the moduli spaces of $U(2,2)$ and $U(3,3)$ Higgs bundles.

Finally, in Section \ref{section:fixedDeterminant} we show how to achieve similar results for the same moduli spaces of $\SL$-connections $(E,\nabla_\SL)$ but where the determinant of $E$ and the trace of $\nabla_\SL$ are fixed.

\subsection{Open questions and future work}

The techniques we employ and the results we obtain give rise to some natural questions. First of all, in this paper, we have used the rich geometry of the degenerating families given by $\SL$-Hodge moduli spaces to derive some properties about the moduli space of $\SL$-connections like connectivity and smoothness, but we believe that a deeper study on these structures can actually be used to obtain many other geometrical properties of these moduli spaces.

Moreover, Theorem \ref{thm:smothsemiproj} opens the door for a whole new ``playground'' to explore questions vastly generalising those studied in the classical Hodge moduli space $\SM^{\Hod}(r,d)$, for example like the study of opers or of conformal limits. 

At the ``further developments'' section of Tortella's thesis \cite[Section 7.3]{Tor11}, Tortella suggested that, if there existed a versal family of Lie algebroids on the space of Lie algebroid structures on $T_X$, then understanding the geometry the corresponding family of moduli spaces of Lie algebroid connections could ``encode important information on the moduli space of Higgs bundles''. The results from this paper reinforce that idea and takes some concrete steps into that program. Theorem \ref{thm:classificationLieAlgebroidsRank1-intro} shows that, in rank $1$, the space $H^0(L^{-1}\otimes T_X)$ provides a possible encoding of all possible deformations of the trivial Lie algebroid structure. In this work we have only considered deformations of the trivial Lie algebroids in ``radial directions'' in that space (corresponding to divisors in the linear system $\mathbb{P}H^0(L^{-1}\otimes T_X)$), but we could, as well, vary the Lie algebroid structures in any other non-trivial way (see Remark \ref{rmk:defo-hypercoh}). The construction done in the proof of that theorem can actually be used to build a versal family of deformations of the trivial Lie algebroid and, through Simpson's construction \cite{Si94}, a versal family of moduli spaces of Lie algebroid connections could be built over $H^0(L^{-1}\otimes T_X)$ representing all possible deformations of the twisted Higgs moduli corresponding to Lie algebroid deformations. We believe that a further analysis on these deformations can shed some new light on the exceptional geometric properties of the twisted Higgs moduli spaces, as well as on the degeneration of moduli spaces of logarithmic connections into moduli spaces of meromorphic connections.
 
On the other hand, even though we prove some results in the case of higher rank Lie algebroids, mostly concerning non-emptiness of the moduli spaces of $\SL$-connections (cf.\ Section \ref{sec:moduli-LandLambda}), our main results hold for Lie algebroids whose underlying bundle is a line bundle. However, we expect that at least some of them to hold for higher rank Lie algebroids. The main obstacle to extend the results to this setting is that a good knowledge of deformation theory of $V$-Higgs bundles, for $\rk(V)>1$, has not yet been completely  achieved (see Remark  \ref{rmk:higherrk}). 

\vspace{.5cm}

\noindent\textbf{Acknowledgments.} 
This research was funded by MICINN (grants MTM2016-79400-P, PID2019-108936GB-C21, PID2022-142024NB-I00, RED2022-134463-T  and ``Severo Ochoa Programme for Centres of Excellence in R\&D'' SEV-2015-0554 and CEX2019-000904-S) and by CMUP -- Centro de Matem\'atica da Universidade do Porto -- under the project with reference UIDB/00144/2020, and by the project ``Mirror symmetry on Higgs bundles moduli spaces'' -- EXPL/MAT-PUR/1162/2021, both financed by national funds through FCT -- Funda\c{c}\~ao para a Ci\^encia e a Tecnologia, I.P.. The first author was also supported by a postdoctoral grant from ICMAT Severo Ochoa project. He would also like to thank the hospitality of CMUP during the research visits which took place in the course of the development of this work. Finally, we thank Luis \'Alvarez-C\'onsul, Philip Boalch, Carlos Florentino, Jos\'e \'Angel Gonz\'alez-Prieto, Peter Gothen, Sergey Mozgovoy, Vicente Mu\~noz and Jaime Silva for helpful discussions, and we also thank the referee for useful comments. 

\section{Moduli spaces of twisted Higgs bundles}
\label{section:Higgs}

Throughout the paper we will only be considering algebraic objects (vector bundles, Lie algebroids, connections, etc.) on smooth projective varieties over $\CC$, this being implicitly assumed whenever the corresponding adjective is missing. We will also always take the usual identification between algebraic vector bundles and locally free sheaves.

Let $Y$ be a smooth complex projective variety and let $V$ be an algebraic vector bundle on $Y$, thus a locally free $\SO_Y$-module. 

\begin{definition}\label{def:Higgsbundle}
A \emph{$V$-twisted Higgs bundle} on $Y$ is a pair $(E,\varphi)$ consisting by an algebraic vector bundle $E$ on $Y$ and a map of $\SO_Y$-modules
$$\varphi:E\longrightarrow E\otimes V,$$
called the \emph{Higgs field}, such that $\varphi\wedge\varphi=0\in H^0(\End(E)\otimes \Lambda^2 V)$. 
\end{definition}

If $\varphi\in H^0(\End(E)\otimes V)$, by $\varphi\wedge\varphi\in H^0(\End(E)\otimes \Lambda^2 V)$ in the previous definition we mean the following. Let $p:V\otimes V\to\Lambda^2V$ be the quotient map. Then 
$$\varphi\wedge\varphi=(\id_E\otimes p)\circ(\varphi\otimes\id_V)\circ\varphi.$$ A local version is given as follows. Take an open set $U\subset Y$ where both $E$ and $V$ are trivialised, and let $w_1,\ldots, w_k$ be a local trivialising basis of $V$; then we can write
$\varphi|_U=\sum_i G_i\otimes w_i$, with $G_i$ a local section of $\End(E)$, so an $\SO_Y(U)$-valued matrix. Then
$$(\varphi\wedge\varphi)|_U=\sum_{i<j}[G_i,G_j]\otimes w_i\wedge w_j.$$

A map between $V$-twisted Higgs bundles $(E,\varphi)$ and $(E',\varphi')$ is an algebraic map $f:E\to E'$ such that $(f\otimes\id_V)\circ\varphi=\varphi'\circ f$, and if there is such an $f$ which is an isomorphism, then $(E,\varphi)$ and $(E',\varphi')$ are \emph{isomorphic}.

Observe that the integrability condition $\varphi\wedge\varphi=0$, implies that providing a Higgs field on $E$ is equivalent to endowing a $\Sym^\bullet(V^*)$-module structure on $E$ compatible with the $\SO_Y$-module structure.

Consider now the case where $Y$ is a smooth projective curve $X$ of genus at least $2$. This is going to be the case we will mostly consider, but we will also allow $Y$ to be any smooth projective variety in some particular cases. Let also $K_X$ denote the canonical line bundle of $X$.

Higgs bundles were introduced by Hitchin in \cite{HitchinHiggs} for $V=K_X$, in the context of gauge theory, namely as objects which are naturally associated to solutions to the now known as Hitchin equations.

Given a vector bundle $E$ over $X$, let
$$\mu(E):= \frac{\deg(E)}{\rk(E)}$$
be the \emph{slope} of $E$. Then $E$ is called \emph{(semi)stable} if for every subbundle $0\ne F\subsetneq E$ we have
$$\mu(F)<\mu(E) \quad (\text{resp.}~\le).$$
Similarly, a $V$-twisted Higgs bundle $(E,\varphi)$ on $X$ is \emph{(semi)stable} if for every subbundle $0\ne F\subsetneq E$ preserved by $\varphi$ (i.e.\ such that $\varphi(F)\subseteq F\otimes V$), we have 
$$\mu(F)<\mu(E) \quad (\text{resp.}~\le).$$

By \emph{degree} of a $V$-twisted Higgs bundle $(E,\varphi)$ on the curve $X$, we mean the degree of the underlying bundle $E$. In analogy with the classical $K_X$-twisted case, let $\SM^{\Dol}_V(r,d)$ denote the moduli space of ($S$-equivalence classes of) semistable $V$-twisted Higgs bundles $(E,\varphi)$ on $X$ of rank $r$ and degree $d$. When $V$ is a line bundle this moduli space was first constructed, via GIT, by Nitsure \cite{Ni}. A gauge theoretic construction for $V=K_X$ is given in \cite{HitchinHiggs}. In general, $\SM^{\Dol}_V(r,d)$ exists as a consequence of Simpson's GIT construction of the moduli space of $\Lambda$-modules \cite{Si94} (the general notion of $\Lambda$-module will be reviewed in Section \ref{section:LambdaModules}). It is a quasi-projective complex algebraic variety and, moreover, we have the following Lemma which is a direct consequence of \cite[Proposition 3.3 and Theorem 1.2]{BGL11}.

\begin{lemma}
\label{lemma:smoothModuliHiggs}
Suppose that $r\ge 1$ and $d$ are coprime and that $L$ is a line bundle with $\deg(L)>2g-2$, where $g\ge 2$ is the genus of $X$. Then the moduli space $\SM^{\Dol}_L(r,d)$ is a smooth connected, so irreducible, quasi-projective variety of dimension $1+r^2\deg(L)$.
\end{lemma}

\begin{remark}
There are choices of the twist $V$ which do not satisfy the assumptions of the Lemma and for which $\SM^{\Dol}_V(r,d)$ is nonetheless a smooth variety. For instance, the classical case $V=K_X$ from \cite{HitchinHiggs} and which was also considered in \cite[Proposition 3.3]{BGL11}. 

However, a condition similar to the one presented in the previous lemma is to be expected in general since there exist many examples of low degree twists for which the moduli space is not smooth and even not irreducible.

For example, if $L=\SO_X(x)$ for a curve $X$ with genus at least 4 and $x\in X$, then $\SM^{\Dol}_{\SO_X(x)}(r,d)$ is, in general, a singular scheme. For generic stable vector bundles $E$, we have $H^0(\End(E)(x))\cong \CC$ by \cite[Lemma 2.2]{BGM13} so $\SM^{\Dol}_{\SO_X(x)}(r,d)$ contains an open subvariety $\SU\subset \SM^{\Dol}_{\SO_X(x)}(r,d)$ which is a line bundle over the open locus of such bundles in the moduli space $\mathrm{\mathbf{M}}(r,d)$ of stable vector bundles. Nevertheless, suppose that we pick the curve $X$ so that there exists a stable vector bundle $F$ on $X$ with a nonzero section $\varphi \in H^0(\End_0(F)(x))$, $\varphi\ne 0$. Then $(F,\varphi)\in \SM^{\Dol}_{\SO_X(x)}(r,d)$ will not belong to the closure of $\SU$, so the moduli space will not be irreducible, and the point $(F,0)\in \SM^{\Dol}_{\SO_X(x)}(r,d)$ will belong to two different irreducible components, making the moduli space a singular scheme.
\end{remark}

Consider a $V$-twisted Higgs bundle $(E,\varphi)$ on $X$. Since $\varphi\wedge\varphi=0$, the Higgs field $\varphi:E\to E\otimes V$ induces maps
$$\wedge^i \varphi: \Lambda^i E  \longrightarrow \Lambda^i E \otimes S^iV,$$
for $i=1,\ldots,r$. If we define
$$s_i=\tr(\wedge^i\varphi)\in H^0(S^iV),$$ then this yields the \emph{Hitchin map}
\begin{equation}\label{eq:Hitchinmap}
H:\SM^{\Dol}_V(r,d)\longrightarrow \bigoplus_{i=1}^r H^0(S^iV), \ \ \ H(E,\varphi)=(s_1,\ldots,s_r).
\end{equation}
We call $W:=\bigoplus_{i=1}^r H^0(S^iV)$ the \emph{Hitchin base}.

\begin{lemma}
\label{lemma:HitchinProper}
The Hitchin map $H:\SM^{\Dol}_V(r,d)\to W$ is proper.
\end{lemma}
\begin{proof}
When $V$ is a line bundle, this was proven by Nitsure \cite[Theorem 6.1]{Ni}. In the case $V=K_X$, besides being proven by Hitchin \cite{HitchinHiggs}, it was also proven by Faltings \cite[Theorem I.3]{Fal93} following an argument by Langton, and the same argument works word by word for an arbitrary twisting $V$.
\end{proof}

\section{Lie algebroid connections, $\Lambda$-modules and moduli spaces}
\label{section:LambdaModules}

\subsection{Lie algebroids and $\SL$-connections}

Let $Y$ be a smooth complex projective variety with tangent bundle $T_Y$. We will be mostly interested in the case $Y$ is the smooth projective curve $X$, but we will actually need to consider a higher dimensional variety in  sections \ref{sec:L-Hodge moduli} and \ref{section:semiprojectivityHodge}.

\begin{definition}\label{def:Liealgebroid}
An algebraic \emph{Lie algebroid} on $Y$ is a triple $\SL=(V,[\cdot\,,\cdot],\delta)$ consisting on
\begin{itemize}
\item an algebraic vector bundle $V$ on $Y$,
\item a $\CC$-bilinear and skew-symmetric map $[\cdot\,,\cdot]:V\otimes_\CC V \to V$, called the \emph{Lie bracket},
\item a vector bundle map $\delta:V\to T_Y$, called the \emph{anchor},
\end{itemize}
satisfying the following properties:
\begin{enumerate}
\item $[u,[v,w]]+[v,[w,u]]+[w,[u,v]]=0$\ \emph{(Jacobi identity)},
\item $[u,fv]=f[u,v]+\delta(u)(f) v$\ \emph{(Leibniz rule)},
\end{enumerate}
for any local sections $u,v,w$ of $V$ and any local function $f$ in $\SO_Y$.

The \emph{rank} of a Lie algebroid $\SL$, denoted by $\rk(\SL)$, is the rank of the underlying vector bundle $V$.
\end{definition}

\begin{example}\mbox{}
\begin{enumerate}
\item The canonical example of a Lie algebroid is the tangent bundle $T_Y$, together with the Lie bracket of vector fields and the identity anchor. Denote it as $\ST_Y=(T_Y,[\cdot\,,\cdot]_{\op{Lie}},\id)$.
\item More generally, an algebraic foliation $F_Y\subset T_Y$ gives also rise to a Lie algebroid $\SF_Y$ simply by restricting the Lie bracket $[\cdot\,,\cdot]_{\op{Lie}}$ and by taking the inclusion $F_Y\hookrightarrow T_Y$ as the anchor map.
\item Any algebraic vector bundle can be seen as a Lie algebroid with the zero bracket and the zero map as the anchor; this is called a \emph{trivial algebroid}.
\end{enumerate}

\end{example}

The Lie bracket of a Lie algebroid $\SL=(V,[\cdot\,,\cdot],\delta)$  endows $V$ with a structure of a sheaf of $\CC$-Lie algebras, which is not a sheaf of $\SO_Y$-Lie algebras unless the anchor $\delta$ is zero. 

A \emph{Lie algebroid map} $f:\SL\to\SL'$ between $\SL=(V,[\cdot\,,\cdot]_V,\delta_V)$ and $\SL'=(V',[\cdot\,,\cdot]_{V'},\delta_{V'})$ is an algebraic $\CC$-Lie algebra bundle map $f:V\to V'$ such $\delta_{V'}\circ f=\delta_V$.
For example, the anchor $\delta: V \to T_Y$ is a Lie algebroid map $\delta:\SL\to\ST_Y$. A \emph{Lie algebroid isomorphism} is a Lie algebroid map which is an isomorphism of the underlying bundles; in that case, the Lie algebroids are said to be \emph{isomorphic}.

 Let $\SL=(V,[\cdot\,,\cdot],\delta)$ be a Lie algebroid. We now define a differential on the complex of exterior powers $\Omega_\SL^\bullet=\Lambda^\bullet V^*$,
$$d_\SL: \Omega_\SL^k \longrightarrow \Omega_\SL^{k+1},$$
generalising the classical de Rham complex $d:\Omega_Y^k\to \Omega_Y^{k+1}$ on $\Omega_Y^\bullet =\Lambda^\bullet T^*_Y$. In degree $0$, define $d_\SL : \SO_Y \to V^*$ as the composition of the canonical differential $d:\SO_Y \to T_Y^*$ with the dual of the anchor, $\delta^t: T^*_Y \to V^*$. Thus, given $v\in V$ and $f$ a local algebraic function on $Y$,
\begin{equation}\label{eq:d_L(f)}
d_\SL(f)(v)=df(\delta(v))=\delta(v)(f).
\end{equation}
The map $d_\SL:\SO_Y\to V^*$ is clearly a $V^*$-valued derivation. Now we extend it to higher order exterior powers through the usual recursive equation, but using the anchor map $\delta$.
For $\omega\in \Omega_\SL^n$ and $v_1, \ldots,v_{n+1}$ local sections of $V$, take
\begin{equation}\label{eq:d_L(omega)}
\begin{split}
d_\SL(\omega)(v_1,\ldots,v_{n+1}) =& \sum_{i=1}^{n+1} (-1)^{i+1} \delta(v_i)(\omega(v_1,\ldots,\hat{v_i},\ldots,v_{n+1}))\\
&+\sum_{1\le i<j\le {n+1}} (-1)^{i+j} \omega([v_i,v_j],v_1,\ldots,\hat{v_i},\ldots,\hat{v_j},\ldots,v_{n+1}).
\end{split}
\end{equation}
This differential satisfies $d_\SL^2=0$, so $(\Omega_\SL^\bullet,d_\SL)$ is a complex, called the \emph{Chevalley--Eilenberg--de Rham complex} of $\SL$. Note that $d_\SL=0$ for a trivial algebroid. 

From $d_\SL$ we now define a generalization of an algebraic connection.

\begin{definition}\label{def:L-connect}
Let $\SL=(V,[\cdot\,,\cdot],\delta)$ be a Lie algebroid on $Y$. An \emph{$\SL$-connection} on $Y$ is a pair $(E,\nabla_\SL)$, where $E$ is an algebraic vector bundle and where $\nabla_\SL$ is a $\CC$-linear algebraic vector bundle map
$$\nabla_\SL:E \longrightarrow E \otimes \Omega_\SL^1=E\otimes V^*,$$
such that
$$\nabla_\SL(fs)=f\nabla_\SL(s) + s\otimes d_\SL(f),$$
for $s$ a local section of $E$ and $f$ a local algebraic function on $Y$. The \emph{rank} of an $\SL$-connection is the rank of the underlying algebraic vector bundle.
\end{definition}

We will also refer to $\nabla_\SL:E \to E \otimes \Omega_\SL^1$ as an \emph{$\SL$-connection on the algebraic vector bundle $E$}.

Note that, for a trivial algebroid $\SL$, the map $\nabla_\SL:E\to E\otimes\Omega_\SL^1$ is actually  $\SO_Y$-linear rather than just $\CC$-linear.

Any $\SL$-connection $(E,\nabla_\SL)$ can be extended to a map
$$\nabla_\SL: E\otimes \Omega_\SL^\bullet \longrightarrow E\otimes \Omega_\SL^{\bullet+1},$$
as
$$\nabla_\SL(s\otimes \omega)=\nabla_\SL(s) \wedge \omega + s\otimes d_\SL(\omega).$$

\begin{definition}
Let  $(E,\nabla_\SL)$ be an $\SL$-connection on $Y$.
The composition $\nabla_\SL^2:E\to E\otimes \Omega_\SL^2$ is called the \emph{curvature} of $(E,\nabla_\SL)$. An $\SL$-connection is \emph{flat} if its curvature vanishes.
\end{definition}

\begin{example}\label{eg-Liealgebroids}\mbox{}
\begin{enumerate}
\item Recall the canonical Lie algebroid $\ST_Y=(T_Y,[\cdot\,,\cdot]_{\op{Lie}},\id)$ given by the tangent bundle of $Y$. Then, a flat $\ST_Y$-connection is just a flat algebraic connection in the usual sense. 
\item If $\SL$ is a  trivial algebroid, i.e.~$\SL=(V,0,0)$, then a flat $\SL$-connection is simply a pair $(E,\nabla_\SL)$ formed by a vector bundle $E$ and an $\SO_Y$-linear map $\nabla_\SL:E\to E\otimes V^*$ such that $\nabla_\SL\wedge \nabla_\SL=0$. This is precisely a $V^*$-twisted Higgs bundle from Definition \ref{def:Higgsbundle}.
\end{enumerate}
\end{example}

Alternatively, we can think of $\SL$-connections, for any Lie algebroid $\SL=(V,[\cdot\,,\cdot],\delta)$, from a slightly different point of view. An $\SL$-connection $\nabla_\SL:E\to E\otimes \Omega_\SL^1$ induces an $\SO_Y$-linear map
\begin{equation}\label{eq:barnabla}
\bar{\nabla}_\SL:V\longrightarrow \End_\CC(E),\ \ \ \ v \mapsto \nabla_{\SL,v},
\end{equation}
which, by \eqref{eq:d_L(f)}, satisfies
\begin{equation}\label{eq:Leibniz-nabla_v}
\nabla_{\SL,v}(fs)=f\nabla_{\SL,v}(s)+\delta(v)(f) s.
\end{equation}
In particular, identifying the local function $f\in \SO_Y$ with the endomorphism $E\to E$, $s\mapsto fs$, we have that $\nabla_{\SL,v}\circ f - f \circ \nabla_{\SL,v}$ is $\SO_Y$-linear,
$$\nabla_{\SL,v}\circ f - f \circ \nabla_{\SL,v} \in \End_{\SO_Y}(E).$$
Hence, $\nabla_{\SL,v}$ is a section of the locally free sheaf $\Diff^1(E)$ of differentials of order at most $1$.
Note that, since $\End_\CC(E)$ is a sheaf of associative algebras, the vector bundle $\Diff^1(E)$ inherits canonically a Lie algebroid structure $\SD(E)=(\Diff^1(E),[\cdot\,,\cdot]_{\SD(E)},\delta_{\SD(E)})$ through the commutator, by taking
$$[A,B]_{\SD(E)}=AB-BA\in \Diff^1(E),$$
for each $A,B\in \Diff^1(E)$, and, by making use of the $\SO_Y$-module structure of $\Diff^1(E)$,
$$\delta_{\SD(E)}(A)(f)=Af-fA\in\SO_Y\subseteq \End_{\SO_Y}(E),$$
for every  $f\in \SO_Y$.
\begin{definition}
The $\SL$-connection $\nabla_\SL$ is said to be \emph{integrable} if
$$[\nabla_{\SL,u},\nabla_{\SL,v}]_{\SD(E)}=\nabla_{\SL,[u,v]},$$
for all $u,v\in V$.
\end{definition}

\begin{remark}
It follows from Section 2.3 of \cite{Tor12} that an $\SL$-connection is flat if and only if it is integrable. Hence, in the remaining part of the paper we will freely interchange both terms.
\end{remark}

Now, it follows by \eqref{eq:Leibniz-nabla_v} that, for every $v$,
$$\delta_{\SD(E)}(\nabla_{\SL,v})=\delta(v).$$
We conclude that the map \eqref{eq:barnabla} can be canonically upgraded to a Lie algebroid map $\bar\nabla_\SL:\SL\to\SD(E)$, sometimes called a \emph{representation} of $\SL$ in $\SD(E)$, if and only if $\nabla_\SL$ is integrable.
Thus, the correspondence $\nabla_\SL\mapsto\bar\nabla_\SL$ yields a bijective correspondence
\begin{equation}\label{eq:intL-conn-repsinDE}
\big\{\text{integrable }\SL\text{-connections on }E\big\}\longleftrightarrow\big\{\text{representations of $\SL$ in $\SD(E)$}\big\}.
  \end{equation}

\subsection{$\Lambda$-modules and $\SL$-connections}
In \cite[\S 3]{Tor11} and \cite[\S 4]{Tor12}, Tortella proved that there is a very close relation between integrable $\SL$-connections and Simpson's notion of $\Lambda$-modules \cite{Si94}, to be introduced below. This will be important for us to deduce properties of the corresponding moduli spaces.

Let $S$ be a smooth variety and let $p:\SX\to S$ be a smooth variety over $S$. We start with the notion of sheaf of rings of differential operators as defined in \cite{Si94}.

\begin{definition}
\label{def:SRDO}
A \emph{sheaf of rings of differential operators} on $\SX$ over $S$ is a sheaf of $\SO_\SX$-algebras $\Lambda$ over $\SX$, with a filtration by subalgebras $\Lambda_0 \subseteq \Lambda_1 \subseteq \cdots\subseteq\Lambda$, verifying the following properties:
\begin{enumerate}
\item $\Lambda=\bigcup_{i=0}^{\infty} \Lambda_i$ and $\Lambda_i \cdot \Lambda_j \subseteq \Lambda_{i+j}$, for every $i,j$;
\item\label{n2} $\Lambda_0=\SO_\SX$;
\item the image of $p^{-1}(\SO_S)$ in $\SO_\SX$ lies in the center of $\Lambda$;
\item the left and right $\SO_\SX$-module structures on $\Gr_i(\Lambda):=\Lambda_i/\Lambda_{i-1}$ are equal;
\item the sheaves of $\SO_\SX$-modules $\Gr_i(\Lambda)$ are coherent;
\item the morphism of sheaves
$$\Gr_1(\Lambda) \otimes \cdots \otimes \Gr_1(\Lambda) \to \Gr_i(\Lambda)$$
induced by the product is surjective.
\end{enumerate}

Moreover, $\Lambda$ is said to be \emph{polynomial} if $\Lambda\cong \Sym^\bullet(\Gr_1(\Lambda))$ and \emph{almost polynomial} if $\Gr_\bullet(\Lambda) \cong \Sym^\bullet(\Gr_1(\Lambda))$.
\end{definition}

\begin{remark}
Simpson's definition is a slightly more general in point \eqref{n2}, allowing $\Lambda_0$ to be a quotient of $\SO_\SX$ and, therefore, allowing $\Lambda$ to be supported on a subscheme of $\SX$. However, we will only be interested on sheaves of rings supported on the whole scheme $\SX$.
\end{remark}

\begin{definition}\label{def:Lambda-module}
Let $\Lambda$ be a sheaf of rings of differential operators over $\SX$, flat over $S$. A \emph{$\Lambda$-module} is a pair $(E,\nabla_\Lambda)$ consisting of an algebraic vector bundle $E$ on $\SX$, flat on $S$, with an action $\nabla_\Lambda:\Lambda\otimes_{\SO_\SX} E\to E$ satisfying the usual module conditions. Moreover, the $\SO_\SX$-module structure on $E$ induced by the inclusion $\SO_\SX\subseteq\Lambda$ coincides with its inherent $\SO_\SX$-module structure as a locally free sheaf. The \emph{rank} of a $\Lambda$-module is the rank of the underlying bundle.
\end{definition}

This notion can be generalized by allowing $E$ to be a coherent $\SO_\SX$-module, but in this paper we stick to locally free case. 

The notions of maps and isomorphisms between $\Lambda$-modules over $\SX$ are the obvious ones, as in the Lie algebroids and Higgs bundles cases.

Let us now briefly recall the relation between integrable  $\SL$-connections and $\Lambda$-modules. For details, see \cite[\S 3]{Tor11} and \cite[\S 4]{Tor12}. 

Fix a Lie algebroid $\SL=(V,[\cdot\,,\cdot],\delta)$ on $Y$. Consider the associated Lie algebroid $$\widehat\SL=(\SO_Y\oplus V,[\cdot\,,\cdot]_1,\delta_1)$$ given by
$$[(f,u),(g,v)]_1 = (\delta(u)(g)-\delta(v)(f), [u,v]),\ \ \ \ \delta_1(f,u)=\delta(u),$$ 
with $u,v \in V$ and $\ f,g \in \SO_Y$.
Note that this means that, if $\SO_Y$ also denotes the trivial algebroid $(\SO_Y,0,0)$, then we have a \emph{split} short exact sequence 
\begin{equation}\label{eq:sesLiealg}	
0\longrightarrow \SO_Y\longrightarrow\widehat\SL\longrightarrow\SL\longrightarrow 0
\end{equation}
of Lie algebroids, where by this we mean $\widehat\SL\cong\SO_Y\oplus\SL$ as Lie algebroids or, equivalently that there is a splitting $\zeta:\SL\to\widehat\SL$ of the sequence \eqref{eq:sesLiealg}, which in this case is simply given by $\zeta(v)=(0,v)$, $v\in V$.

From the Lie algebroid $\widehat\SL$, hence from $\SL$, there is an associated almost polynomial sheaf of rings of differential operators $\Lambda_\SL$ whose weight $1$ piece $\Lambda_{\SL,1}\subset \Lambda_\SL$ in the corresponding filtration is isomorphic to $\SO_Y\oplus V$. 
This is roughly constructed as follows. The universal enveloping algebra of the Lie algebra $(\SO_Y\oplus V,[\cdot\,,\,\cdot]_1)$ is the sheaf of $\SO_Y$-algebras, defined by
$$U=T^\bullet( \SO_Y \oplus V) / \left\langle x\otimes y -y\otimes x - [x,y]_1  \, | \,  x,y\in \SO_Y\oplus V \right\rangle,$$
with $T^\bullet( \SO_Y \oplus V)$ being the tensor algebra on $\SO_Y\oplus V$.
If $i:\SO_Y\oplus V \hookrightarrow U$ is the canonical inclusion, then let $U^\dagger \subset U$ be the subalgebra generated by $i(\SO_Y\oplus V)$ and take
$$\Lambda_\SL = U^\dagger /\left ( i(f,0)\cdot i(g,v) - i(fg,fv)   \, \middle | \,   f,g\in \SO_Y, v\in V \right ).$$
This is the \emph{universal enveloping algebra of the Lie algebroid $\SL$}.
Then the graded structure from the tensor algebra $T^\bullet (\SO_Y \oplus V)$ induces a filtered algebra structure on $\Lambda_\SL$ which satisfies the axioms of an almost polynomial sheaf of rings of differential operators.
We thus get the correspondence 
\begin{equation}\label{eq:Liealg->srdo}
\SL\mapsto\Lambda_\SL.
\end{equation}

Conversely, take an almost polynomial sheaf of rings of differential operators $\Lambda$. Then the $\SO_Y$-module $\Lambda_1\subset \Lambda$ inherits a Lie algebroid structure in the following way \cite[Proposition 28]{Tor11} \cite[\S 4.1]{Tor12}.
Given local sections $u,v$ of $\Lambda_1$ and $f$ of $\SO_Y$, define
\begin{equation}\label{eq:LiealgLambda1}
[u,v]_{\Lambda_1}=uv-vu\ \ \ \ \ \text{ and }\ \ \ \ \ \delta_{\Lambda_1}(u)(f)=uf-fu,
\end{equation}
using the product in $\Lambda$. Let us prove that $[u,v]_{\Lambda_1}\in \Lambda_1$ and that $\delta_{\Lambda_1}(u)(f)\in \SO_Y$. As $\Lambda$ is almost polynomial, then $\Gr_\bullet(\Lambda)$ is abelian. Let $\op{sb}:\Lambda \to \Gr_\bullet(\Lambda)$ be the symbol map sending an element to the highest graded class where it is nonzero. One can check that $\op{sb}$ is multiplicative. We have
$$\op{sb}(uv)-\op{sb}(vu)=\op{sb}(u)\op{sb}(v)-\op{sb}(v)\op{sb}(u)=0\in \Gr_\bullet(\Lambda)$$
As $uv,vu\in \Lambda_2$, we conclude that $uv\equiv vu \mod \Lambda_1$, that is, $uv-vu\in \Lambda_1$, proving the first claim. For the second claim, we use the fact that the left and right $\SO_Y$-module structures on $\Gr_1(\Lambda)=\Lambda_1/\Lambda_0$ agree. Hence $uf-fu\in \Lambda_0=\SO_Y$ and the second claim follows.
Now, using the associativity of $\Lambda$ and the fact that elements of $\SO_Y$ commute with all elements of $\Lambda$, one sees that $[\cdot\,,\cdot]_{\Lambda_1}$ is $\SO_Y$-bilinear, skew-symmetric and verifies the Jacobi identity and that $\delta_{\Lambda_1}$ is so that Leibniz rule also holds. So \eqref{eq:LiealgLambda1} gives the mentioned Lie algebroid structure on $\Lambda_1$. 

Now, it is clear that the induced Lie algebroid structure on $\SO_Y=\Lambda_0\subset\Lambda_1$ is trivial, thus \eqref{eq:LiealgLambda1} descends to give a Lie algebroid structure on $\Gr_1(\Lambda)$, which we denote by $\SL_\Lambda$:
$$\SL_\Lambda=(\Gr_1(\Lambda),[\cdot\,,\cdot]_{\Gr_1(\Lambda_1)},\delta_{\Gr_1(\Lambda_1)}).$$
Moreover, we have the following short exact sequence of Lie algebroids
\begin{equation}\label{eq:sesLiealg2}
0\longrightarrow (\SO_X,0,0) \longrightarrow (\Lambda_1,[\cdot\,,\cdot]_{\Lambda_1},\delta_{\Lambda_1}) \longrightarrow \SL_\Lambda \longrightarrow 0.
\end{equation}

Thus, we have the correspondence 
\begin{equation}\label{eq:srdo->Liealg}
\Lambda\mapsto\SL_\Lambda.
\end{equation}

\begin{definition}\label{def:split}
An almost polynomial sheaf of rings of differential operators $\Lambda$ is \emph{split} if \eqref{eq:sesLiealg2} splits as a short exact sequence of Lie algebroids.
\end{definition}

By \cite[Theorem 4.4]{Tor12},  $\Lambda$ is split almost polynomial if and only if it is isomorphic to the universal enveloping algebra of $\Gr_1(\Lambda)$. Hence, the upshot of this discussion is the following theorem.
\begin{theorem}\label{thm:equiv:Liealge-sapsrdo}
The correspondences \eqref{eq:Liealg->srdo} and \eqref{eq:srdo->Liealg} induce inverse equivalences of categories:
\begin{equation}\label{eq:Liealg<->srdo}
\begin{split}
\left\{\begin{tabular}{c}
      \text{isomorphism classes of} \\
      \text{Lie algebroids on $Y$}
    \end{tabular}\right\}&\longleftrightarrow\left\{\begin{tabular}{c}
      \text{isomorphism classes of split} \\
      \text{almost polynomial sheaves of rings}  \\
      \text{of differential operators on $Y$}
      \end{tabular}\right\}.\\
\SL&\longmapsto\Lambda_\SL
\end{split}
\end{equation} 
\end{theorem}

Moreover, \eqref{eq:Liealg<->srdo} induces a correspondence between flat $\SL$-connections and $\Lambda_\SL$-modules. This goes roughly as follows (see again the above mentioned references for details).

Let $\SL=(V,[\cdot\,,\cdot],\delta)$ be a Lie algebroid with corresponding sheaf $\Lambda_\SL$. So we have a short exact sequence $$0\to\SO_Y\to\Lambda_{\SL,1}\to\SL\to 0$$ of Lie algebroids, with 
$$\Lambda_{\SL,1}=\SO_Y\oplus\SL\subset\Lambda_\SL.$$

Given a $\Lambda_\SL$-module $\nabla_{\Lambda_\SL}:\Lambda_\SL\otimes E\to E$, define the $\SL$-connection $(E,\nabla_\SL)$ by taking 
\begin{equation}\label{eq:L-conn}
\nabla_\SL:E\to E\otimes V^*,\ \ \ \ \ \nabla_\SL(s)(v)=\nabla_{\Lambda_\SL}((0,v)\otimes s),
\end{equation}
for $s$ and $v$ local sections of $E$ and $V$ respectively. This is an integrable, hence flat, $\SL$-connection.

Conversely, from a flat $\SL$-connection $(E,\nabla_\SL)$, define 
$$\nabla_{\Lambda_{\SL,1}}:\Lambda_{\SL,1}\otimes E\to E,\ \ \ \ \nabla_{\Lambda_{\SL,1}}((f,v)\otimes s)=fs+\nabla_\SL(s)(v),$$ where $f\in\SO_Y$, $v\in V$ and $s\in E$. By successive compositions, this defines a $(\Lambda_{\SL,1})^{\otimes \bullet}$-module structure on $E$, which descends to a $\Lambda_\SL$-module structure 
\begin{equation}\label{eq:Lambda_L-modu}
\nabla_{\Lambda_\SL}:\Lambda_\SL\otimes E\to E
\end{equation}
because $\nabla_\SL$ is integrable and satisfies Leibniz.

Now, from \cite[Proposition 5.3]{Tor12}, we have the following.

\begin{theorem}\label{thm:equivcatL-conn-<>lambda-mod}
The correspondences \eqref{eq:L-conn} and \eqref{eq:Lambda_L-modu} induce inverse equivalences of categories:
\begin{equation}\label{eq:L-connect<->lambdamod}
\left\{\begin{tabular}{c}
      \text{isomorphism classes of} \\
      \text{flat $\SL$-connections on $Y$}
    \end{tabular}\right\}\longleftrightarrow
    \left\{\begin{tabular}{c}
      \text{isomorphism classes of} \\
      \text{$\Lambda_\SL$-modules on $Y$}
      \end{tabular}\right\}.
\end{equation} 
\end{theorem}

\begin{remark}
The results of Tortella in \cite{Tor11} and \cite{Tor12} are actually more general, in the sense that there is no split condition on the almost polynomial sheaves of differential operators; cf.~Definition \ref{def:split}. In this more general version there is a certain cohomology class $Q$ which is the obstruction to the existence of a Lie algebroid splitting of the short exact sequence \eqref{eq:sesLiealg2}. So Theorem \ref{thm:equivcatL-conn-<>lambda-mod} corresponds to $Q=0$, and then $\Lambda_\SL$ is written $\Lambda_{\SL,[(0,0)]}$ in Tortella's notation. The generalized version of the correspondence \eqref{eq:L-connect<->lambdamod}, for any $Q$, still holds, but on the left-hand side one has to consider not flat (i.e.~ not integrable) connections, but rather projectively flat $\SL$-connections, that is, $\SL$-connections with constant central curvature, while on the right-hand side the sheaves $\Lambda_\SL$ are not necessarily split; see \cite[Proposition 5.3]{Tor12}.

In any case, in this paper we are only interested in the $Q=0$ case because this is the only case, under the assumption that \eqref{eq:sesLiealg2} splits as a sequence of $\SO_Y$-modules, for which the corresponding moduli spaces (to be introduced in the next section) are non-empty; see \cite[Proposition 39]{Tor11} and \cite[Corollary 5.4]{Tor12}. In other words, \cite[Corollary 5.4]{Tor12} implies that the only non-empty moduli space of semistable projectively flat $\SL$-connections on $Y$ is the one of flat $\SL$-connections.
\end{remark}

\subsection{Moduli spaces of $\SL$-connections and of $\Lambda$-modules}\label{sec:moduli-LandLambda}

Let us return now to the case where $Y$ is our fixed smooth complex projective curve $X$ and where $\SX\to S$ is an $S$-family of smooth complex projective curves over a scheme $S$. An $\SL$-connection $(E,\nabla_\SL)$ is \emph{(semi)stable} if for every subbundle $0\ne F\subsetneq E$ preserved by $\nabla_\SL$, i.e.~such that $\nabla_\SL(F)\subseteq F\otimes \Omega_\SL^1$, we have 
$$\mu(F)<\mu(E) \quad (\text{resp. } \le)$$
and that a $\Lambda$-module $(E,\nabla_\Lambda)$ on a fibre $X_s$ of $\SX\to S$ is (semi)stable if for every $0\ne F\subsetneq E$ preserved by $\nabla_\Lambda$, i.e.~such that $\nabla_\Lambda(\Lambda\otimes F)\subseteq F$ we have
$$\mu(F)<\mu(E) \quad (\text{resp. } \le).$$
Note that if the vector bundle $E$ is semistable, then any $\SL$-connection and any $\Lambda$-module are also semistable.
Clearly, an $\SL$-connection is (semi)stable if and only if the corresponding $\Lambda_\SL$-module is (semi)stable.  Define the \emph{degree} of an $\SL$-connection or of a $\Lambda$-module over $X$ to be the degree of the underlying vector bundle.

For any sheaf of rings of differential operators $\Lambda$ on $\SX$ over $S$, and any rank $r>0$ and $d\in\ZZ$, Simpson \cite{Si94} proved that there exists a moduli space $\SM_\Lambda(\SX,r,d)$ of semistable $\Lambda$-modules of rank $r$ and degree $d$ and that it is a complex quasi-projective variety over $S$. If the family $\SX$ is clear from the context, we denote the moduli space simply by $\SM_\Lambda(r,d)$. If $s\in S$ is a point, let $\Lambda_s$ be the pullback of $\Lambda$ to the fibre $X_s$ of $\SX\to S$. A closed point of $\SM_\Lambda(\SX,r,d)$ over a point $s\in S$ represents a semistable $\Lambda_s$-module $(E,\nabla_{\Lambda_s})$ over $X_s$.

Then, Theorem \ref{thm:equivcatL-conn-<>lambda-mod} also identifies $\SM_{\Lambda_\SL}(r,d)$ with the moduli space $\SM_{\SL}(r,d)$ of integrable $\SL$-connections of rank $r$ and degree $d$ \cite{Tor12}.

\begin{remark}
Alternatively, a moduli space of (possibly non-integrable) Lie algebroid connections was also constructed by Krizka \cite{Kr08} using analytic tools. When $\SL$ has rank $1$, integrability is automatic and, therefore, this construction gives an alternative description for the moduli space of integtrable Lie algebroid connections. On the other hand, in an unpublished work \cite{Kr10}, the same author works with flat connections and proposes a general analytical construction of such moduli space.
\end{remark}

We saw in Example \ref{eg-Liealgebroids} that when $\SL$ has the trivial algebroid structure $\SL=(V,0,0)$, an $\SL$-connection is precisely the same thing as a $V^*$-twisted Higgs bundle. Moreover, it is obvious that (semi)stability is preserved under this identification. Hence, the moduli space of integrable $\SL$-connections coincides with the moduli space of $V^*$-twisted Higgs bundles. Therefore, for a trivial Lie algebroid $\SL=(V,0,0)$, we have the following canonical identification
$$\SM_{\Lambda_{(V,0,0)}}(r,d)=\SM_{(V,0,0)}(r,d)=\SM^{\Dol}_{V^*}(r,d)$$
between the moduli spaces of integrable $\SL$-connections, $\Lambda_\SL$-modules and $V^*$-twisted Higgs bundles over $X$, all of the same rank $r$ and degree $d$.

Note that, depending on the choice of the rank, degree and $\Lambda$, the moduli space $\SM_\Lambda(\SX,r,d)$ may be singular or even empty. For instance, if $\Lambda$ is the sheaf of differential operators $\SD_X$ on a curve $X$, then $\SM_{\SD_X}(r,d)$ coincides with the moduli space of semistable ($\ST_X$-)connections on $X$ whose underlying vector bundle has rank $r$ and degree $d$. But the flatness condition (which is automatic since $\ST_X$ is a line bundle) implies that $d=0$, so the moduli space is empty for $d\neq 0$. Moreover, for $d=0$ and $r\geq 2$, the moduli space of flat connections is singular due to the existence of strictly semistable objects.
In the remaining part of this section we will prove sufficient conditions for these moduli spaces to be non-empty and smooth varieties.

The existence of algebraic connections (i.e.~$\ST_X$-connections in the language of Lie algebroids) on an algebraic vector bundle $E$ over $X$ has been studied by Atiyah in \cite{A57}. There it was proved the existence of a class in $\Ext^1(T_X,\End(E))$ --- now known as the Atiyah class of $E$ ---  whose vanishing is equivalent to the existence of such an algebraic connection.
A generalization of this picture to $\SL$-connections was carried out in \cite[Section 2.4.4]{Tor11} \cite[\S 4]{Tor12}. Keep considering the bundle $E$ and let $\SL=(V,[\cdot\,,\cdot],\delta)$ be a Lie algebroid. Let $\fa(E)\in \Ext^1(T_X,\End(E))$ be the Atiyah class of $E$. Define the \emph{$\SL$-Atiyah class of $E$} as
$$\fa_\SL(E)=\delta^*(\fa(E)) \in \Ext^1(V,\End(E)).$$

\begin{proposition}{\cite[Proposition 17]{Tor11}}
\label{prop:existenceConnection}
An algebraic vector bundle $E$ admits an algebraic $\SL$-connection if and only if $\fa_\SL(E)=0$.
\end{proposition}

From this we can exhibit some concrete examples of $\SL$-connections.

\begin{corollary}
\label{cor:existsConnection}
Let $E$ be a semistable vector bundle on the curve $X$. Let $\SL=(V,[\cdot\,,\cdot],\delta)$ be a Lie algebroid such that the vector bundle $V$ is semistable and $-\mu(V)>2g-2$. Then $E$ admits an $\SL$-connection. Moreover, if $\rk(\SL)=1$, then $E$ admits an integrable $\SL$-connection.
\end{corollary}
\begin{proof}
The $\SL$-Atiyah class $\fa(E)$ is an element of
$$\Ext^1(V,\End(E)) \cong H^1(\End(E)\otimes V^*) \cong H^0( \End(E)\otimes K_X\otimes V)^*,$$
by Serre duality.
Since both $E$ and $V$ are semistable then $\End(E)\otimes K_X \otimes V$ is also semistable. Moreover
$$\mu(\End(E)\otimes K_X \otimes V) = 2g-2+\mu(V)<0,$$
so $H^0( \End(E)\otimes K_X\otimes V)=0$. Thus $\fa(E)=0$, and the result follows from Proposition \ref{prop:existenceConnection}. 

If in addition $\rk(\SL)=1$, then $\Omega_\SL^2=\Lambda^2V^*=0$ thus any $\SL$-connection on $E$ is automatically flat.
\end{proof}

A Lie algebroid $\SL$ is called \emph{transitive} if the anchor $\delta$ is surjective and \emph{intransitive} otherwise.

\begin{corollary}
\label{cor:existsConnection-intransitive}
Let $\SL$ be any intransitive Lie algebroid on a curve $X$ and let $E$ be a semistable algebraic vector bundle. Then $E$ admits an integrable $\SL$-connection.
\end{corollary}
\begin{proof}
As $\SL=(V,[\cdot\,,\cdot],\delta)$ is intransitive, the anchor map $\delta:V\to T_X$ is not surjective. Then, as $T_X$ is a line bundle, there exists a point $x\in X$ such that $\delta|_x=0$. Thus, $\delta$ factors through $\bar\delta:V\to T_X(-x)\subset T_X$. Actually, the line bundle $T_X(-x)$ inherits a natural Lie algebroid structure, denoted by $\ST(-x)$, from the one of $\ST=(T_X,[\cdot\,,\cdot]_{\op{Lie}},\id)$, since the Lie bracket of two local vector fields which annihilate at $x$ also annihilates at $x$. Thus, the anchor in $T_X(-x)$ is just the inclusion $T_X(-x)\hookrightarrow T_X$. Since we also know the anchor maps are also Lie algebroid maps, we have a commutative diagram of Lie algebroid maps
$$\begin{tikzcd}
&\SL\ar{r}{\delta}\ar{rd}[swap]{\bar\delta}&\ST\\
    &&\ST(-x).\ar[u,hook]
    \end{tikzcd}$$

Now, $T_X(-x)$ is a line bundle such that $-\deg(T_X(-x))=2g-1>2g-2$, so the previous corollary shows that $E$ admits an integrable $\ST(-x)$-connection. Moreover, by \eqref{eq:intL-conn-repsinDE} such connection is determined by a representation $\ST(-x)\to\SD(E)$. Pre-composing it with $\bar\delta$ yields a representation $\SL\to\SD(E)$ and, hence, again by \eqref{eq:intL-conn-repsinDE}, gives rise to an integrable $\SL$-connection on $E$. 
\end{proof}

The previous results can be now immediately used to prove, under certain conditions, non-emptiness of the moduli spaces of $\SL$-connections of any rank and degree.

\begin{proposition}
\label{prop:moduliNonEmpty}
For any rank $r$ and degree $d$ and any Lie algebroid $\SL$ such that either
\begin{enumerate}
\item $\rk(\SL)=1$ and $\deg(\SL)<2-2g$, or
\item $\SL$ is intransitive,
\end{enumerate}
the moduli space $\SM_{\SL}(r,d)$ is non-empty.
\end{proposition}

\begin{proof}
Let $\SL$ be any Lie algebroid verifying either of the two given conditions. Choose any semistable vector bundle $E$ over $X$. By Corollary \ref{cor:existsConnection}, if $\rk(\SL)=1$ and $\deg(\SL)<2-2g$, or by Corollary \ref{cor:existsConnection-intransitive}, if $\SL$ is intransitive, we conclude that $E$ admits an integrable $\SL$-connection $\nabla_\SL$. Moreover, since $E$ is semistable, then so is $(E,\nabla_\SL)$, which therefore represents a point in $\SM_{\SL}(r,d)$.
\end{proof}

\subsection{The $\SL$-Hodge moduli spaces}
\label{sec:L-Hodge moduli}

Let us now introduce the deformation of a split almost polynomial sheaf of rings of differential operators (cf.~Definition \ref{def:SRDO}) to its associated graded ring, as described in \cite[p. 86]{Si94}.

Let $\SL=(V,[\cdot\,,\cdot],\delta)$ be a Lie algebroid over $X$ and let $\Lambda_\SL$ be its associated split almost polynomial sheaf of rings of differential operators. By construction, we know that
$$\op{Gr}_\bullet(\Lambda_\SL) \cong \Sym^\bullet(V).$$
We can associate to $\Lambda_\SL$ a sheaf of rings of differential operators $\Lambda_\SL^{\op{red}}$ on $X\times \CC$ over $\CC$ whose fibre over $1$ is $\Lambda_\SL$ and whose fibre over $0$ is isomorphic to its graded algebra $\Sym^\bullet(V)$. Let $\lambda$ be the coordinate of $\CC$ and let $p_X: X\times \CC\to X$ be the projection. We define $\Lambda_\SL^{\op{red}}$ as the subsheaf of $p_X^*(\Lambda_\SL)$ generated by sections of the form $\sum_{i\ge 0} \lambda^i u_i$, where $u_i$ is a local section of $\Lambda_{\SL,i}\subseteq\Lambda_{\SL}$. This subsheaf is a sheaf of filtered algebras on $X\times \CC$ over $\CC$, which coincides with the Rees algebra construction of the sheaf of filtered algebras $\Lambda_\SL$ and one can verify that it satisfies all properties of Definition \ref{def:SRDO}, making it a sheaf of rings of differential operators on $X\times \CC$ over $\CC$.

On the other hand, given the Lie algebroid $\SL=(V,[\cdot\,,\cdot],\delta)$ and $\lambda\in \CC$, we can define another Lie algebroid $\SL_\lambda$ as 
\begin{equation}\label{eq:Llambda}
\SL_\lambda=(V,\lambda[\cdot\,,\cdot],\lambda\delta),
\end{equation}
Then, the following can be checked through direct computation
\begin{equation}\label{eq:Lambdared|lambda=Slambda}
\Lambda_\SL^{\op{red}}|_{X\times \{\lambda\}} \cong \Lambda_{\SL_\lambda}.
\end{equation}
Observe that if $\lambda\ne 0$ then multiplicacion by $\lambda$ defines an isomorphism of Lie algebroids
$$\SL_\lambda \stackrel{\sim}{\longrightarrow} \SL, \ \ \ \ \ v \mapsto \lambda v$$
and for $\lambda=0$, we have that $\SL_0=(V,0,0)$ is the trivial algebroid over $V$. This implies the following properties of the fibres of $\Lambda_\SL^{\op{red}}$ over $\lambda\in \CC$ which are also well known consequences of the Rees construction.
\begin{enumerate}
\item For every $\lambda\ne 0$ we have $\Lambda_\SL^{\op{red}}|_{X\times \{\lambda\}}\cong \Lambda_\SL^{\op{red}}|_{X\times \{1\}}\cong  \Lambda_\SL$.
\item If $\lambda=0$ then $\Lambda_\SL^{\op{red}}|_{X\times \{0\}}\cong \op{Gr}^\bullet(\Lambda_\SL) \cong \Sym^\bullet(V)$.
\end{enumerate}

Now, consider the moduli space $\SM_{\Lambda_\SL^{\op{red}}}(r,d)$ of $\Lambda_\SL^{\op{red}}$-modules of rank $r$ and degree $d$. By Simpson's construction, it is a quasiprojective variety with a map
\begin{equation}\label{eq:project->lambda}
\pi:\SM_{\Lambda_\SL^{\op{red}}}(r,d) \longrightarrow \CC
\end{equation}
and, by \eqref{eq:Lambdared|lambda=Slambda}, it parametrises $S$-equivalence classes of triples $(E,\nabla_{\SL_\lambda},\lambda)$ where $\lambda =\pi(E,\nabla_{\SL_\lambda},\lambda)\in\CC$ and $(E,\nabla_{\SL_\lambda})$ is a semistable integrable $\SL_\lambda$-connection of rank $r$ and degree $d$ on $X$.
By properties (1) and (2) above, it becomes clear that, for all $\lambda \in \CC^*$,
\begin{equation}\label{eq:isomfnonzeroibres}
\pi^{-1}(\lambda)=\SM_{\Lambda_{\SL_\lambda}}(r,d) \cong \SM_{\Lambda_\SL}(r,d)=\pi^{-1}(1)\ \ \text{  and  }
\ \ \pi^{-1}(0)\cong \SM_{V^*}^{\Dol}(r,d).
\end{equation}
so $\SM_{\Lambda_\SL^{\op{red}}}(r,d)$ is a variety which ``interpolates'' between the moduli space of $\SL$-connections and the moduli space of $V^*$-twisted Higgs bundles.

Let us explore more precisely what is an $\SL_\lambda$-connection, for $\SL_\lambda$ as in \eqref{eq:Llambda}. By \eqref{eq:d_L(f)} and \eqref{eq:d_L(omega)} we see that if $(\Omega_{\SL}^\bullet,d_\SL)$ is the Chevalley--Eilenberg--de Rham complex of $\SL$, then $\Omega_{\SL_\lambda}^\bullet=\Omega_{\SL}^\bullet$ and $(\Omega_{\SL_\lambda}^\bullet,\lambda d_\SL)$ is the Chevalley--Eilenberg--de Rham complex of $\SL_\lambda$. Hence, in an $\SL_\lambda$-connection $(E,\nabla_{\SL_\lambda})$, the map $\nabla_{\SL_\lambda}:E\to E\otimes V^*$ verifies 
\begin{equation}\label{eq:lambda-connect}
\nabla_{\SL_\lambda}(fs)=f\nabla_{\SL_\lambda}(s) + \lambda s\otimes d_\SL(f).
\end{equation}

Motivated by the classical example of a \emph{$\lambda$-connection} on an algebraic vector bundle $E$ (see the example below), we also call an $\SL_\lambda$-connection to be a $(\lambda,\SL)$-connection, and consider it as a triple $(E,\nabla_\SL,\lambda)$, where $\nabla_\SL:E\to E\otimes V^*$ verifies \eqref{eq:lambda-connect}. Then we identify $\SM_{\Lambda_\SL^{\op{red}}}(r,d)$ with the moduli space of $(\lambda,\SL)$-connections and, in analogy with the terminology for the moduli space of $\lambda$-connections, we call it the \emph{$\SL$-Hodge moduli space}. Furthermore, in order to simplify the notation, we will denote it by $\SM_\SL^{\Hod}(r,d)$.

\begin{example}\label{ex:lambda,L-conn}\mbox{}
\begin{enumerate}
\item A $(\lambda,\ST_X)$-connection is a \emph{$\lambda$-connection} on an algebraic vector bundle $E$.
\item An $\SL$-connection is a the same thing as a $(1,\SL)$-connection.
\item For $\SL=(V,[\cdot\,,\cdot],\delta)$, a $(0,\SL)$-connection is a $V^*$-twisted Higgs bundle.
\end{enumerate}
\end{example}

The moduli space $\SM_{\SL}^{\Hod}(r,d)$ is endowed with a $\CC^*$-action scaling the $(\lambda,\SL)$-connection,
\begin{equation}\label{eq:C*-action-LHodgemod}
t\cdot (E,\nabla_\SL,\lambda)=(E,t\nabla_\SL,t\lambda),\ t\in\CC.
\end{equation}
Indeed, since $(E,\nabla_\SL)$ is flat and semistable, then so is $(E,t\nabla_\SL)$ for every $t$.
Furthermore, $(E,t\nabla_{\SL_\lambda})$ is a $(t\lambda,\SL)$-connection. Thus the map $\pi$ is $\CC^*$-equivariant with respect to this $\CC^*$-action and the standard one on $\CC$.

Note that for $\lambda=0$, the action \eqref{eq:C*-action-LHodgemod} restricts to the usual $\CC^*$-action on the Higgs bundle moduli space by scaling the Higgs field (cf.~\eqref{eq:C*-actHiggs}), and which will play an important role in section \ref{section:HiggsPolynomials}.

Finally, observe that clearly $(E,t\nabla_\SL,t)$ is semistable if and only if $(E,\nabla_\SL,1)$ is, so the $\CC^*$-action \eqref{eq:C*-action-LHodgemod} induces an isomorphism
$$\pi^{-1}(\CC^*) \cong \pi^{-1}(1)\times \CC^*= \SM_{\SL}(r,d)\times \CC^*.$$

\subsection{Classification of Lie algebroid structures on line bundles}
\label{subsection:classification}
The main objects of interest for the remaining part of the paper will be moduli spaces of $\SL$-connections when $\SL$ is a rank 1 Lie algebroid over a smooth complex projective curve. Before continuing with the main part of our analysis, we dedicate this section to study the types of Lie algebroid structures that actually exist over line bundles on algebraic curves. Besides its intrinsic interest, this will help to clarify the classes of moduli spaces being considered in the main theorems of this, and also to provide a somehow more explicit description of how such spaces fit together in families.

So the following theorem provides an explicit classification of all possible Lie algebroid structures on a line bundle $L$.

\begin{theorem}
\label{thm:classificationLieAlgebroidsRank1}
Let $X$ be a smooth complex projective curve and let $L$ be a line bundle on $X$. Then there is a one-to-one correspondence between isomorphism classes of Lie algebroid structures $(L,[\cdot\,,\cdot],\delta)$ on $L$ and orbits of the $\CC^*$-action on $H^0(L^{-1}\otimes T_X)$ sending
$$\lambda \cdot \delta = \delta \quad \forall \lambda \in \CC^* \,\, \forall \delta\in H^0(L^{-1}\otimes T_X).$$
\end{theorem}

\begin{proof}
We will prove that, given a line bundle $L$, for every choice of an anchor $\delta\in H^0(\Hom(L,T_X))=H^0(L^{-1}\otimes T_X)$, there exists a unique Lie bracket $[\cdot\,,\cdot]_\delta:L\otimes_\CC L \to L$ such that $(L,[\cdot,\cdot]_\delta,\delta)$ is a Lie algebroid.

Let us start by proving uniqueness. Assume that $[\cdot,\cdot]_1$ and $[\cdot\,,\cdot]_2$ are Lie brackets on $L$ compatible with the same anchor $\delta$. Then, they both satisfy Leibniz rule, so for each local sections $u,v\in L$ and $f\in \SO_X$,
$$[u,fv]_1=f[u,v]_1+\delta(u)(f)v,$$
$$[u,fv]_2=f[u,v]_2+\delta(u)(f)v.$$
Subtracting both equations yields
$$[u,fv]_1-[u,fv]_2=f([u,v]_1-[u,v]_2)$$
so $g(u,v)=[u,v]_1-[u,v]_2$ defines a skew-symmetric $\SO_X$-linear map $g:L\wedge L \to L$. But $L$ is a line bundle, so $L\wedge L=0$ and $g\equiv 0$ and thus $[\cdot\,,\cdot]_1=[\cdot\,,\cdot]_2$.

Let us prove the existence of the Lie bracket. Let $\delta:L\to T_X$ be an anchor map. Let $U\subset X$ be an open subset such that $L|_U$ is trivial. Let $u\in H^0(U,L)$ be a nowhere-vanishing local section over $U$ trivializing $L$. Then for every couple of local sections $v,w\in H^0(U,L)$ there exist unique local functions $a,b\in H^0(U,\SO_X)$ such that $v=au$ and $w=bu$. Define
$$[v,w]_\delta=[au,bu]_\delta:=(a\delta(u)(b)-b\delta(u)(a))u.$$
First of all, observe that the previous definition does not depend on the choice of the trivializing section $u$. If $u'\in H^0(U,L)$ is another nowhere-vanishing local section then there exists a nowhere-vanishing local function $s\in H^0(U,\SO_X)$ such that $u=su'$. Thus, $v=asu'$ and $w=bsu'$ and, taking into account that $\delta$ is $\SO_X$-linear and that $\delta(u)$ is a derivation, we have
\begin{multline*}
[asu',bsu']_\delta=(as\delta(u')(bs)-bs\delta(u')(as))u' =(as^2\delta(u')(b)-bs^2\delta(u')(a))u'\\
 = (a\delta(su')(b)-b\delta(su')(a))su' = (a\delta(u)(b)-b\delta(u)(a))u = [au,bu]_\delta.
\end{multline*}
The independence from the trivialization and the fact that $\delta(u)$ is $\CC$-linear for each local section $u$, implies that $[\cdot\,,\cdot]_\delta:L\otimes_\CC L \to L$ is a well defined map of sheaves. Let us verify that it satisfies Jacobi and Leibniz. It is clearly enough to verify each identify for local sections over trivializing open subsets. Thus, it is enough to consider sections of the form $au,bu,cu\in H^0(U,L)$ where $a,b,c\in H^0(U,\SO_X)$ and $u\in H^0(U,L)$ is a nowhere-vanishing local section over a certain open subset $U\subset X$. Then we have
\begin{multline*}
\sum_{cyc}[au,[bu,cu]_\delta]_\delta=\sum_{cyc} \left(a\delta(u)(b\delta(u)(c)-c\delta(u)(a))-\delta(u)(a)(b\delta(u)(c)-c\delta(u)(b))\right)u\\
=u\underbrace{\left(\sum_{cyc}ab \delta(u)(\delta(u)(c)) - \sum_{cyc}ac\delta(u)(\delta(u)(b))\right)}_0
-u\underbrace{\left(\sum_{cyc}\delta(u)(a) b \delta(u)(c) - \sum_{cyc}\delta(u)(a)\delta(u)(b)c\right)}_0=0,
\end{multline*}
where $\sum_{cyc}$ indicates a cyclic summation (over $a,b,c$), so $[\cdot\,,\cdot]_\delta$ satisfies the Jacobi identity. On the other hand, if $a,b,f\in H^0(U,\SO_X)$, then
\begin{multline*}
[au,fbu]_\delta = (a\delta(u)(fb)-fb\delta(u)(a))u = (af\delta(u)(b)+ab\delta(u)(f)-fb\delta(u)(a))u\\
=f(a\delta(u)(b)-b\delta(u)(a))u + \delta(au)(f)bu = f[au,bu]_\delta +\delta(au)(f)bu,
\end{multline*}
so the bracket satisfies Leibniz rule.

Finally, let us consider when can two different $\delta,\delta'\in H^0(L^{-1}\otimes T_X)$ give rise two isomorphic Lie algebroid structures $(L,[\cdot\,,\cdot]_\delta,\delta)\cong (L,[\cdot\,,\cdot]_{\delta'},\delta')$. As $L$ is a line bundle, the automorphisms of $L$ are given by the scaling action of $\CC^*$ on its fibres. Thus, the two Lie algebroid structures can only be isomorphic if they are in the same $\CC^*$-orbit, i.e., if
$$(L,[\cdot\,,\cdot]_{\delta'},\delta') = (L,\lambda [\cdot\,,\cdot]_\delta,\lambda\delta)$$
for some $\lambda \in \CC^*$. By construction, it is clear that
$$[\cdot\,,\cdot]_{\lambda \delta} = \lambda [\cdot\,,\cdot]_\delta$$
so the $\CC^*$-action on $H^0(L^{-1}\otimes T_X)$ sending $\lambda \cdot \delta = \lambda \delta$ actually lifts to isomorphisms
$$(L,[\cdot\,,\cdot]_{\lambda \delta},\lambda \delta)=(L,\lambda[\cdot\,,\cdot]_{\delta},\lambda \delta)\cong (L,[\cdot\,,\cdot]_\delta,\delta)$$
and, therefore, the isomorphism classes of Lie algebroid structures on $L$ are in one-to-one correspondence with the orbits of such $\CC^*$-action on $H^0(L^{-1}\otimes T_X)$.
\end{proof}

\begin{corollary}
\label{cor:classificationLieAlgebroidsRank1}
Let $L$ be a line bundle on a genus $g$ smooth complex projective curve.
\begin{itemize}
\item If $\deg(L)>2-2g$, or $\deg(L)=2-2g$ but $L\ne T_X$, then the only Lie algebroid structure admitted by $L$ is the trivial one $(L,0,0)$.
\item $T_X$ only admits two nonisomorphic Lie algebroid structures: the trivial one $(T_X,0,0)$ and the canonical one given by the Lie bracket of vector fields $\ST_X=(T_X,[\cdot\,,\cdot]_{\op{Lie}},\id)$.
\end{itemize}
\end{corollary}

\begin{proof}
If $\deg(L)>2-2g=\deg(T_X)$, then $H^0(L^{-1}\otimes T_X)=0$, so $L$ only admits the trivial structure.

The same holds if $\deg(L)=\deg(T_X)$ but $L\ne T_X$, as then $L^{-1}\otimes T_X$ is a nontrivial degree $0$ line bundle which cannot have nontrivial sections.

If $L=T_X$, then the space of possible isomorphism classes of Lie algebroid structures on $T_X$ is parameterised by the orbits of the $\CC^*$-action on $H^0(\SO_X)\cong \CC$. There are only two such orbits, one corresponding to $\delta=0$ and another corresponding to $\delta=\id$. The Lie algebroid corresponding to $\delta=0$ is the trivial structure. The Lie algebroid corresponding to $\delta=\id$ corresponds to the canonical Lie bracket.
\end{proof}

As a consequence of this corollary, we observe that if $\deg(L)\ge 2-2g$ then the only possible moduli spaces of Lie algebroid connections that can be considered are either a moduli space of $L^{-1}$-twisted Higgs bundles or, if $L=T_X$, the moduli space of flat connections. Moreover, the latter is only non-empty if the degree is $0$. This justifies focusing our analysis on moduli spaces of $\SL$-connections when $\deg(L)<2-2g$, which is the only case where the space of possible Lie algebroid structures actually contain other different moduli spaces. Given $\delta\in H^0(L^{-1}\otimes T_X)$, the moduli space of $(L,[\cdot\,,\cdot]_\delta,\delta)$-connections can be given a very concrete description. It parameterises vector bundles $E$ on $X$ together with maps
$$\nabla:E\longrightarrow E\otimes L^{-1}$$
such that for each local section $s\in E$ and each local function $f\in \SO_X$ we have
$$\nabla(fs)=f\nabla(s)+s\otimes \delta^t(df).$$
which can be interpreted uniformly as moduli spaces of generalized $L^{-1}$-valued connections, where the opperator $d_\SL=\delta^t\circ d : \SO_X \longrightarrow L^{-1}$ is the first order Chevalley-Eilenberg-de Rham differential of $(L,[\cdot\,,\cdot]_\delta,\delta)$. However, there is an alternative interesting interpretation for these spaces. The following corollary shows that, if $\delta\ne 0$, then each of these moduli spaces is always isomorphic to a suitable moduli space of logarithmic or meromorphic connections.

Recall the canonical Lie algebroid $\ST_X=(T_X,[\cdot\,,\cdot]_{\op{Lie}},\id)$ on our smooth projective curve $X$. Consider an effective divisor $D=\sum_{i=1}^n k_i x_i$ on $X$, with $k_i\geq 1$. Let $\SM_{\op{conn}}(D,r,d)$ be the moduli space of rank $r$ and degree $d$ semistable \emph{singular} ($\ST_X$-)connections, with poles of order at most $k_i$ over each $x_i\in D$.  These connections are possibly irregular if $k_i>1$ for some $i$.

Take the Lie subalgebroid $\ST_X(-D)\subset\ST_X$, thus with underlying bundle $T_X(-D)\subset T_X$, the induced Lie bracket of vector fields, and the inclusion anchor map $i_D: T_X(-D)\hookrightarrow T_X$. Then $\SM_{\ST_X(-D)}(r,d)=\SM_{\op{conn}}(D,r,d)$.

Observe that, contrary to the moduli space of semistable integrable $\ST_X$-connections, which is only non-empty if $d=0$, the moduli space $\SM_{\ST_X(-D)}(r,d)$ of semistable integrable $\ST_X(-D)$-connections is always non-empty for any rank $r$ and degree $d$, as follows from Proposition \ref{prop:moduliNonEmpty}. So the same holds for $\SM_{\op{conn}}(D,r,d)$.

\begin{corollary}
\label{cor:classificationLieAlgebroidsRank1Meromorphic}
Let $\SL=(L,[\cdot,\cdot]_{\delta},\delta)$ be a rank 1 Lie algebroid with $\deg(L)<2-2g$ and $\delta\ne 0$. Then there exists a unique effective divisor $D$ in the linear system $|L^{-1}\otimes T_X|$ such that
$$\SM_{\SL}(r,d)\cong \SM_{\op{conn}}(D,r,d)$$
\end{corollary}

\begin{proof}
By Theorem \ref{thm:classificationLieAlgebroidsRank1}, the isomorphism classes of rank 1 Lie algebroid structures on $L$ are in bijective correspondence with the $\CC^*$-orbits of their anchors $\delta\in H^0(L^{-1}\otimes T_X)$. There is a one on one correspondence between the nonzero classes $[\delta]\in \mathbb{P}H^0(L^{-1}\otimes T_x)$ and effective divisors in the linear system $|L^{-1}\otimes T_X|$ associated to the line bundle $L^{-1}\otimes T_X$. Each nonzero class $[\delta]\in \mathbb{P}H^0(L^{-1}\otimes T_X)$ is associated to the unique effective divisor $D\in |L^{-1}\otimes T_X|$ representing the zeroes of the map $\delta:L\longrightarrow T_X$. Thus, the map $\delta$ induces an isomorphism $\bar{\delta}:L \stackrel{\cong}{\longrightarrow} T_X(-D)$ and the following diagram commutes
\begin{eqnarray*}
\xymatrix{
L \ar[r]^-{\bar{\delta}}_-{\cong} \ar[rd]_{\delta} & T_X(-D) \ar@{^(->}[d]^{i_D} \\
& T_X 
}\, .
\end{eqnarray*}
Dualizing and calling $\varphi:=(\bar{\delta}^t)^{-1}: L\stackrel{\cong}{\longrightarrow} K(D)$ we get
\begin{eqnarray*}
\xymatrix{
L^{-1} \ar[r]^-{\varphi}_-{\cong} & K_X(D) \\
& K_X \ar[ul]^{\delta^t}  \ar@{^(->}[u]_{i_D^t}
}\, .
\end{eqnarray*}
Then, $\varphi$ induces an isomorphism $\SM_{\SL}(r,d)\cong \SM_{\op{conn}}(D,r,d)$ as follows. Send an $\SL$-connection $(E,\nabla)$ with
$$\nabla: E\longrightarrow E\otimes L^{-1}$$
to $(E, (\id\otimes \varphi)\circ \nabla)$. We know that $\nabla$ is an $\SL$-connection if and only if for every local section $s$ of $E$ and every local regular function $f$ we have
$$\nabla(fs)=f\nabla(s)+s\otimes \delta^t(df)\, .$$
Then the map
$$(\id\otimes \varphi)\circ \nabla : E\longrightarrow E\otimes K(D)\, .$$
is a meromorphic connection with poles precisely over $D$ because we have
$$\left((\id \otimes \varphi)\circ \nabla\right) (fs) =  f \left((\id \otimes \varphi)\circ \nabla\right)(s) + s \otimes (\varphi \circ \delta^t)(df) = f (\id \otimes \varphi)(s) + s \otimes i_D^t(df)\, .$$
As the considered Lie algebroids have all rank 1, both the $\SL$-connection $\nabla$ and the resulting meromorphic connection $(\id\otimes \varphi)\circ \nabla$ are immediately integrable. Moreover, composing with $\id\otimes \varphi^{-1}$ clearly gives an inverse to this construction and both transformations work in families, so $\varphi$ induces the described isomorphism between the moduli spaces.
\end{proof}

As a consequence of the previous correspondence, we can then use this Lie algebroid connections formalism to treat homogeneously diverse moduli spaces of twisted Higgs bundles, logarithmic connections and meromorphic connections with poles over different divisors. This will allow us to find nontrivial relations between its geometries.

Notice now that the quotient space $H^0(L^{-1}\otimes T_X)/\CC^*$ is not Hausdorff and that the point $0$ (representing the trivial Lie algebroid and, thus, the moduli space of twisted Higgs bundles) lies in the closure of all the $\CC^*$ orbits in $H^0(L^{-1}\otimes T_X)$. In fact, it is the unique GIT-unstable point for the $\CC^*$-action. In the following sections, we will use this fact to deduce several properties of the moduli spaces of Lie algebroid connections by relating them to the moduli space of twisted Higgs bundles. The idea will be to use the $\SL$-Hodge moduli spaces described in Section \ref{sec:L-Hodge moduli} to provide algebraic families of moduli spaces of Lie algebroid connections over $\CC$ endowed with a $\CC^*$-action which lifts the $\CC^*$-action on $H^0(L^{-1}\otimes T_X)$. In particular, given some $\delta \in H^0(L^{-1}\otimes T_X)$, the corresponding Hodge moduli space is a family over $\CC$ whose fibre over $\lambda \in \CC$ is isomorphic to the moduli space of $(L,[\cdot\,,\cdot]_{\lambda \delta}, \lambda \delta)$-connections. The generic fibre in this family is isomorphic to the moduli space of $(L,[\cdot\,,\cdot]_\delta,\delta)$-connections (which we know is isomorphic to a certain moduli space of logarithmic or meromorphic connections for an appropriate divisor $D$), but the fibre over $0$ will always be isomorphic to the same fixed moduli space of $L^{-1}$-twisted Higgs bundles, independently on the choice of the initial structure $\delta$ (i.e.~on the choice of the divisor $D$ in the linear system $|L^{-1}\otimes T_X|$). By studying and analyzing the geometry of these ``gluing families'', we will be able to transfer the analysis of several topological (dimension, smoothness, irreducibility, homotopy groups) and algebraic (motives, Hodge structure, $E$-polynomials) properties of all moduli spaces of Lie algebroid connections to more manageable moduli spaces of twisted Higgs bundles.

\section{Smoothness and semiprojectivity of $\SL$-Hodge moduli spaces}\label{sec:smoothn-semiproj-L-Hodge}

Let $Y$ be a quasi-projective variety endowed with an algebraic $\CC^*$-action, denoted as $Y\to Y$, $x\mapsto t\cdot x$, $x\in Y$, $t\in\CC^*$.
\begin{definition}{\cite[Definition 1.1]{HRV15}}\label{def:semiproj}
The variety $Y$ is \emph{semiprojective} if the following conditions are satisfied:
\begin{enumerate}
\item for each $x\in Y$ the limit $\lim_{t\to 0} t\cdot x$ exists in $Y$;
\item the fixed-point locus of the $\CC^*$ action $Y^{\CC^*}$ is proper.
\end{enumerate}
\end{definition}

In this section we will prove that, under some conditions, the $\SL$-Hodge moduli space $\SM_{\SL}^{\Hod}(r,d)\cong \SM_{\Lambda_\SL^{\op{red}}}(r,d)$ is a smooth and semiprojective variety. This is a crucial result to prove the results on the geometrical and topological properties of the moduli spaces of $\SL$-connections, which will be obtained in the subsequent sections. The fact that the Higgs bundle moduli spaces are semiprojective is well-known  but, for completeness, we include the proof in the following section. This fact is used in the proof that $\SM_{\SL}^{\Hod}(r,d)$ is semiprojective.

\subsection{Semiprojectivity of the moduli space of Higgs bundles}

Let $L$ be a line bundle over the genus $g$ curve $X$. In this section we show the well-known fact that the moduli space $\SM^{\Dol}_L(r,d)$ of $L$-twisted Higgs bundles over  $X$ is a smooth semiprojective variety, under the usual conditions  on the degree of $L$ and on $r$ and $d$.
In the next section, we will prove the analogous result for the $\SL$-Hodge moduli space and that will need a more substantial amount of work. 

The moduli $\SM^{\Dol}_L(r,d)$ admits a natural $\CC^*$-action by scaling the Higgs field
\begin{equation}\label{eq:C*-actHiggs}
t\cdot (E,\varphi) =(E,t\varphi).
\end{equation}
Note that this is a particular case of \eqref{eq:C*-action-LHodgemod}.

Recall now the Hitchin map from \eqref{eq:Hitchinmap}. Then the Hitchin base $W=\bigoplus_{i=1}^r H^0(L^i)$ also admits a natural $\CC^*$-action given by
$$t\cdot (s_1,\ldots,s_r) = (t s_1, t^2s_2,\ldots, t^rs_r),$$
which makes the Hitchin map $H:\SM^{\Dol}_L(r,d)\to W$ a $\CC^*$-equivariant map.

Let us first prove that the $\CC^*$-action on $\SM^{\Dol}_L(r,d)$ verifies the first condition on Definition \ref{def:semiproj}.

\begin{lemma}
\label{lemma:limitExistsHiggs}
Let $(E,\varphi)$ be a semistable Higgs bundle on $X$. Then the limit $\lim_{t \to 0} (E,t \varphi)$ exists in $\SM^{\Dol}_L(r,d)$.
\end{lemma}
\begin{proof}
As $H:\SM^{\Dol}_L(r,d)\to W$ is $\CC^*$-equivariant, we have
$$\lim_{t\to 0} H(E,t\varphi)=\lim_{t\to 0} t\cdot H(E,\varphi)=0,$$
thus the map $\CC^*\to W$ given by $t\mapsto H(E,t\varphi)$ extends to $\CC\to W$. By Lemma \ref{lemma:HitchinProper}, $H$ is proper, so by the valuative criterion of properness the map $\CC^*\to \SM^{\Dol}_L(r,d)$ given by $t\mapsto (E,t\varphi)$ must also extend to a map $\CC\to \SM^{\Dol}_L(r,d)$, thus providing the desired limit.
\end{proof}

Now we consider the second condition on Definition \ref{def:semiproj}.

\begin{lemma}
\label{lemma:fixedProperHiggs}
The fixed-point set of the $\CC^*$-action on $\SM^{\Dol}_L(r,d)$ is a proper scheme contained in $H^{-1}(0)$.
\end{lemma}
\begin{proof}
Since the Hitchin map $H$ is $\CC^*$-equivariant, the fixed-point set $\SM^{\Dol}_L(r,d)^{\CC^*}$ must be a closed subset of $H^{-1}(W^{\CC^*})=H^{-1}(0)$. By Lemma \ref{lemma:HitchinProper}, $H$ is proper, so $H^{-1}(0)$ is proper and hence so is  
$\SM^{\Dol}_L(r,d)^{\CC^*}$.
\end{proof}

\begin{proposition}
\label{prop:semiprojectiveHiggs}
Suppose that $g\ge 2$ and that $r$ and $d$ are coprime. Suppose the line bundle $L$ is such that $\deg(L)>2g-2$. Then the moduli space $\SM^{\Dol}_L(r,d)$ is a smooth complex semiprojective variety.
\end{proposition}
\begin{proof}
By Lemma \ref{lemma:smoothModuliHiggs} the moduli space $\SM^{\Dol}_L(r,d)$ is a smooth complex variety. Then Lemmas \ref{lemma:limitExistsHiggs} and \ref{lemma:fixedProperHiggs} prove that the action $t\cdot (E,\varphi)=(E,t\varphi)$ satisfies the semiprojectivity conditions.
\end{proof}

\subsection{Smoothness and semiprojectivity of the $\SL$-Hodge moduli space}
\label{section:semiprojectivityHodge}

Our task now is to prove that the $\SL$-Hodge moduli space $\SM_\SL^{\Hod}(r,d)\cong\SM_{\Lambda_\SL^{\op{red}}}(r,d)$ is also a smooth semiprojective variety, whenever $\rk(\SL)=1$, $\deg(\SL)<2-2g$ and $r$ and $d$ are coprime. Most of the work will be dedicated to the proof of smoothness at points in $\pi^{-1}(0)$ (cf.~\eqref{eq:isomfnonzeroibres}), namely to see that the zero fibre smoothly glues with the non-zeros fibres. 

We will need to explicitly use both interpretations, provided by Theorem \ref{thm:equivcatL-conn-<>lambda-mod}, of the points parameterised by $\SM_{\Lambda_\SL^{\op{red}}}(r,d)\cong \SM_\SL^{\Hod}(r,d)$, namely semistable $\Lambda_\SL^{\op{red}}$-modules and semistable $(\lambda,\SL)$-connections (which are automatically flat since $\rk(\SL)=1$). For instance, the proof that condition (1) of Definition \ref{def:semiproj} holds is going to be proved by closely following an argument by Simpson, via $\Lambda$-modules, but all the arguments required to prove smoothness of the moduli will be carried out by taking the $\SL$-connections point of view, because the deformation theory of such objects has been developed, contrary to the deformation theory of $\Lambda$-modules. In order to help the exposition, both the notation of $\SL$-connections and of $\Lambda$-modules will be used for the objects and the moduli spaces appropriately as needed during these proofs, depending on whether we need to study the geometry and properties of these objects from one perspective or the other. Notice that several of the intermediate results will be proved for higher rank Lie algebroids and we expect that semiprojectivity of the $\SL$-Hodge moduli also holds in that case (cf.~Remark \ref{rmk:higherrk}).

Recall the $\CC^*$-action \eqref{eq:C*-action-LHodgemod} on the $\SL$-Hodge moduli space $\SM_{\SL}^{\Hod}(r,d)$ given by
$$t \cdot (E,\nabla_\SL,\lambda) \mapsto (E,t\nabla_\SL, t\lambda).$$
We will show that with this $\CC^*$-action  $\SM_{\SL}^{\Hod}(r,d)$ becomes a smooth semiprojective variety. Recall also the surjective map $\pi:\SM_{\SL}^{\Hod}(r,d)\to\CC$ defined in \eqref{eq:project->lambda}.
\begin{lemma}
\label{lemma:limitExists}
Let $\SL$ be any Lie algebroid on $X$. Let $(E,\nabla_\SL,\lambda)\in \SM_{\SL}^{\Hod}(r,d)$ be any $(\lambda,\SL)$-connection. Then the limit $\lim_{t \to 0} (E,t \nabla_\SL, t\lambda)$ exists in $\pi^{-1}(0)\subset \SM_{\Lambda_\SL^{\op{red}}}(r,d)$.
\end{lemma}
\begin{proof}
The proof is analogous to \cite[Corollary 10.2]{Si97}. We will use \eqref{eq:L-connect<->lambdamod} to consider the associated $\Lambda_{\SL_\lambda}$-module $(E,\nabla_{\Lambda_{\SL_\lambda}})$, with $\nabla_{\Lambda_{\SL_\lambda}}:\Lambda_{\SL_\lambda}\otimes E\to E$, instead of the $(\lambda,\SL)$-connection (i.e.~$\SL_\lambda$-connection) $(E,\nabla_\SL,\lambda)$.

Consider the $\CC^*$-flat family of relative $\Lambda_{\SL_\lambda}^{\op{red}}|_{X\times \CC^*}$-modules over $\pi_{\CC}: X\times \CC \to \CC$, where $\pi_\CC(x,t)=t\lambda$, given by 
$$\left (\SE,\nabla_{\Lambda_{\SL_\lambda}^{\op{red}}} \right)=\left (\pi_X^* E, t \pi_X^*\nabla_{\Lambda_{\SL_\lambda}} \right),$$
where $\pi_X: X\times \CC^* \to X$ is the projection. For $t\ne 0$, the fibre of the family is semistable, as for any $t\ne 0$ we clearly have that the corresponding $(t\lambda,\SL)$-connection $(E,t\nabla_\SL,t\lambda)$ is semistable if and only if $(E,\nabla_\SL,\lambda)$ is semistable. By \cite[Theorem 10.1]{Si97}, there exists a family $(\overline{\SE},\overline{\nabla_{\Lambda_{\SL_\lambda}^{\op{red}}}})$ of $\Lambda_{\SL_\lambda}^{\op{red}}$-modules over $\pi_\CC:X\times \CC \to \CC$, flat over $\CC$, such that $(\overline{\SE},\overline{\nabla_{\Lambda_{\SL_\lambda}^{\op{red}}}})|_{X\times \CC^*} \cong (\pi_X^*E,t\pi_X^*\nabla_{\Lambda_{\SL_\lambda}})$ and such that $(\overline{\SE},\overline{\nabla_{\Lambda_{\SL_\lambda}^{\op{red}}}})|_{X\times \{0\}}$ is semistable. Thus, $(\overline{\SE},\overline{\nabla_{\Lambda_{\SL_\lambda}^{\op{red}}}})|_{X\times \{0\}}\in \pi^{-1}(0)$ is the limit at $t=0$ of the $\CC^*$-orbit of $(E,\nabla_\SL,\lambda)$ in $\SM_{\Lambda_\SL^{\op{red}}}(r,d)\cong \SM_{\SL}^{\Hod}(r,d)$.
\end{proof}

Now we will focus on the regularity of the $\SL$-Hodge moduli space. We will switch to use the $\SL$-connections interpretation of the points of the moduli space in our study. We start with a simple lemma.

\begin{lemma}
\label{lemma:stableHom} Let $\SL$ be any Lie algebroid on the curve $X$, and let $(E,\nabla_\SL)$ and $(E',\nabla_\SL')$ be semistable $\SL$-connections.
\begin{enumerate}
\item If $\mu(E)>\mu(E')$ then $\Hom((E,\nabla_\SL),(E',\nabla_\SL'))=0$.
\item Suppose $(E,\nabla_\SL)$ and $(E',\nabla_\SL')$ are stable and $\mu(E)=\mu(E')$. Let  $\psi\in\Hom((E,\nabla_\SL),(E',\nabla_\SL'))$ be a non-zero map. Then it is an isomorphism.
\item If $(E,\nabla_\SL)$ is stable, then its only endomorphisms are the scalars, i.e.~$\End(E,\nabla_\SL)\cong \CC$.
\end{enumerate}
\end{lemma}
\begin{proof}
The proof is classical. (1) and (2) are completely analogous to \cite[Lemma 3.2]{BGL11}. To prove (3), let $\alpha:(E,\nabla_\SL)\to (E,\nabla_\SL)$ be any endomorphism. Choose any point $x\in X$. Then $\alpha$ induces an endomorphism $\alpha_x$ of the fibre $E_x$. Let $\lambda\in \CC$ be an eigenvalue of such morphism. As $\nabla_\SL$ is $\CC$-linear, then $\alpha-\lambda \id \in \End(E,\nabla_\SL)$. By (2), we know that this map is either zero or an isomorphism. Nevertheless, we know that $\lambda$ is an eigenvalue of $\alpha_x$, so $\alpha-\lambda \id$ has a nontrivial kernel at the fibre over $x$ and, therefore, it cannot be an isomorphism. Thus, $\alpha-\lambda \id=0$, so $\alpha=\lambda \id$.
\end{proof}

Given any Lie algebroid $\SL$, the deformation theory of flat $\SL$-connections was studied in Chapter 5 of \cite{Tor11}. In particular, it follows from Theorem 47 of loc. cit. that that the Zariski tangent space to the moduli space at an integrable $\SL$-connection $(E,\nabla_\SL)$ is isomorphic to $\HH^1(X, C^\bullet(E,\nabla_\SL))$, where $C^\bullet(E,\nabla_\SL)$ is the complex
\begin{equation}\label{eq:def-complex}
C^\bullet(E,\nabla_\SL) \, : \, \End(E) \stackrel{[-,\nabla_\SL]}{\longrightarrow} \End(E)\otimes \Omega_\SL^1\stackrel{[-,\nabla_\SL]}{\longrightarrow} \ldots \stackrel{[-,\nabla_\SL]}{\longrightarrow} \End(E)\otimes \Omega_\SL^{\rk(\SL)},
\end{equation}
and that the obstruction for the deformation theory lies in $\HH^2(X, C^\bullet(E,\nabla_\SL))$.

In the next lemma we only consider rank $1$ Lie algebroids.

\begin{lemma}
\label{lemma:dimensionTangent}
Let $\SL$ be a Lie algebroid of rank $1$ on $X$ and let $(E,\nabla_\SL)$ be a stable $\SL$-connection of rank $r$ and degree $d$. Then the dimension of the Zariski tangent space to $\SM_{\SL}(r,d)$ at $(E,\nabla_\SL)$ is given by
$$\dim T_{(E,\nabla_\SL)} \SM_{\SL}(r,d) = 1-r^2\deg(\SL)+\dim \left( \HH^2(C^\bullet(E,\nabla_\SL)) \right).$$
\end{lemma}

\begin{proof}
Since $\rk(\SL)=1$, then the deformation complex \eqref{eq:def-complex} only has two terms,
$$C^\bullet(E,\nabla_\SL) \, : \, \End(E) \stackrel{[-,\nabla_\SL]}{\longrightarrow} \End(E)\otimes \Omega_\SL^1,$$
thus its hypercohomology fits in the following exact sequence
\begin{multline*}
0\longrightarrow \HH^0(C^\bullet(E,\nabla_\SL)) \longrightarrow H^0(\End(E))\longrightarrow H^0(\End(E)\otimes \Omega_\SL^1) \longrightarrow \\
\HH^1(C^\bullet(E,\nabla_\SL)) \longrightarrow H^1(\End(E)) \longrightarrow H^1(\End(E)\otimes \Omega_\SL^1) \longrightarrow \HH^2(C^\bullet(E,\nabla_\SL)) \longrightarrow 0.
\end{multline*}
Therefore,
$$\dim(\HH^1(C^\bullet(E,\nabla_\SL))) = \dim(\HH^0(C^\bullet(E,\nabla_\SL)))+\dim(\HH^2(C^\bullet(E,\nabla_\SL)))+\chi(\End(E)\otimes \Omega_\SL^1)-\chi(\End(E)).$$
We can compute each term in the previous expression working in an analogous way to \cite[Proposition 3.3]{BGL11}. By construction, $\HH^0(C^\bullet(E,\nabla_\SL))$ corresponds to sections of $\End(E)$ belonging to the kernel of the commutator $[-,\nabla_\SL]$, so $\HH^0(C^\bullet(E,\nabla_\SL))\cong H^0(\End(E,\nabla_\SL))$. By stability of $(E,\nabla_\SL)$, point (3) of Lemma \ref{lemma:stableHom} shows that $\dim(\HH^0(C^\bullet(E,\nabla_\SL)))=\dim (H^0(\End(E,\nabla_\SL))) =1$.
On the other hand, $\chi(\End(E))=r^2(1-g)$ and $\chi(\End(E)\otimes \Omega_\SL^1)=-r^2\deg(\SL)+r^2(1-g)$.
Hence
$$\dim T_{(E,\nabla_\SL)}\SM_{\SL}(r,d)=\dim(\HH^1(C^\bullet(E,\nabla_\SL)))= 1-r^2\deg(\SL)+\dim \left (\HH^2(C^\bullet(E,\nabla_\SL)) \right),$$
as claimed.
\end{proof}

Let $\SL=(V,[\cdot\,,\cdot],\delta)$ be any Lie algebroid on $X$, so no constrains on the algebraic vector bundle $V$. Recall the associated Lie algebroid $\SL_\lambda$ given by \eqref{eq:Llambda}. Then $\SL_0=(V,0,0)$ is the trivial algebroid with underlying bundle $V$. Now we aim to study the first order deformations of a semistable integrable $\SL_0$-connection $(E,\nabla_{\SL_0})$ of rank $r$ and degree $d$ (i.e.~a semistable $V^*$-twisted Higgs bundle), not just inside of $\pi^{-1}(0)=\SM_{\Lambda_{\SL_0}}(r,d)\cong\SM_{V^*}^{\Dol}(r,d)$, but rather inside the $\SL$-Hodge moduli space $\SM_{\Lambda_\SL^{\op{red}}}(r,d)\cong \SM_{\SL}^{\Hod}(r,d)$. So we allow deformations of $(E,\nabla_{\SL_0})$ not only along $\pi^{-1}(0)$ but also to $\pi^{-1}(\lambda)$ for some $\lambda\neq 0$. Recall that here $\pi:\SM_{\SL}^{\Hod}(r,d)\to\CC$ is the projection \eqref{eq:project->lambda}.

\begin{lemma}
\label{lemma:HodgeTangent}
Let $\SL=(V,[\cdot\,,\cdot],\delta)$ be any Lie algebroid on the curve $X$. Then the Zariski tangent space to the $\SL$-Hodge moduli space $\SM_{\Lambda_{\SL}^{\op{red}}}(r,d)$ at a stable point $(E,\nabla_\SL,0)$ lying over the $0$ fibre is
$$T_{(E,\nabla_\SL,0)}\SM_{\SL}^{\Hod}(r,d) \cong \frac{\left \{(c,C,\lambda_\varepsilon) \in \left( \begin{array}{l}
C^1(\SU,\End(E)) \times \\
C^0 (\SU,\End(E)\otimes \Omega_\SL) \times \CC 
\end{array} \right) \middle | \begin{array}{l}
\partial c= 0\\
\partial C = \tilde{\nabla}_\SL c + \lambda_\varepsilon \omega\\
\tilde{\nabla}_\SL C = -\lambda_\varepsilon d_\SL(\nabla_\SL)
\end{array} \right \}}{ \left \{ (\partial \eta ,\tilde{\nabla}_\SL \eta, 0) \, \middle | \,  \eta\in C^0(\SU,\End(E)) \right \}}$$
where $\SU=\{U_\alpha\}$ is an open cover of $X$ such that $E$ is trivial over each open subset $U_\alpha$ and $\omega\in C^1(\SU,\End(E)\otimes \Omega_\SL)$ is some $1$-cocycle.
Moreover, if $\pi: \SM_{\SL}^{\Hod}(r,d) \to \CC$ is the map sending $(E,\nabla_\SL,\lambda)$ to $\lambda$, then its differential $d\pi: T_{(E,\nabla_\SL,0)}\SM_{\SL}^{\Hod}(r,d) \to \CC$ is just $[(c,C,\lambda_\varepsilon)] \mapsto \lambda_\varepsilon$.
\end{lemma}

\begin{proof}
We will proceed analogously to \cite[\S 5.2]{Tor11}. Let $(E,\nabla_\SL,0)\in \pi_{\lambda}^{-1}(0)$. By definition, $(E,\nabla_\SL)$ is a $V^*$-twisted semistable Higgs bundle. Fix an open cover $\SU=\{U_\alpha\}$ of $X$ such that $E$ is trivial over each open subset $U_\alpha$. We will use the usual notation $U_{\alpha\beta}:=U_\alpha\cap U_\beta$, etc. to denote the intersections of the open subsets. For each $\alpha$ and $\beta$, let $g_{\alpha\beta}:U_{\alpha\beta} \to \GL(r,\CC)$ be the transition functions of $E$, and let $G_\alpha$ be the matrix valued function representing the Higgs field $\nabla_\SL$ in the local coordinates over $U_\alpha$.

The first order deformations of the Higgs bundle $(E,\nabla_\SL)$ are given by families of $\Lambda_\SL^{\op{red}}$-modules over each $\Spec(\CC[\varepsilon]/\varepsilon^2) \to \Spec(\CC[\lambda])$. The possible maps $\Spec(\CC[\varepsilon]/\varepsilon^2) \to \Spec(\CC[\lambda])$ are given by the choice of the image of $\varepsilon$, which must be of the form $\lambda_\varepsilon \lambda$ for some $\lambda_\varepsilon \in \CC$. Fix the value $\lambda_\varepsilon$. Then a family over $\Spec(\CC[\varepsilon]/\varepsilon^2) \to \Spec(\CC[\lambda])$ for that parameter $\lambda_\varepsilon$ is a triple $(E',\nabla_\SL',\lambda_\varepsilon \varepsilon)$ such that $E'$ is a vector bundle over $X\times \Spec(\CC[\varepsilon]/\varepsilon^2)$ and $\nabla_\SL'$ is a $(\lambda_\varepsilon \varepsilon,\SL)$-connection on $E'$. Relative to the open cover $\{U_\alpha \times \Spec(\CC[\varepsilon]/\varepsilon^2) \}$, we can write the transition functions of $E'$ as
$$g'_{\alpha\beta}=g_{\alpha\beta}+\varepsilon g_{\alpha\beta}^1,$$
where $g_{\alpha\beta}^1 \in \SO_X(U_{\alpha\beta}) \times \mathfrak{gl}_r$. Similarly, we can write locally $\nabla_\SL'$ over $U_\alpha$ as
$$\nabla_{\SL,\alpha}'=\lambda_\varepsilon \varepsilon d_\SL + G_\alpha+\varepsilon G^1_\alpha,$$
where $G^1_\alpha \in \Omega_\SL(U_\alpha)\otimes \mathfrak{gl}_r$.

Moreover, define the 1-cocycle $c\in C^1(\SU, \End(E))$ in the following way. For each $U_{\alpha\beta}$, let 
\begin{equation}\label{eq:c-cocycle}
(c_{\alpha\beta})^{(\alpha)}=g_{\alpha\beta}^1g_{\beta\alpha},
\end{equation} where we use the notation $(-)^{(\alpha)}$ to denote the matrix with respect to the basis given by the trivialization over $U_\alpha$.

Since $E$ and $E'$ are vector bundles, the following equations must be satisfied:
$$g_{\alpha\beta}g_{\beta\alpha} = 1, \ \ \ g'_{\alpha\beta}g'_{\beta\alpha} = 1,\ \ \ g_{\alpha\beta}g_{\beta\gamma}g_{\gamma\alpha}=1 \ \ \ \text{ and }\ \ \ g'_{\alpha\beta}g'_{\beta\gamma}g'_{\gamma\alpha}=1.$$
A direct computation with the first two equations yields $g_{\beta\alpha}^1=-g_{\beta\alpha}g_{\alpha\beta}^1g_{\beta\alpha}$, thus
\begin{equation}
\label{eq:cClosed1}
c_{\beta\alpha}^{(\alpha)}=g_{\alpha\beta}c_{\beta\alpha}^{(\beta)}g_{\beta\alpha} = g_{\alpha\beta} g_{\beta\alpha}^1 = -g_{\alpha\beta}^1g_{\beta\alpha}=- c_{\alpha\beta}^{(\alpha)}.
\end{equation}
On the other hand, the last couple of equations imply
$g_{\alpha\beta}^1g_{\beta\gamma} g_{\gamma\alpha}+g_{\alpha\beta}g_{\beta\gamma}^1 g_{\gamma\alpha}+g_{\alpha\beta}g_{\beta\gamma} g_{\gamma\alpha}^1=0$.
We can rewrite each summand of the last equation in terms of the cocycle $c$ as 
$$
g_{\alpha\beta}^1g_{\beta\gamma} g_{\gamma\alpha} = g_{\alpha\beta}^1g_{\beta\alpha}  = c_{\alpha\beta}^{(\alpha)},\ \ \ g_{\alpha\beta}g_{\beta\gamma}^1 g_{\gamma\alpha} =g_{\alpha\beta}g_{\beta\gamma}^1 g_{\gamma\beta} g_{\beta\alpha}  = c_{\beta\gamma}^{(\alpha)}\ \ \ \text{and}\ \ \ g_{\alpha\beta}g_{\beta\gamma} g_{\gamma\alpha}^1 =g_{\alpha\gamma}g_{\gamma\alpha}^1  =c_{\gamma\alpha}^{(\alpha)}.
$$
Thus, we obtain
\begin{equation}
\label{eq:cClosed2}
c_{\alpha\beta}^{(\alpha)}+c_{\beta\gamma}^{(\alpha)}+c_{\gamma\alpha}^{(\alpha)}=0,
\end{equation}
so $c\in C^1(\SU,\End(E))$ in \eqref{eq:c-cocycle} is a 1-cocycle.

On the other hand, as $(E,\nabla_\SL)$ is a $V^*$-twisted Higgs bundle and $(E',\nabla_\SL',\lambda_\varepsilon \varepsilon)$ is a ($\lambda_\varepsilon \varepsilon, \SL)$-connection, then we must have $\nabla_{\SL,\beta}=g_{\beta\alpha}\nabla_{\SL,\alpha} g_{\alpha\beta}$ and $\nabla_{\SL,\beta}'=g_{\beta\alpha}\nabla_{\SL,\alpha}' g_{\alpha\beta}$ on $U_{\alpha\beta}$.
Expanding each side of the last expression and taking into account that $\varepsilon^2=0$ we obtain
\begin{equation*}
\begin{split}
\lambda_\varepsilon \varepsilon d_\SL + G_\beta+\varepsilon G^1_\beta &= g'_{\beta\alpha}(\lambda_\varepsilon \varepsilon d_\SL + G_\alpha+\varepsilon G^1_\alpha)g'_{\alpha\beta} \\
&=\lambda_\varepsilon\varepsilon d_\SL + \lambda_\varepsilon \varepsilon g_{\beta\alpha} d_\SL g_{\alpha\beta}
 + g_{\beta\alpha} G_\alpha g_{\alpha\beta}+ \varepsilon \left( g_{\beta\alpha}^1 G_\alpha g_{\alpha\beta} + g_{\beta\alpha}G_\alpha^1g_{\alpha\beta}+g_{\beta\alpha}G_\alpha g_{\alpha\beta}^1\right),
\end{split}
\end{equation*}
hence we conclude that
\begin{equation}
\label{eq:smoothHodgeModuli1}
G_\beta^1=g_{\beta\alpha}^1 G_\alpha g_{\alpha\beta} + g_{\beta\alpha}G_\alpha^1g_{\alpha\beta}+g_{\beta\alpha}G_\alpha g_{\alpha\beta}^1+\lambda_\varepsilon g_{\beta\alpha} d_\SL g_{\alpha\beta}.
\end{equation}

Define the $0$-cocycle $C\in C^0(\SU,\End(E)\otimes \Omega_\SL)$ by taking $$C_\alpha^{(\alpha)}=G_\alpha^1,$$ for each $\alpha$. Then, the equality \eqref{eq:smoothHodgeModuli1} written in terms of the cocycles $c$ and $C$, reads as
\begin{equation*}
\begin{split}
C_\beta^{(\alpha)}&=g_{\alpha\beta} G_\beta^1 g_{\beta\alpha}\\
&= g_{\alpha\beta}g_{\beta\alpha}^1 G_\alpha + G_\alpha^1 + G_\alpha g_{\alpha\beta}^1 g_{\beta\alpha} + \lambda_\varepsilon (d_\SL g_{\alpha\beta})g_{\beta\alpha}\\
&=-c_{\alpha\beta}^{(\alpha)} G_\alpha + C_\alpha^{(\alpha)} + G_\alpha c_{\alpha\beta}^{(\alpha)} + \lambda_\varepsilon (d_\SL g_{\alpha\beta})g_{\beta\alpha}\\
&= C_\alpha^{(\alpha)}+ \left [G_\alpha,c_{\alpha\beta}^{(\alpha)} \right ]+ \lambda_\varepsilon (d_\SL g_{\alpha\beta})g_{\beta\alpha}.
\end{split}
\end{equation*}

Consider the $1$-cochain $\omega\in C^1(\SU,\End(E)\otimes \Omega_\SL)$ defined as
\begin{equation}\label{eq:omega-cocycle}
\omega_{\alpha\beta}^{(\alpha)}=(d_\SL g_{\alpha\beta})g_{\beta\alpha},
\end{equation}
for each $\alpha,\beta$. Thus
\begin{equation}
\label{eq:Cdifferential}
C_\beta^{(\alpha)}-C_\alpha^{(\alpha)} = [G_\alpha,c_{\alpha\beta}^{(\alpha)}]+\lambda_\varepsilon \omega_{\alpha\beta}^{(\alpha)}.
\end{equation}
Observe that
$$(d_\SL g_{\alpha\beta})g_{\beta\alpha} + g_{\alpha\beta}(d_\SL g_{\beta\alpha}) = d_\SL(g_\alpha\beta g_\beta\alpha)=d_\SL(1)=0$$
so $d_\SL g_{\beta\alpha}=-g_{\beta\alpha} (d_\SL g_{\alpha\beta}) g_{\beta\alpha}$, and we get
$$\omega_{\beta\alpha}^{(\alpha)}=g_{\alpha\beta} (d_\SL g_{\beta\alpha}) g_{\alpha\beta} g_{\beta\alpha} = g_{\alpha\beta} (d_\SL g_{\beta\alpha}) = -(d_\SL g_{\alpha\beta}) g_{\beta\alpha} = -\omega_{\alpha\beta}^{(\alpha)}.$$
On the other hand
\begin{equation*}
\begin{split}
0&=d_\SL(1)\\&
=d_\SL(g_{\alpha\beta}g_{\beta\gamma}g_{\gamma\alpha})\\
&=(d_\SL g_{\alpha\beta}) g_{\beta\gamma} g_{\gamma\alpha} + g_{\alpha\beta}(d_\SL g_{\beta\gamma})g_{\gamma\alpha}+g_{\alpha\beta}g_{\beta\gamma}(d_\SL g_{\gamma\alpha})\\
&= (d_\SL g_{\alpha\beta}) g_{\beta\alpha} + g_{\alpha\beta} (d_\SL g_{\beta\gamma}) g_{\gamma \beta} g_{\beta\alpha} + g_{\alpha\gamma} (d_\SL g_{\gamma \alpha}) g_{\alpha\gamma} g_{\gamma\alpha}\\
& =\omega_{\alpha\beta}^{(\alpha)} + \omega_{\beta\gamma}^{(\alpha)} + \omega_{\gamma\alpha}^{(\alpha)}.
\end{split}
\end{equation*}
Hence $\omega$ is in fact a $1$-cocycle.

Finally, flatness of $\nabla_\SL$ and $\nabla_\SL'$ implies that for each $\alpha$ we have $0=\nabla_{\SL,\alpha}^2 = G_\alpha\wedge G_\alpha$ and
$$0=(\nabla_{\SL,\alpha}')^2 = (\lambda_\varepsilon \varepsilon d_\SL + G_\alpha+ \varepsilon G_\alpha^1)^2 = \lambda_\varepsilon \varepsilon d_\SL(G_\alpha) + G_\alpha \wedge G_\alpha + \varepsilon G_\alpha \wedge G_\alpha^1 + \varepsilon G_\alpha^1 \wedge G_\alpha.$$
Denote by $d_\SL(\nabla_\SL) \in C^0(\SU, \End(E)\otimes \Omega_L^2)$ the $0$-cocycle defined locally as $d_\SL(\nabla_\SL)^{(\alpha)}=d_\SL(G_\alpha)$. Then, we can write the previous equation in terms of $C$ and $d_\SL(\nabla_\SL)$ as follows. The flatness equations yield $G_\alpha\wedge G_\alpha^1 + G_\alpha^1 \wedge G_\alpha = -\lambda_\varepsilon d_\SL(G_\alpha)$, so
\begin{equation}
\label{eq:CCurvature}
\tilde{\nabla}_\SL C_\alpha^{(\alpha)} = -\lambda_\varepsilon d_\SL(\nabla_\SL)^{(\alpha)},
\end{equation}
where $\tilde{\nabla}_\SL=[-,\nabla_\SL] : \End(E) \to \End(E)\otimes \Omega_\SL$ is the induced map on $\End(E)$ by $\nabla_\SL$. We can express equations \eqref{eq:cClosed1}, \eqref{eq:cClosed2}, \eqref{eq:Cdifferential} and \eqref{eq:CCurvature} globally as follows. Each deformation of $(E,\nabla_\SL,0)$ is given by a triple $$(c,C,\lambda_\varepsilon),$$ with $c\in C^1(\SU,\End(E))$, $C \in C^0 (\SU,\End(E)\otimes \Omega_\SL)$ and $\lambda_\varepsilon \in \CC$,
such that
\begin{equation}
\label{eq:DeformationSpace}
\left \{ \begin{array}{l}
\partial c= 0\\
\partial C = \tilde{\nabla}_\SL c + \lambda_\varepsilon \omega\\
\tilde{\nabla}_\SL C = -\lambda_\varepsilon d_\SL(\nabla_\SL).
\end{array} \right.
\end{equation}
On the other hand, two such triples $(c,C,\lambda_\varepsilon)$ and $(\bar{c},\bar{C},\bar{\lambda}_\varepsilon)$ give rise to equivalent deformations $(E',\nabla_\SL',\lambda_\varepsilon)$ and $(\overline{E}',\overline{\nabla}'_{\SL},\bar{\lambda}_\varepsilon)$ of $(E,\nabla_\SL,0)$ if and only if $\lambda_\varepsilon=\bar{\lambda}_\varepsilon$ and there exists a $0$-cocycle of local automorphisms $\xi_\alpha: U_\alpha \times \Spec(\CC[\varepsilon]/\varepsilon^2) \to \GL_r$ of the form $\xi_\alpha=\id+\varepsilon \eta_\alpha$ with $\eta_\alpha:U_\alpha \to \mathfrak{gl}_r$ such that
\begin{equation}
\label{eq:eqDeformation1}
\overline{g}_{\alpha\beta}' (g_{\alpha\beta}')^{-1} = \xi_\beta \xi_\alpha^{-1}
\end{equation}
and
\begin{equation}
\label{eq:eqDeformation2}
\overline{\nabla}'_{\SL,\alpha}= \xi_\alpha^{-1} \nabla_{\SL,\alpha}' \xi_\alpha.
\end{equation}
Here, following the previous notation, we write $\overline{E}'$ and $\overline{\nabla}'_{\SL}$ locally in the corresponding trivialization over $\SU$ as
$$\overline{g}'_{\alpha\beta}=\overline{g}_{\alpha\beta}+\varepsilon \overline{g}_{\alpha\beta}^1=g_{\alpha\beta}+\varepsilon \overline{g}_{\alpha\beta}^1\ \ \ \text{ and }\ \ \ \overline{\nabla}'_{\SL,\alpha}=\bar{\lambda}_\varepsilon \varepsilon d_\SL + \overline{G}_\alpha+\varepsilon \overline{G}^1_\alpha=\lambda_\varepsilon \varepsilon d_\SL + G_\alpha+\varepsilon \overline{G}^1_\alpha.$$
We have that $\xi_\beta \xi_\alpha^{-1}=(\id+\varepsilon\eta_\beta)(\id-\varepsilon \eta_\alpha)=\id + \varepsilon \left(\eta_\beta-\eta_\alpha \right)$
and
$$\overline{g}_{\alpha\beta}' (g_{\alpha\beta}')^{-1} = (g_{\alpha\beta}+\varepsilon \overline{g}_{\alpha\beta}^1)(g_{\beta\alpha}+\varepsilon g_{\beta\alpha}^1) = \id+\varepsilon ( g_{\alpha\beta}g_{\beta\alpha}^1 + \overline{g}_{\alpha\beta}^1g_{\beta\alpha}) = \id+\varepsilon(\overline{c}_{\alpha\beta} - c_{\alpha\beta}).$$
From \eqref{eq:eqDeformation1}, we obtain $\overline{c}_{\alpha\beta} - c_{\alpha\beta} = \eta_\beta-\eta_\alpha$,
so $\overline{c}-c=\partial \eta$.
On the other hand,
\begin{equation*}
\begin{split}
\xi_\alpha^{-1} \nabla_{\SL,\alpha}' \xi_\alpha&= \lambda_\varepsilon  \varepsilon d_\SL + \lambda_\varepsilon  \varepsilon (\id-\varepsilon \eta_\alpha) d_\SL (\id+ \varepsilon \eta_\alpha) + (\id-\varepsilon \eta_\alpha) G_\alpha (\id+\varepsilon \eta_\alpha) + \varepsilon (\id-\varepsilon \eta_\alpha) G_\alpha^1 (\id+\varepsilon \eta_\alpha) \\&= \lambda_\varepsilon \varepsilon d_\SL + G_\alpha + \varepsilon \left( G_\alpha\eta_\alpha-\eta_\alpha G_\alpha + G_\alpha^1 \right).
\end{split}
\end{equation*}
Hence, \eqref{eq:eqDeformation2} yields $\overline{G}_\alpha^1 = [G_\alpha,\eta_\alpha] + G_\alpha^1$
or, equivalently, $\overline{C}_\alpha-C_\alpha = [G_\alpha,\eta_\alpha]=\tilde{\nabla}_\SL\eta_\alpha$,
and thus, $\overline{C}-C=\tilde{\nabla}_\SL\eta$.
We finally conclude that the deformation space of $\SM_{\SL}^{\Hod}(r,d)$ at $(E,\nabla_\SL,0)$ is
\begin{equation}\label{eq:Zartangspce}
T_{(E,\nabla_\SL,0)}\SM_{\SL}^{\Hod}(r,d) \cong \frac{\left \{(c,C,\lambda_\varepsilon) \in \left( \begin{array}{l}
C^1(\SU,\End(E)) \times \\
C^0 (\SU,\End(E)\otimes \Omega_\SL) \times \CC 
\end{array} \right) \middle | \begin{array}{l}
\partial c= 0\\
\partial C = \tilde{\nabla}_\SL c + \lambda_\varepsilon \omega\\
\tilde{\nabla}_\SL C = -\lambda_\varepsilon d_\SL(\nabla_\SL)
\end{array} \right \}}{ \left \{ (\partial \eta ,\tilde{\nabla}_\SL \eta, 0) \,  \middle | \, \eta\in C^0(\SU,\End(E)) \right \}}
\end{equation}
and the map $d\pi: T_{(E,\nabla_\SL,0)}\SM_{\SL}^{\Hod}(r,d) \to \CC$ is just $[(c,C,\lambda_\varepsilon)] \mapsto \lambda_\varepsilon$. Notice that a priori we should mod out the right-hand side of \eqref{eq:Zartangspce} by the automorphism group $\Aut(E,\nabla_\SL)$ of $(E,\nabla_\SL)$. However, since $(E,\nabla_\SL)$ is stable, then it is simple by Lemma \ref{lemma:stableHom} (3) and so $\Aut(E,\nabla_\SL)=\CC^*$ acts trivially.
\end{proof}

\begin{remark}\label{rmk:defo-hypercoh}
While the deformation theory of $\SM_{\SL}(r,d)$ is governed by a nice deformation complex (as expectable for this type of deformation problems), we have not been able to provide, in general, a natural cohomological interpretation for the deformation theory of $\SM_{\SL}^{\Hod}(r,d)$.

This can, however, be achieved in certain cases. For instance, suppose $\rk(\SL)=1$ and suppose $E$ is a stable vector bundle over $X$. By Corollary \ref{cor:existsConnection}, $E$ admits an integrable $\SL$-connection $\nabla_{\SL,0}:E\to E\otimes \Omega_\SL$. Let us consider the family $(\pi_X^* E, \lambda \pi_X^* \nabla_{\SL,0},\lambda)$ over $X\times \CC$, where $\pi_X:X\times \CC \to X$ is the projection, and consider the infinitesimal family over $\Spec(\CC[\varepsilon]/\varepsilon^2)$ around $0$ given by $(\pi_X^* E,\varepsilon \pi_X^*\nabla_{\SL,0},\varepsilon)$. We can now express the family locally in a similar way to the previous Lemma. Given an open cover $\{U_\alpha \times \Spec(\CC[\varepsilon]/\varepsilon^2)$, let $g_{\alpha\beta}:U_{\alpha\beta}\to \GL(r,\CC)$ be the transition functions of $E$. As $\varepsilon\nabla_{\SL,0}$ is an $(\varepsilon,\SL)$-connection, we can express it locally over $U_\alpha$ as
$$\varepsilon\nabla_{\SL,0,\alpha}=\varepsilon d_\SL+\varepsilon G_\alpha^1$$
for some $G_\alpha^1\in \Omega_\SL(U_\alpha)\otimes \mathfrak{gl}_r$. As $G_\alpha^1$ comes from an actual $\SL$-connection, we must have 
$$\varepsilon\nabla_{\SL,0,\beta}=g_{\beta\alpha} \varepsilon \nabla_{\SL,0,\alpha} g_{\alpha\beta}$$
on the overlaps $U_{\alpha\beta}$.
Plugging in the local representation of $\nabla_\SL$ yields
$$\varepsilon d_\SL+\varepsilon G_\beta^1=\varepsilon d_\SL+ \varepsilon  g_{\beta\alpha} d_\SL g_{\alpha\beta}+\varepsilon g_{\beta\alpha} G_\alpha^1 g_{\alpha\beta}.$$
Thus,
$$g_{\alpha\beta} G_\beta^1 g_{\beta\alpha}=d_\SL g_{\alpha\beta} g_{\beta\alpha} + G_\alpha^1= \omega_{\alpha\beta}^{(\alpha)}+ G_\alpha^1.$$
Let $\Omega \in C^0(\SU,\End(E)\otimes \Omega_L)$ be the $0$-cocycle defined as $\Omega_\alpha^{(\alpha)}=G_\alpha^1$.
Then $\omega_{\alpha\beta}^{(\alpha)}=\Omega_\beta^{(\alpha)}-\Omega_\alpha^{(\alpha)},$
hence $$\omega=\partial \Omega.$$
Now, since $\SL$ has rank $1$, the integrability condition is automatic and so 
$$T_{(E,\nabla_\SL,0)}\SM_{\SL}^{\Hod}(r,d) \cong \frac{\left \{(c,C,\lambda_\varepsilon) \in \left( \begin{array}{l}
C^1(\SU,\End(E)) \times \\
C^0 (\SU,\End(E)\otimes \Omega_\SL) \times \CC 
\end{array} \right) \middle | \begin{array}{l}
\partial c= 0\\
\partial C = \tilde{\nabla}_\SL c + \lambda_\varepsilon \partial \Omega
\end{array} \right \}}{ \left \{ (\partial \eta ,\tilde{\nabla}_\SL \eta, 0) \,  \middle | \, \eta\in C^0(\SU,\End(E)) \right \}}.$$
Then, the map $(c,C,\lambda_\varepsilon) \mapsto (c,C-\lambda_\varepsilon \Omega, \lambda_\varepsilon)$ induces an isomorphism
\begin{equation*}
\begin{split}
T_{(E,\nabla_\SL,0)}\SM_{\SL}^{\Hod}(r,d)
&\cong \frac{\left \{(c,D,\lambda_\varepsilon) \in \left( \begin{array}{l}
C^1(\SU,\End(E)) \times \\
C^0 (\SU,\End(E)\otimes \Omega_\SL) \times \CC 
\end{array} \right) \middle | \begin{array}{l}
\partial c= 0\\
\partial D = \tilde{\nabla}_\SL c
\end{array} \right \}}{ \left \{ (\partial \eta ,\tilde{\nabla}_\SL \eta, 0) \,  \middle | \, \eta\in C^0(\SU,\End(E)) \right \}} \\
&\cong \HH^1(C^\bullet(E,\nabla_\SL)) \times \CC,
\end{split}
\end{equation*}
yielding the desired cohomological interpretation of the deformation space. It is clear in this case that $\HH^1(C^\bullet(E,\nabla_\SL))$ parameterises the deformations along the fibre $p_\lambda^{-1}(0)$ of the projection $p_\lambda:\SM_{\SL}^{\Hod}(r,d)\to\CC$ from \eqref{eq:project->lambda}, and $\CC$ parameterises the deformations of $\lambda$ i.e.~along the target of $p_\lambda$.

However, a vector bundle $E$ need not admit an integrable $\SL$-connection and, therefore, the associated cocycle $\omega$ in \eqref{eq:omega-cocycle} may not be exact. Moreover, for higher rank $\SL$, the presence of the integrability condition breaks the previous trivialization of the deformation theory and the corresponding cohomological description.

We believe that the somehow unnatural presentation of the deformation theory of Lemma \ref{lemma:HodgeTangent} is a reflection of the fact that the moduli space of Higgs bundles admits a broader range of deformations than the ones considered in this section, as suggested in \cite[Section 7.3]{Tor11} and as described explicitly for rank 1 Lie algebroids in Theorem \ref{thm:classificationLieAlgebroidsRank1}. More precisely, for each family of Lie algebroid structures over $V$, $\SL\to X\times T$ on $T$, we obtain a moduli space $\SM_{\Lambda_\SL}(r,d) \to T$ over $T$. Each family going through the trivial Lie algebroid $(V,0,0)$ gives rise to a deformation of the moduli space $\SM_{V^*}^{\Dol}(r,d)$. More generally, the infinitesimal deformations of the trivial algebroid structure give rise to deformations of $\SM_{V^*}^{\Dol}(r,d)$. We expect that if we considered a versal family which represented the whole space of such infinitesimal deformations (which, for Lie algebroids on a line bundle $L$, would correspond to a family of Lie algebroids over $H^0(L^{-1}\otimes T_X)$) then that space would get indeed a natural cohomological interpretation.

In Lemma  \ref{lemma:HodgeTangent}  we are only considering ``radial sections'' of such deformation space, corresponding to families over $\CC$ of the form \eqref{eq:Llambda}, therefore obtaining ``sections'' or ``cuts'' of the whole deformation space, and these deformations do not seem to exhibit a natural cohomological description anymore.
\end{remark}

The preceding Lemmas \ref{lemma:dimensionTangent} and \ref{lemma:HodgeTangent} now allow us to establish smoothness of the $\SL$-Hodge moduli space under the following conditions.

\begin{lemma}
\label{lemma:smoothHodgeModuli}
Let $r\geq 1$ and $d$ be coprime and  $\SL=(L,[\cdot\,,\cdot],\delta)$ be a Lie algebroid on a smooth complex projective curve of genus $g\ge 2$ such that $\rk(\SL)=1$ and $\deg(\SL)<2-2g$. Then:
\begin{enumerate}
\item $\SM_{\SL}(r,d)$ is a smooth variety, whose connected components have all dimension $1-r^2\deg(L)$;
\item  $\SM_{\SL}^{\Hod}(r,d)$ is a smooth variety, whose connected components have all dimension $2-r^2\deg(L)$. Moreover, the map $\pi:\SM_{\SL}^{\Hod}(r,d)\to\CC$ from \eqref{eq:project->lambda} is a smooth submersion.
\end{enumerate}
\end{lemma}

\begin{proof}
Let us start by proving that $\SM_{\SL}(r,d)$ is a smooth variety of dimension $1-r^2\deg(L)$. Let $(E,\nabla_\SL)\in \SM_{\SL}(r,d)$. Consider the map $\CC^* \to \SM_{\SL}^{\Hod}(r,d)$ given by $t \mapsto (E,t\nabla_\SL,t)$. By Lemma \ref{lemma:limitExists}, the limit of the $\CC^*$-action at zero exists, so this map extends to a curve $\gamma: \CC \to \SM_{\SL}^{\Hod}(r,d)$, which is a section of the map $\pi: \SM_{\SL}^{\Hod}(r,d) \to \CC$. Let $(E_0,\nabla_{\SL,0},0):=\gamma(0)$.

Consider the map $\rho:\CC  \to \ZZ$ given by
$$\rho(\lambda) =\dim \left. \left(\gamma^*T\SM_{\SL}^{\Hod}(r,d) \right) \right |_{\lambda}.$$
We know from \eqref{eq:isomfnonzeroibres} that the $\CC^*$-action produces an isomorphism between any nonzero fibre of $\pi$ and $\pi^{-1}(1)$, yielding an isomorphism
$$\pi^{-1}(\CC^*) \cong \SM_{\SL}(r,d) \times \CC^*.$$
Therefore, for every $\lambda \ne 0$ we have
$$\rho(\lambda)=\dim \left . \left(\gamma^*T\SM_{\SL}^{\Hod}(r,d) \right)\right |_{\lambda} = \dim T_{(E,\nabla_\SL)} \SM_{\SL}(r,d) +1$$
The map $\rho$ is upper semicontinuous, so applying Lemma \ref{lemma:dimensionTangent}, we get
\begin{equation}\label{eq:rho(o)geq}
\rho(0)\ge \dim T_{(E,\nabla_\SL)} \SM_{\SL}(r,d) +1 = 2-r^2\deg(L)+\dim \left( \HH^2(C^\bullet(E,\nabla_\SL)) \right)\ge 2-r^2\deg(L).
\end{equation}
Note that since $r$ and $d$ are coprime, every semistable $\SL$-connection is actually stable, so Lemma \ref{lemma:dimensionTangent} applies at every point of $\SM_{\SL}(r,d)$.

On the other hand, we have that
$$\dim T_{(E_0,\nabla_{\SL,0},0)}\SM_{\SL}^{\Hod}(r,d) = \dim \ker d\pi+ \dim \im d\pi \le \dim \ker d\pi+1.$$
The kernel of $d\pi|_{(E_0,\nabla_{\SL,0},0)}$ can be computed explicitly through our formula for the Zariski tangent space given in Lemma \ref{lemma:HodgeTangent} (which holds since every point is stable again by coprimality of $r$ and $d$)
\begin{equation*}
\begin{split}
\ker d\pi|_{(E_0,\nabla_{\SL,0},0)}& \cong \frac{\left \{(c,C,0) \in \left( \begin{array}{l}
C^1(\SU,\End(E_0)) \times \\
C^0 (\SU,\End(E_0)\otimes \Omega_\SL) \times \CC 
 \end{array} \right) \middle | \begin{array}{l}
\partial c= 0\\
\partial C = \tilde{\nabla}_{\SL,0} c\\
\tilde{\nabla}_{\SL_0} C = 0
\end{array} \right \}}{ \left \{ (\partial \eta ,\tilde{\nabla}_{\SL,0} \eta, 0) \, \middle | \, \eta\in C^0(\SU,\End(E_0)) \right \}}\\
&\cong \HH^1\left (\End(E_0) \stackrel{[-,\nabla_{\SL,0}]}{\longrightarrow} \End(E_0)\otimes \Omega_\SL^1 \right)\\& \cong T_{(E_0,\nabla_{\SL,0})} \SM_{L^{-1}}^{\Dol}(r,d).
\end{split}
\end{equation*}
Using Lemma \ref{lemma:smoothModuliHiggs}, we know that $\SM_{L^{-1}}^{\Dol}(r,d)$ is smooth of dimension $1-r^2\deg(L)$, so
$$\dim \ker d\pi|_{(E_0,\nabla_{\SL,0},0)} =1-r^2\deg(L)$$
and, therefore,
\begin{equation}\label{eq:rho(o)leq}
\rho(0)=\dim T_{(E_0,\nabla_{\SL,0},0)}\SM_{\SL}^{\Hod}(r,d) \le \dim \ker d \pi|_{(E_0,\nabla_{\SL,0},0)} + 1 \le 2-r^2\deg(L).
\end{equation}

From \eqref{eq:rho(o)geq} and \eqref{eq:rho(o)leq}, we conclude that
$$\rho(0)=2-r^2\deg(L),$$
that is, $\HH^2(C^\bullet(E,\nabla_\SL))=0$ for all $(E,\nabla_\SL)\in \SM_{\SL}(r,d)$. 

The upshot is that the deformation theory at $(E,\nabla_\SL)$ is unobstructed and the dimension of the Zariski tangent space $T_{(E,\nabla_\SL)}\SM_{\SL}(r,d)$ is $1-r^2\deg(L)$ for each $(E,\nabla_\SL)$. As a consequence, by \cite{FM98} the moduli space $\SM_{\SL}(r,d)$ is a smooth variety of dimension $1-r^2\deg(L)$, completing the proof of (1).

Let us now consider point (2), i.e.~the regularity of the $\SL$-Hodge moduli space and the map $\pi$. First note that we have the isomorphism
$$\pi^{-1}(\CC^*) \cong \SM_{\SL}(r,d) \times \CC^*, $$
and by (1), $\SM_{\SL}(r,d)$ is smooth, so $\pi^{-1}(\CC^*)$ is also smooth and the map $\pi|_{\pi^{-1}(\CC^*)}:\pi^{-1}(\CC^*)\to \CC^*$ is clearly a smooth submersion. Therefore, it is enough to study the deformation of the elements in the zero fibre of $\pi$ and then check that the dimension of the corresponding Zariski tangent spaces coincides with the expected one and that the differential of the map $\pi$ is surjective at those points.

Let us consider the subvariety
$$\overline{\pi^{-1}(\CC^*)} \subset \SM_{\SL}^{\Hod}(r,d)$$
given by the closure of $\pi^{-1}(\CC^*)$ in the $\SL$-Hodge moduli space.
The $\CC^*$-flow through any point of $\SM_{\SL}^{\Hod}(r,d)$ has a limit at $0$ in $\pi^{-1}(0)$, due to Lemma \ref{lemma:limitExists}, so
$$\pi^{-1}(0)\cap \overline{\pi^{-1}(\CC^*)} \ne \emptyset.$$
By (1) and \eqref{eq:isomfnonzeroibres}, we have $\dim \pi^{-1}(\lambda)=1-r^2\deg(L)$ for every $\lambda\ne 0$. Hence, by semicontinuity, each component of $\pi^{-1}(0)\cap \overline{\pi^{-1}(\CC^*)}$ has dimension at least $1-r^2\deg(L)$. By Lemma \ref{lemma:smoothModuliHiggs} the variety $\pi^{-1}(0)=\SM_{L^{-1}}^{\Dol}(r,d)$ is smooth and connected of dimension $1-r^2\deg(L)$, so we conclude that 
$$\pi^{-1}(0)\cap \overline{\pi^{-1}(\CC^*)} = \pi^{-1}(0)$$
and thus $\overline{\pi^{-1}(\CC^*)} = \SM_{\SL}^{\Hod}(r,d)$.

As $\pi^{-1}(\CC^*) \cong \SM_{\SL}(r,d)\times \CC^*$, then we know that for any $(E',\nabla_\SL',\lambda')\in \pi^{-1}(\CC^*)$,
$$\dim T_{(E',\nabla_\SL',\lambda')}\SM_{\SL}^{\Hod}(r,d) = \dim T_{(E',\nabla_\SL'/\lambda)} \SM_{\SL}(r,d) +1 = 2-r^2\deg(L)$$
so, by semicontinuity, for each $(E,\nabla_\SL,0)\in \pi^{-1}(0)\cap \overline{\pi^{-1}(\CC^*)} = \pi^{-1}(0)$ we have
$$2-r^2\deg(L) \le \dim T_{(E,\nabla_\SL,0)}\SM_{\SL}^{\Hod}(r,d)= 1-r^2\deg(L)+\dim \im d\pi\le 2-r^2\deg(L).$$
Hence, we have
$$\dim T_{(E,\nabla_\SL,0)}\SM_{\SL}^{\Hod}(r,d)= 2-r^2\deg(L)\ \ \ \text{ and }\ \ \ \dim\im d\pi|_{(E,\nabla_\SL,0)}=1,$$
for each $(E,\nabla_\SL,0)\in \pi^{-1}(0)$.
So the map $\pi$ is a smooth submersion with equidimensional fibres and the dimension of the Zariski tangent space of the moduli space $\SM_{\SL}^{\Hod}(r,d)$ at any point is constant and coincides with the dimension of the scheme which is therefore smooth.
\end{proof}

\begin{remark}
We will see in Theorem \ref{thm:homotopygrps} below that, under the stated conditions, $\SM_{\SL}(r,d)$ is actually connected, hence so is $\SM_{\SL}^{\Hod}(r,d)$.
\end{remark}

Combining the previous results yields the desired semiprojectivity and regularity of the $\SL$-Hodge moduli space.

\begin{theorem}
\label{thm:semiprojectiveHodge}
Let $X$ be a smooth complex projective curve of genus $g$. Let $\SL=(L,[\cdot\,,\cdot],\delta)$ be a Lie algebroid on $X$ such that $\rk(\SL)=1$ and $\deg(\SL)<2-2g$. Then the moduli space $\SM_{\SL}^{\Hod}(r,d)\cong \SM_{\Lambda_\SL^{\op{red}}}(r,d)$, with the $\CC^*$-action $t\cdot (E,\nabla_\SL,\lambda)=(E,t\nabla_\SL,t\lambda)$, is a semiprojective variety.

If, moreover, $r$ and $d$ are coprime, $r\ge 1$ and $g\ge 2$, then it is a smooth semiprojective variety and the map $\pi : \SM_{\SL}^{\Hod}(r,d) \to \CC$ from \eqref{eq:project->lambda} is a surjective $\CC^*$-equivariant submersion covering the standard action on $\CC$.
\end{theorem}
\begin{proof}
By the GIT construction of \cite{Si94}, $\SM_{\Lambda_\SL^{\op{red}}}(r,d)$ is a complex quasi-projective variety. Lemma \ref{lemma:limitExists} ensures that for every $(E,\nabla_\SL,\lambda)\in \SM_{\SL}^{\Hod}(r,d)\cong \SM_{\Lambda_\SL^{\op{red}}}(r,d)$ the limit $\lim_{t\to 0} (E,t\nabla_\SL,t\lambda)$ exists. Moreover, the fixed-point set corresponds to the fixed-point set of the $\CC^*$-action in $\pi^{-1}(0)$, which coincides with the moduli space of $L^{-1}$-twisted Higgs bundles. Then Lemma \ref{lemma:fixedProperHiggs} implies that $\SM_{\SL}^{\Hod}(r,d)^{\CC^*}$ is proper. So $\SM_{\SL}^{\Hod}(r,d)$ is a semiprojective variety.
In the coprime case, the smoothness claim follows from Lemma \ref{lemma:smoothHodgeModuli}.
\end{proof}

\begin{remark}\label{rmk:higherrk}
We expect that the above results still hold true for higher rank Lie algebroids $\SL=(V,[\cdot\,,\cdot],\delta)$ with $V$ polystable such that $\mu(V)<2-2g$. Indeed most of the above arguments go through immediately in this situation, except in two related steps. First, Lemma \ref{lemma:dimensionTangent} really requires rank $1$ Lie algebroids, because it is only in that setting that the deformation complex \eqref{eq:def-complex} has only two terms. If $\rk(\SL)=n$, then \eqref{eq:def-complex} has $n$ terms, and it is not clear how to proceed to compute the dimension of $\SM_{\SL}(r,d)$. Similarly, Lemma \ref{lemma:smoothModuliHiggs} also requires the twisting to be a line bundle, and the corresponding result for higher rank twistings is not yet known, by similar reasons (notice that the infinitesimal study carried out in \cite{BR94} is done for any twisting, but it does not take into account the integrability condition on the Higgs field).
\end{remark}

\section{Grothendieck motives of smooth semiprojective varieties}
\label{section:motives}

By the end of this section we will prove that, under certain assumptions on $\SL=(L,[\cdot\,,\cdot],\delta)$, the motivic class of the moduli space of integrable $\SL$-connections $\SM_\SL(r,d)$ equals that of the $L^{-1}$-twisted Higgs bundle moduli space. The corresponding result for $E$-polynomials follows immediately. By using results of Hausel--Rodriguez-Villegas, we can also state the analogous result for their Hodge structures. This will follow from Theorem \ref{thm:semiprojectiveHodge}, together with a general result on the Grothendieck motivic class of any smooth semiprojective variety which will be proved along the way, via Bialynicki-Birula stratifications induced by the given $\CC^*$-action. 

\subsection{Grothendieck ring of varieties, motives and $E$-polynomials}

The main goal of this paper is to compare the class of the moduli spaces $\SM_{\SL}(r,d)$ in the (completion of the) Grothendieck ring of varieties by varying the Lie algebroid $\SL$. In this brief section we recall such ring and its basic properties.

Denote by $\VarC$ the category of quasi-projective varieties over $\CC$. For each $Y \in \VarC$, let $[Y]$ denote the corresponding isomorphism class. Consider the group obtained by the free abelian group on isomorphism classes $[Y]$, modulo the relation 
\[[Y]=[Y']+[Y\setminus Y'],\] where $Y'\subset Y$ is a Zariski-closed subset. In particular, in such group, $$[Y]+[Z]=[Y\sqcup Z],$$ where $\sqcup$ denotes disjoint union.
If we define the product 
\[[Y]\cdot [Z]=[Y\times Z],\]
in this quotient, then we obtain a commutative ring, known as the \emph{Grothendieck ring of varieties} and denoted by $K (\VarC)$.
Then $0=[\emptyset]$ and $1=[\mathrm{Spec}(\CC)]$ are the additive and multiplicative units of this ring. 

The following is an extremely useful property of $K (\VarC)$, which follows directly from the definitions via local trivializations, and which we will repeatedly use without further notice. For a proof, see for example \cite[Proposition 5(ii)]{GS96} or \cite[Proposition 4.6]{G-P23}.

\begin{proposition}
If $\pi:Y\to B$ is an algebraic fibre bundle (thus Zariski locally trivial), with fibre $F$, then $[Y]=[F]\cdot [B]$.
\end{proposition}

The class of the affine line, sometimes called the \emph{Lefschetz object}, is denoted by 
$$\LL:=[\mathbb{A}^1]=[\mathbb{C}].$$ Of course, $\LL^n=[\mathbb{A}^n]=[\mathbb{C}^n]$.
We will consider the localization $K (\VarC)[\LL^{-1}]$, and then the dimensional completion
 \[
 \hat{K} (\VarC)= \left\{ \sum_{r\geq 0} [Y_r] \, \LL^{-r} \, \middle | \,
 [Y_r]\in K(\VarC) \text{ with }\dim Y_r -r \longrightarrow -\infty \right\}.\]
 Notice that we have a map $K(\VarC)\to \hat{K} (\VarC)$.
 Observe also that $\LL^n-1$ is invertible in $\hat{K} (\VarC)$, for every $n$, with inverse equal to $\LL^{-n}\sum_{k=0}^\infty\LL^{-nk}$. This is the reason why we had to introduce the completion $\hat{K} (\VarC)$: there will be computations in which we will need to invert elements of the form $\LL^n$ or $\LL^n-1$.
 
In this paper, by motive we mean the following. 
 
\begin{definition}
Let $Y$ be a quasi-projective variety. The class $[Y]$ in $K(\VarC)$ or in $\hat{K}(\VarC)$ is called the \emph{motive}, or \emph{motivic class}, of $Y$.
\end{definition}

There are other notions of motive in different, but related, categories, such as Chow motive or Voevodsky motive. These will not appear anywhere in this paper, except in section \ref{sec:ChowVoev}.

The motive $[Y]$ is an important invariant of $Y\in\VarC$, from which it is possible to read off geometric information, such as the $E$-polynomial. 
If $Y$ is $d$-dimensional,  with pure Hodge structure, then its \emph{$E$-polynomial} is defined as
\begin{equation}\label{def:E-pol}
E(Y)=E(Y)(u,v)=\sum_{p,q=0}^d (-1)^{p+q}h_c^{p,q}(Y)u^pv^q,
\end{equation}
where $h_c^{p,q}(Y)$ stands for the dimension of the compactly supported cohomology groups $H_c^{p,q}(Y)$.
For instance, $E(\mathbb{C})=uv$.
\begin{remark}
\label{remark:E-poly}
Actually, the $E$-polynomial can be seen as a ring map
\begin{equation}\label{E-poly}
E:\hat{K}(\VarC)\longrightarrow \ZZ[u,v]\left[\hspace{-.08cm}\left[\frac{1}{uv}\right]\hspace{-.08cm}\right]
\end{equation}
with values in the Laurent series in $uv$, which takes values in $\ZZ[u,v]$ when restricted to $K(\VarC)$. Hence two varieties with the same motive have the same $E$-polynomial. 
In particular, if $Y'\subset Y$ is a closed subvariety, then $E(Y)=E(Y')+E(Y\setminus Y')$ and for an algebraic fibre bundle $Y\to B$ with fibre $F$, we have $E(Y)=E(F)E(B)$. 
\end{remark}

\subsection{Bialynicki-Birula stratifications and motives of smooth semiprojective varieties}

Let now $Y$ be any smooth semiprojective variety; recall Definition \ref{def:semiproj}. Then it admits a canonical stratification as follows. Let
$$Y^{\CC^*}=\bigcup_{\mu\in I} F_\mu$$
be the decomposition of the $\CC^*$-fixed-point loci into connected components. For each $F_\mu$, consider the subsets
$$U_\mu^+=\{ x\in Y \, | \, \lim_{t\to 0} t\cdot x \in F_\mu\}\ \ \ \text{ and }\ \ \ U_\mu^-=\{ x\in Y \, | \, \lim_{t\to \infty} t\cdot x \in F_\mu\}.$$
By (1) of Definition \ref{def:semiproj}, every point in $Y$ belongs to exactly one of the subsets $U_\mu^+$, hence there is a decomposition
$$Y= \bigcup_{\mu\in I} U_\mu^+,$$
called the \emph{Bialynicki-Birula decomposition} of $Y$. 

The arguments in \cite[\S 4]{BB73}, \cite{Kirwan84} and \cite[\S 1]{HRV15} prove the following lemma (see \cite[Appendix A]{HL21} for a compact complete proof).

\begin{lemma}\label{lem:properties-BB-decomp} Using the above notations, the following properties hold.
\begin{enumerate}
\item For every $\mu\in I$, the map $U_\mu^+\to F_\mu$ defined by $x \mapsto \lim_{t\to 0} t\cdot x$ and the map $U_\mu^-\to F_\mu$ given by $x \mapsto \lim_{t\to \infty} t\cdot x$ are Zariski locally trivial fibrations in affine spaces.
\item For every $\mu\in I$, $U_\mu^+$ is a locally closed subset of $Y$.
\item There exists an order of the components $\{\mu_i\}_{i=1}^n$ such that
$$0\subset  U_{\mu_1}^+ \subset \ldots \subset \bigcup_{i\le j} U_{\mu_i}^+ \subset \ldots \subset \bigcup_{i=1}^n U_{\mu_i}^+= Y$$
is a stratification of $Y$.
\end{enumerate}
\end{lemma}

For each $p\in F_\mu$ the tangent space $T_pY$ splits as follows
$$T_pY = T_p(U_\mu^+|_p) \oplus T_p F_\mu \oplus T_p (U_\mu^-|_p).$$
Define
\begin{equation}\label{eq:notat-dim-tangent}
N_\mu^+=\dim T_p(U_\mu^+|_p),\ \ \  N_\mu^0=\dim T_pF_\mu\ \ \ \text{ and }\ \ \ N_\mu^-=\dim T_p(U_\mu^-|_p).
\end{equation}
Clearly $N_\mu^+=\dim(U_\mu^+)-\dim(F_\mu)$ is the rank of the affine bundle $U_\mu^+\to F_\mu$ and, as we assumed that $Y$ is smooth,
\begin{equation}\label{eq:upward+downard+fix=dim}
N_\mu^++N_\mu^0+N_\mu^-=\dim Y.
\end{equation}

\begin{lemma}
\label{lemma:BBMotiveDecomposition}
Let $Y$ be a smooth complex semiprojective variety and consider the above notations. Then the motivic class of $Y$ decomposes as
$$[Y]=\sum_{\mu\in I} \LL^{N_\mu^+}[F_\mu].$$
\end{lemma}
\begin{proof}
As $Y$ is semiprojective, we have a Bialynicki-Birula decomposition which, in virtue of properties (2) and (3) of Lemma \ref{lem:properties-BB-decomp}, forms a stratification
$$0\subset  U_{\mu_1}^+ \subset \ldots \subset \bigcup_{i\le j} U_{\mu_i}^+ \subset \ldots \subset \bigcup_{i=1}^n U_{\mu_i}^+= Y.$$
As the Grothendieck class is additive on closed subvarieties, we have
\begin{equation}\label{eq:1}
[Y]=\sum_{\mu\in I} [U_\mu^+].
\end{equation}
By property (1) of Lemma \ref{lem:properties-BB-decomp}, each $\mu\in I$ is a Zariski locally trivial affine fibration over $F_\mu$ whose fibre has dimension $N_\mu^+$, so we have
\begin{equation}\label{eq:2}
[U_\mu^+] = [\CC^{N_\mu^+}] [F_\mu]=\LL^{N_\mu^+}[F_\mu].
\end{equation}
Now, \eqref{eq:1} and \eqref{eq:2} prove the lemma.
\end{proof}

Given this lemma, we now have the following proposition that is a generalization of \cite[Corollary 1.3.3]{HRV15} to the motivic setup.

\begin{theorem}
\label{thm:semiprojectiveMotive}
Let $Y$ be a smooth complex semiprojective variety together with a surjective $\CC^*$-equivariant submersion $\pi: Y \to \CC$ covering the standard scaling action on $\CC$. Then in $\hat{K}(\VarC)$ we have
$$[\pi^{-1}(0)]=[\pi^{-1}(1)]\ \ \ \text{ and }\ \ \ [Y]=\LL [\pi^{-1}(0)].$$
\end{theorem}

\begin{proof}
Clearly, the fixed-point locus of $Y$ is concentrated in $\pi^{-1}(0)$. As $\pi^{-1}(0)$ is a smooth closed subspace of $Y$, then $\pi^{-1}(0)$ is also a smooth semiprojective variety. Moreover, since $\pi$ is $\CC^*$-equivariant, then the fixed points of the $\CC^*$-action on $Y$ are precisely those of $\pi^{-1}(0)$. Let
$$Y^{\CC^*}= \pi^{-1}(0)^{\CC^*} = \bigcup_{\mu\in I} F_\mu$$
be the decomposition of the fixed-point locus into connected components. As we have discussed above, both $Y$ and $\pi^{-1}(0)$ admit Bialynicki-Birula stratifications of the form
$$Y=\bigcup_{\mu\in I} \tilde{U}_\mu^+\ \ \ \text{ and }\ \ \  \pi^{-1}(0)=\bigcup_{\mu\in I} U_\mu^+,$$
where
$$\tilde{U}_\mu^+ = \left \{ p \in Y  \, \middle |\,  \lim_{t\to 0} t \cdot p \in F_\mu \right \}\ \ \ \text{ and }\ \ \ U_\mu^+ = \left \{ p \in \pi^{-1}(0) \,  \middle | \, \lim_{t\to 0} t \cdot p \in F_\mu \right \}.$$
Moreover, $\tilde{U}_\mu^+$ and $U_\mu^+$ are affine bundles over $F_\mu$ of rank $\tilde{N}_\mu^+$ and $N_\mu^+$ respectively. On the other hand, let
$$\tilde{U}_\mu^- = \left \{ p \in Y \,  \middle | \,  \lim_{t\to \infty} t \cdot p \in F_\mu \right \}\ \ \ \text{ and }\ \ \ U_\mu^- = \left \{ p \in \pi^{-1}(0) \,   \middle | \,  \lim_{t\to \infty} t \cdot p \in F_\mu \right \}.$$
Then $\tilde{U}_\mu^-$ and $U_\mu^-$ are also affine bundles over $F_\mu$, of rank $\tilde{N}_\mu^-$ and $N_\mu^-$ respectively, and we have the following decomposition of the tangent spaces $T_pY$ and $T_p(\pi^{-1}(0))$ at each $p\in F_\mu$,
$$T_pY = T_p (\tilde{U}_\mu^+|_p) \oplus T_p (\tilde{U}_\mu^-|_p) \oplus T_p F_\mu\ \ \ \text{ and }\ \ \ T_p(\pi^{-1}(0))= T_p (U_\mu^+|_p) \oplus T_p (U_\mu^-|_p) \oplus T_p F_\mu.$$
Using the smoothness assumption, this yields
\begin{equation}
\label{eq:semiprojectiveMotive1}
\dim Y = \tilde{N}_\mu^+ + \tilde{N}_\mu^- + \dim(F_\mu)\ \ \ \text{ and }\ \ \ \dim \pi^{-1}(0) = N_\mu^+ + N_\mu^- + \dim(F_\mu).
\end{equation}

Since $\CC^*$-action contracts the points of $Y$ to the $0$ fibre of $\pi^{-1}(0)$, then, for each $\mu\in I$, all the points points $p$ of $Y$ such that $\lim_{t\to \infty} t\cdot p \in F_\mu$ must lie in $\pi^{-1}(0)$. Thus $\tilde{U}_\mu^- = U_\mu^-$ and we have $\tilde{N}_\mu^-=N_\mu^-$. On the other hand, as $\pi:Y\to \CC$ is a submersion of smooth varieties, we have that $\dim Y=\dim \pi^{-1}(0)+1$, so from \eqref{eq:semiprojectiveMotive1} we conclude that for each $\mu$ we have
\begin{equation}\label{eq:tildeN=N+1}
\tilde{N}_\mu^+= N_\mu^+ +1.
\end{equation}
Thus, using the Bialynici-Birula decompositions of $Y$ and $\pi^{-1}(0)$, we can apply Lemma \ref{lemma:BBMotiveDecomposition} to decompose the corresponding motives as $[\pi^{-1}(0)] = \sum_{\mu\in I} \LL^{N_\mu^+} [F_\mu]$, and
\begin{equation}\label{eq:(2) holds}
[Y] = \sum_{\mu\in I} \LL^{\tilde{N}_\mu^+} [F_\mu]=\LL[\pi^{-1}(0)].
\end{equation}

On the other hand, the $\CC^*$-action yields an isomorphism $\pi^{-1}(\CC^*) \cong \pi^{-1}(1) \times \CC^*$, so we can write
$$[Y] = [\pi^{-1}(0)]+[\pi^{-1}(\CC^*)]=[\pi^{-1}(0)]+(\LL-1)[\pi^{-1}(1)],$$
and, by \eqref{eq:(2) holds}, this shows that $[\pi^{-1}(1)]=[\pi^{-1}(0)]$ in $\hat{K}(\VarC)$.
\end{proof}

\subsection{Invariance of the motive and $E$-polynomial with respect to the Lie algebroid structure}

We continue with our fixed base curve $X$, of genus $g\ge 2$. Now that we have established the required regularity conditions and the semiprojectivity of the $\SL$-Hodge moduli space, for $\SL=(L,[\cdot\,,\cdot],\delta)$ of rank $1$ and degree less than $2-2g$, we can address the invariance of the motive with respect to the algebroid structure of $\SL$ by keeping $L$ fixed and varying $\lambda$ in $\CC$. Hence the variation on the Lie algebroid structure we are considering is the one given by \eqref{eq:Llambda}. By \eqref{eq:isomfnonzeroibres}, this is clearly true if one changes the Lie algebroid structure by varying from $\lambda\in\CC^*$ to $\lambda'\in\CC^*$, so the main point is that the motivic class remains unchanged when we go to the trivial algebroid structure, thus $\lambda=0$. This is one of the contents of the next theorem, which follows directly from Theorems \ref{thm:semiprojectiveHodge} and \ref{thm:semiprojectiveMotive}.

\begin{theorem}
\label{thm:equalMotiveHodge}
 Let $\SL=(L,[\cdot\,,\cdot],\delta)$ be a Lie algebroid on $X$ such that $L$ is a line bundle with $\deg(L)<2-2g$. If $r$ and $d$ are coprime, then the following equalities hold in $\hat{K}(\VarC)$,
$$[\SM_{\SL}(r,d)] = [\SM_{L^{-1}}^{\Dol}(r,d)],\ \ \ \ \ [\SM_{\SL}^{\Hod}(r,d)] = \LL[\SM_{L^{-1}}^{\Dol}(r,d)].$$
In particular,
$$E(\SM_{\SL}(r,d)) = E(\SM_{L^{-1}}^{\Dol}(r,d)),\ \ \ \ \  E(\SM_{\SL}^{\Hod}(r,d)) = uv E(\SM_{L^{-1}}^{\Dol}(r,d)).$$
Moreover, we have an isomorphism of Hodge structures $$H^\bullet(\SM_{\SL}(r,d)) \cong H^\bullet(\SM_{L^{-1}}^{\Dol}(r,d))$$ and both $\SM_{\SL}(r,d)$ and $\SM_{\SL}^{\Hod}(r,d)$ have pure mixed Hodge structures.
\end{theorem}
\begin{proof}
By Theorem \ref{thm:semiprojectiveHodge}, the moduli space $\SM_{\SL}^{\Hod}(r,d)$ is a smooth semiprojective variety for the $\CC^*$-action \eqref{eq:C*-action-LHodgemod}. Moreover, the map $\pi$ from \eqref{eq:project->lambda} is a surjective $\CC^*$-equivariant submersion covering the standard $\CC^*$-action on $\CC$. Then Theorem \ref{thm:semiprojectiveMotive} gives the desired motivic equalities,
$$[\SM_{L^{-1}}^{\Dol}(r,d)]=[\pi^{-1}(0)]=[\pi^{-1}(1)]=[\SM_{\SL}(r,d)]\ \ \ \text{ and }\ \ \ [\SM_{\SL}^{\Hod}(r,d)] = \LL [\pi^{-1}(0)]=\LL[\SM_{L^{-1}}^{\Dol}(r,d)],$$
which yield the corresponding equalities of $E$-polynomials,
$$E(\SM_{\SL}(r,d)) = E(\SM_{L^{-1}}^{\Dol}(r,d))\ \ \ \text{ and }\ \ \ E(\SM_{\SL}^{\Hod}(r,d)) = uv E(\SM_{L^{-1}}^{\Dol}(r,d)).$$
Moreover, by \cite[Corollary 1.3.3]{HRV15}, the fibres $\SM_{L^{-1}}^{\Dol}(r,d)=\pi^{-1}(0)$ and $\SM_{\SL}(r,d)=\pi^{-1}(1)$ have isomorphic cohomology supporting pure mixed Hodge structures. Finally, as $\SM_{\SL}^{\Hod}(r,d)$ is also smooth and semiprojective, its cohomology is also pure by \cite[Corollary 1.3.2]{HRV15}.
\end{proof}

\section{Motives of moduli spaces of twisted Higgs bundles}
\label{section:HiggsPolynomials}

One of the main goals of this paper is to prove that given any two rank $1$ Lie algebroids $\SL$ and $\SL'$, over the genus $g$ curve $X$, such that $\deg(\SL)=\deg(\SL')<2-2g$, we have an equality of motivic classes
\begin{equation}\label{eq:motiv-inv-Liealg}
[\SM_{\SL}(r,d)]=[\SM_{\SL'}(r,d)]\in \hat{K}(\VarC),
\end{equation}
where $d$ is coprime with $r$. This will be achieved in Theorem \ref{thm:equalMotive}.

Theorem \ref{thm:equalMotiveHodge} already says that $$[\SM_{\SL}(r,d)]=[\SM_{L^{-1}}^{\Dol}(r,d)],$$ therefore the motive of the moduli space of $\SL$-connections is invariant with respect to the Lie algebroid structure on $L$. We will now prove that $[\SM_{L^{-1}}^{\Dol}(r,d)]$ is also invariant with respect to the choice of the twisting line bundle $L^{-1}$ as long as we fix its degree. Thus, we prove that for any pair of line bundles $L,L'$ with $\deg(L)=\deg(L')<2-2g$ and any $d$ coprime with $r$ we have
$$[\SM_{L^{-1}}^{\Dol}(r,d)]=[\SM_{(L')^{-1}}^{\Dol}(r,d)].$$
Combining the two results yields \eqref{eq:motiv-inv-Liealg}.

In particular, such equalities imply by \eqref{E-poly} that the $E$-polynomials of these moduli spaces are also equal. Then, we can go further by using results by Maulik--Shen \cite{MS20-chiindependence} or of Groechening, Wyss and Ziegler \cite{GWZ20} which imply that for every line bundle $N\to X$ such that $\deg(N)>2g-2$, and every $d$ and $d'$ coprime with $r$, we have
$$E(\SM_N^{\Dol}(r,d))=E(\SM_N^{\Dol}(r,d')).$$
Hence, this implies, together with the above equality of motivic classes, that
$$E(\SM_{\SL}(r,d))=E(\SM_{\SL'}(r,d')),$$
for every Lie algebroids $\SL$ and $\SL'$ of degree less than $2-2g$ and every $d$ and $d'$ coprime to $r$.

We start by considering special types of Higgs bundles, known as variations of Hodge structure.

\subsection{Variations of Hodge structure and chains}

Let $L$ be a line bundle over the curve $X$. An \emph{$L$-twisted variation of Hodge structure} of \emph{type} $\overline{r}=(r_1,\ldots,r_k)$ and \emph{multidegree} $\overline{d}=(d_1,\ldots,d_k)$ is an $L$-twisted Higgs bundle $(E,\varphi)$ of the form
\begin{equation}\label{eq:VHS}
(E_\bullet,\varphi_\bullet)= \left( \bigoplus_{i=1}^k E_i , \smtrx{0 & 0 & \cdots & 0 & 0\\
\varphi_1 & 0 & \cdots & 0 & 0\\
0 & \varphi_2 & \cdots & 0 & 0\\
\vdots & \vdots & \ddots & \ddots & \vdots \\
0 & 0 & \cdots & \varphi_{k-1} & 0}\right),
\end{equation}
where $E_i$ are vector bundles on $X$, with $\rk(E_i)=r_i$ and $\deg(E_i)=d_i$, and $\varphi_i: E_i \to E_{i+1}\otimes L$ for each $i=1,\ldots,k$, with $\varphi_k=0$.

Let $r=\sum_{i=1}^k r_i$ and $d=\sum_{i=1}^k d_i$ and denote by 
$$\VHS_L(\overline{r},\overline{d})\subset \SM_L^{\Dol}(r,d)$$ the subscheme of the moduli space of $L$-twisted Higgs bundles corresponding to semistable variations of Hodge structure of type $\overline{r}$ and multi-degree $\overline{d}$.

On the other hand, recall that an algebraic \emph{chain} on $X$ is a quiver bundle of type A, hence a sequence of algebraic vector bundles $(E_1,\ldots,E_k)$, together with maps $\varphi_i:E_i\to E_{i+1}$. We denote a chain by the symbol $(\tilde E_\bullet,\tilde\varphi_\bullet)$.

Given real numbers $\alpha=(\alpha_1,\ldots,\alpha_k)$, we define the \emph{$\alpha$-degree} of the chain $(\tilde E_i,\tilde \varphi_i)$ as
\begin{equation}\label{eq:alpha-deg}
\deg_\alpha(\tilde E_\bullet,\tilde \varphi_\bullet)=\sum_{i=1}^k (\deg(E_i)+\rk(E_i)\alpha_i)
\end{equation}
and the \emph{$\alpha$-slope} as
$$\mu_\alpha(\tilde E_\bullet,\tilde \varphi_\bullet)=\frac{\deg_\alpha(\tilde E_\bullet,\tilde \varphi_\bullet)}{\sum_{i=1}^k \rk(E_i)}.$$
We say that $(\tilde E_\bullet,\tilde \varphi_\bullet)$ is of type $\overline{r}=(r_1,\ldots,r_k)$ if $\rk(E_i)=r_i$ for each $i=1,\ldots, k$ and we call $\overline{d}=(\deg(E_1),\ldots,\deg(E_k))$ its \emph{multidegree}.

A \emph{subchain} $(\tilde F_\bullet,\tilde \varphi_\bullet)$ of $(\tilde E_\bullet,\tilde \varphi_\bullet)$ is a collection $(F_1,\ldots,F_k)$ of subbundles of $(E_1,\ldots, E_k)$, i.e.~$F_i\subset E_i$, such that $\varphi_i(F_i)\subset F_{i+1}$, so that $(\tilde F_\bullet,\tilde \varphi_\bullet|_{F_\bullet})$ is itself a chain.
An algebraic chain $(\tilde E_\bullet,\tilde \varphi_\bullet)$ is \emph{(semi)stable} if for any subchain $(\tilde F_\bullet,\tilde \varphi_\bullet|_{F_\bullet})\subset (\tilde E_\bullet,\tilde \varphi_\bullet)$, we have
$$\mu_{\alpha}(\tilde F_\bullet,\tilde \varphi_\bullet|_{F_\bullet}) < \mu_{\alpha}(\tilde E_\bullet,\tilde \varphi_\bullet) \quad (\text{resp. }\le).$$
Denote by $$\HC^\alpha(\overline{r},\overline{d})$$ the moduli space of $\alpha$-semistable algebraic chains on $X$ of type $\overline{r}$ and multidegree $\overline{d}$.

Now, given a variation of Hodge structure $(E_\bullet,\varphi_\bullet)$ of type $\overline{r}=(r_1,\ldots,r_k)$ and multidegree $\overline{d}=(d_1,\ldots,d_k)$, an algebraic chain $(\tilde{E}_\bullet, \tilde{\varphi}_\bullet)$ can be constructed as follows. Take
$$\tilde{E}_i= E_i\otimes L^{i-k}.$$
Then $\varphi_i$ induces a map
$$\tilde{\varphi_i}=\varphi_i\otimes \id_{L^{i-k}} : \tilde{E_i}\longrightarrow \tilde{E}_{i+1},$$
thus $(\tilde{E}_\bullet, \tilde{\varphi}_\bullet)$ is a chain of type $\overline{r}$ and multidegree $\overline{d}_L=(d_1+r_1(1-k)\deg(L),\ldots,d_k)$. This construction is reversible, giving a variation of Hodge structure from an algebraic chain, and hence giving a bijection between these two kinds of objects.

It turns out that their (semi)stability conditions also match, if one chooses a particular set of real numbers $\alpha$. Indeed, if $\alpha_L=((k-1)\deg(L),\ldots, \deg(L),0)$, then the $\alpha_L$-degree \eqref{eq:alpha-deg} of any chain $(\tilde{E}_\bullet,\tilde{\varphi}_\bullet)$, is 
$$\deg_{\alpha_L}(\tilde{E}_\bullet,\tilde{\varphi}_\bullet)=\sum_{i=1}^k \left(\deg(E_i \otimes L^{i-k}) + \rk(E_i)(k-i)\deg(L) \right) = \sum_{i=1}^k \deg(E_i)$$
so
$$\mu_\alpha(\tilde{E}_\bullet,\tilde{\varphi}_\bullet) =\mu(E_\bullet,\varphi_\bullet),$$ 
where $(E_\bullet,\varphi_\bullet)$ is the corresponding variation of Hodge structure.

The proof of the next lemma follows by the exact same argument as in \cite[Proposition 3.5]{ACGP01}, by replacing the canonical line bundle $K_X$ by $L$.

\begin{proposition}
A variation of Hodge structure $(E_\bullet,\varphi_\bullet)$ is (semi)stable (as an $L$-twisted Higgs bundle) if and only if for every choice of subbundles $F_i\subset E_i$ with $\varphi_i(F_i)\subset F_{i+1}$ we have
$$\mu(F_\bullet,\varphi_\bullet|_{F_\bullet}) < \mu(E_\bullet,\varphi_\bullet) \quad (\text{resp. }\le).$$
Hence $(E_\bullet,\varphi_\bullet)$ is (semi)stable if and only if the corresponding chain $(\tilde E_\bullet,\tilde \varphi_\bullet)$ is $\alpha_L$-(semi)stable.
\end{proposition}

So the following corollary is immediate.

\begin{corollary}
\label{cor:VHSChains}
Fix an algebraic line bundle $L$ over $X$. Let  $\overline{r}=(r_1,\ldots,r_k)$, $\overline{d}=(d_1,\ldots,d_k)$ and $\overline{d}_L=(d_1+r_1(1-k)\deg(L),d_2+r_2(2-k)\deg(L),\ldots,d_k)$. The previously described correspondence between chains and $L$-twisted variations of Hodge structure induces an isomorphism,
$$\VHS_L(\overline{r},\overline{d}) \cong \HC^{\alpha_L}(\overline{r}, \overline{d}_L),$$
for $\alpha_L=((k-1)\deg(L),\ldots, \deg(L),0)$.
\end{corollary}

\subsection{Independence of the motives of Higgs moduli on the twisting line bundle}

We will use Corollary \ref{cor:VHSChains} to show that the motivic class of the moduli spaces of $L$-twisted Higgs bundles for coprime rank and degree only depend on the degree of the twisting line bundle $L$, whenever it is big enough.

Recall from \eqref{eq:C*-actHiggs} that the moduli space $\SM_L^{\Dol}(r,d)$ has a natural $\CC^*$-action given by scaling the Higgs field.
 Suppose that $d$ is coprime with $r\ge 2$ and that $L$ is a line bundle with $\deg(L)>2g-2$. Then by Proposition \ref{prop:semiprojectiveHiggs}, the moduli space $\SM_L^{\Dol}(r,d)$ is a smooth semiprojective variety. Accordingly, it admits a Bialynicki-Birula stratification
$$\SM_L^{\Dol}(r,d)= \bigcup_{\mu\in I} U_\mu^+,$$
which, by Lemma \ref{lemma:BBMotiveDecomposition}, induces the decomposition 
\begin{equation}\label{eq:motiv-decomp-Higgs}
[\SM_L^{\Dol}(r,d)]=\sum_{\mu\in I} \LL^{N_\mu^+}[F_\mu]
\end{equation}
of its motivic class, where $N_\mu^+=\dim(U_\mu^+)-\dim(F_\mu)$ is the rank of the affine bundle $U_\mu^+\to F_\mu$, corresponding to those Higgs bundles $(E,\varphi)\in \SM_L^{\Dol}(r,d)$ such that $\lim_{t\to 0}(E,t\varphi)$ lies in the $\CC^*$-fixed point set $F_\mu$. The characterization of the fixed points under the $\CC^*$-action carried out by  Simpson in \cite[\S 4]{SimpsonHiggs} also applies to the $L$-twisted case, obtaining the following lemma.

\begin{lemma}
\label{lemma:fixedPointsVHS}
Let $(E,\varphi)$ be any $L$-twisted Higgs bundle such that $(E,\varphi) \cong (E,t\varphi)$ for some $t\in \CC^*$ which is not a root of unity. Then $E$ has the structure of an $L$-twisted variation of Hodge structure \eqref{eq:VHS}. Reciprocally, any $L$-twisted variation of Hodge structure is a fixed point of the $\CC^*$-action.
\end{lemma}

Given any multirank $\overline{r}=(r_1,\ldots,r_k)$ and multidegree $\overline{d}=(d_1,\ldots,d_k)$, define 
\begin{equation}\label{eq:notation}
|\overline{r}|=\sum_{j=1}^k r_j,\ \ \ |\overline{d}|=\sum_{j=1}^k d_j\ \ \ \text{ and }\ \ \ \Delta_L=\{(\overline{r},\overline{d}) \, | \,  \VHS_L(\overline{r},\overline{d})\ne \emptyset\}.
\end{equation}
 
The previous lemma says that the semistable $\CC^*$-fixed points are precisely  those in $\VHS_{L}(\overline{r},\overline{d})\subset \SM_{L}^{\Dol}(r,d)$ for each suitable choice of $\overline{r}$ and $\overline{d}$. Thus, we rewrite \eqref{eq:motiv-decomp-Higgs} as
\begin{equation}
\label{eq:EPolHiggs}
[\SM_L^{\Dol}(r,d)]= \sum_{\begin{array}{c}
{\scriptstyle (\overline{r},\overline{d})\in\Delta_L}\\
{\scriptstyle |\overline{r}|=r,\,|\overline{d}|=d}
\end{array}} \LL^{N_{L,\overline{r},\overline{d}}^+} [\VHS_L(\overline{r},\overline{d})],
\end{equation}
where $N_{L,\overline{r},\overline{d}}^+$ is the notation for $N_\mu^+$ in this case. We will also use the notations $N_{L,\overline{r},\overline{d}}^-$ for $N_\mu^-$ and $N_{L,\overline{r},\overline{d}}^0$ for $N_\mu^0$; see \eqref{eq:notat-dim-tangent}.

Next, we will focus on the computation and invariance with respect to $L$ of the ranks $N_{L,\overline{r},\overline{d}}^{\pm}$ in \eqref{eq:EPolHiggs}.
Suppose  that $r$ and $d$ are coprime, so that the moduli space $\SM_L^{\Dol}(r,d)$ is smooth. Following \cite{HitchinHiggs,Kirwan84} and working as in \cite{BGL11}, we will proceed by analysing the Bialynicki-Birula stratification from a Morse-theoretic point of view. The moduli space $\SM_{L}^{\Dol}(r,d)$ has a K\"ahler structure which is preserved by the action of $S^1\subset \CC^*$.  Therefore, the $\CC^*$-action induces a Hamiltonian action of $S^1$, with moment map
$$\mu: \SM_L^{\Dol}(r,d) \longrightarrow\RR,\ \ \ \mu(E,\varphi)=\frac{1}{2}\lVert \varphi \rVert^2,$$
where the $L_2$-norm is given with respect to the (harmonic) metric solving the Hitchin equations corresponding to the stable Higgs bundle $(E,\varphi)$ under the Hitchin-Kobayashi correspondence; cf.~\cite{HitchinHiggs}.

By \cite{F59} the map $\mu$ becomes a perfect Morse-Bott function in $\SM_L^{\Dol}(r,d)$ and we have the following lemma.

\begin{lemma}
\label{lemma:morseIndex}
Suppose that $g\ge 2$, that $r$ and $d$ are coprime and that $\deg(L)>2g-2$. Let $\overline{r}=(r_1,\ldots, r_k)$ and $\overline{d}=(d_1,\ldots, d_k)$ be such that $r=r_1+\cdots+r_k$,  $d=d_1+\cdots+d_k$ and $\VHS_L(\overline{r},\overline{d})$ is non-empty. Then $\VHS_L(\overline{r},\overline{d})$ is a component of the critical point set of $\mu$ and if $M_{L,\overline{r},\overline{d}}$ is its Morse index, then
$M_{L,\overline{r},\overline{d}}=2N_{L,\overline{r},\overline{d}}^-$.
In particular,
$$N_{L,\overline{r},\overline{d}}^++N_{L,\overline{r},\overline{d}}^0+M_{L,\overline{r},\overline{d}}/2=\dim(\SM_{L}^{\Dol}(r,d))=1+r^2\deg(L).$$
\end{lemma}

\begin{proof}
$\SM_L^{\Dol}(r,d)$ is smooth by Lemma \ref{lemma:smoothModuliHiggs}. Then, by \cite[Theorem 6.18, Example 9.4 and Corollary 13.2]{Kirwan84}, we conclude that, for each $(\overline{r},\overline{d})$ in the given conditions, the component $\VHS_L(\overline{r},\overline{d})$ of the fixed-point locus $\SM_L^{\Dol}(r,d)^\CC$ is a component the critical point set of $\mu$ and that the affine bundle $U_{\overline{r},\overline{d}}^-\to \VHS_L(\overline{r},\overline{d})$ coincides with the downwards Morse flow of $\mu$. Then, for each point $p\in \VHS_L(\overline{r},\overline{d})$, we have
$$N_{L,\overline{r},\overline{d}}^- = \dim \left ( T_p \left( U_{\overline{r},\overline{d}}^-|_p \right)\right) = \frac{1}{2} \dim_\RR \left (U_{\overline{r},\overline{d}}^-|_p \right) = \frac{1}{2} M_{L,\overline{r},\overline{d}}.$$
The last statement follows from \eqref{eq:upward+downard+fix=dim} and again Lemma \ref{lemma:smoothModuliHiggs}.
\end{proof}

\begin{lemma}
\label{lemma:equalMorseIndex}
Let $L$ and $L'$ be line bundles on $X$ such that $\deg(L)=\deg(L')>2g-2$. Suppose that $g\ge 2$ and that $r$ and $d$ are coprime. Then the Morse index $M_{L,\overline{r},\overline{d}}$ of $\VHS_L(\overline{r},\overline{d})\subset \SM_L^{\Dol}(r,d)$ is the same as the Morse index $M_{L',\overline{r},\overline{d}}$ of $\VHS_{L'}(\overline{r},\overline{d})\subset \SM_{L'}^{\Dol}(r,d)$.
\end{lemma}

\begin{proof}
Since $r$ and $d$ are coprime, every semistable $L$-twisted Higgs bundle $(E,\varphi)\in \SM_L^{\Dol}(r,d)$ is stable, hence simple. So by \cite[Theorem 2.3]{BR94} (see also Remark 2.8(ii) of loc.~cit.) or by \cite[Theorem 47]{Tor11}, the tangent space to the moduli space $\SM_L^{\Dol}(r,d)$ at $(E,\varphi)$ is isomorphic to $\HH^1(X, C^\bullet(E,\varphi))$, where $C^\bullet(E,\varphi)$ is the following complex
$$C^\bullet(E,\varphi) \, : \, \End(E) \stackrel{[-,\varphi]}{\longrightarrow} \End(E)\otimes L.$$
At a variation of Hodge structure $(E_\bullet,\varphi_\bullet)\in \VHS_L(\overline{r},\overline{d})$, this deformation complex decomposes as
\begin{equation}\label{eq:complx-decomp}
C^\bullet(E,\varphi) = \bigoplus_{l=-k+1}^{k-1} C_l^\bullet(E,\varphi)
\end{equation}
where
$$C_l^\bullet(E,\varphi) : \bigoplus_{j-i=l} \Hom(E_i,E_j) \stackrel{[-,\varphi]}{\longrightarrow} \bigoplus_{j-i=l+1} \Hom(E_i,E_j)\otimes L$$
and, thus, the tangent space decomposes as
$$\HH^1(C^\bullet(E,\varphi)) = \bigoplus_{l=-k+1}^{k-1} \HH^1(C^\bullet_l(E_\bullet,\varphi_\bullet)).$$

Then, the computations in \cite[\S 5]{BGL11} show that if $r$ and $d$ are coprime and $\deg(L)>2g-2$, then the Morse index of $\VHS_{L}(\overline{r},\overline{d})\subset \SM_L^{\Dol}(r,d)$ is
\begin{equation}\label{eq:Morseindex}
M_{L,\overline{r},\overline{d}} = 2\sum_{l=1}^{k-1}\dim(\HH^1(C^\bullet_l(E_\bullet,\varphi_\bullet))) = - 2\sum_{l=1}^{k-1} \chi(C^\bullet_l(E_\bullet,\varphi_\bullet)),
\end{equation}
where, for each $l=1,\ldots,k-1$,
\begin{equation}\label{eq:Eulerchar}
	\begin{split}
	-\chi(C^\bullet_l(E_\bullet,\varphi_\bullet))&= \sum_{i=1}^{k-l-1} \chi (\Hom(E_i,E_{i+l+1}) \otimes L) - \sum_{i=1}^{k-l} \chi(\Hom(E_i,E_{i+l}))\\
&=\sum_{i=1}^{k-l-1} (-r_{i+l+1} d_i + r_i d_{i+l+1}+r_ir_{i+l} \deg(L)+ r_ir_{i+l+1}(1-g))\\
 &\ \ \ - \sum_{i=1}^{k-l} (-r_{i+l} d_i + r_i d_{i+l} + r_ir_{i+l}(1-g)).
	\end{split}
\end{equation}
Thus $\chi(C^\bullet_l(E_\bullet,\varphi_\bullet))$ depends on the degree of $L$, but not on $L$ itself, and hence the same is true for the Morse index.
\end{proof}

\begin{theorem}
\label{thm:equalMotiveHiggs}
Let $X$ be a smooth complex projective curve of genus $g\geq 2$. Let $L$ and $L'$ be line bundles over $X$ such that $\deg(L)=\deg(L')>2g-2$. Assume that the rank $r$ and degree $d$ are coprime. Then the motivic classes of the corresponding moduli spaces
$[\SM_L^{\Dol}(r,d)]$ and $[\SM_{L'}^{\Dol}(r,d)]$ are equal in $K(\VarC)$. Moreover, if $d'$ is any integer coprime with $r$, then $E(\SM_L^{\Dol}(r,d))=E(\SM_{L'}^{\Dol}(r,d'))$.
Finally, if $L=L'=K_X(D)$ for some effective divisor $D$, then there is an actual isomorphism of pure mixed Hodge structures
$H^\bullet(\SM_L^{\Dol}(r,d)) \cong  H^\bullet(\SM_{L'}^{\Dol}(r,d'))$.
\end{theorem}

\begin{proof}
By Corollary \ref{cor:VHSChains}, for each $k=1,\ldots,r$ and each $\overline{r}=(r_1,\ldots,r_k)$ and $\overline{d}=(d_1,\ldots,d_k)$ with $|\overline{r}|=r$ and $|\overline{d}|=d$ (recall \eqref{eq:notation}), we have 
$$\VHS_L(\overline{r},\overline{d}) \cong \HC^{\alpha_L}(\overline{r}, \overline{d}_L),$$
where $\alpha_L=((k-1)\deg(L),\ldots, \deg(L),0)$ and $\overline{d}_L=(d_{L,i})$ with $d_{L,i}=d_i+r_i(i-k)\deg(L)$. As $\deg(L)=\deg(L')$ we have $\alpha_L=\alpha_{L'}$ and $\overline{d}_L=\overline{d}_L'$, so we obtain an isomorphism
\begin{equation}\label{eq:isomVHS-LL'}
\VHS_L(\overline{r},\overline{d}) \cong \HC^{\alpha_L}(\overline{r}, \overline{d}_L) = \HC^{\alpha_{L'}}(\overline{r}, \overline{d}_{L'}) \cong \VHS_{L'}(\overline{r},\overline{d}).
\end{equation}
In particular, we have $\Delta_L=\Delta_{L'}$; cf.~\eqref{eq:notation}. 

On the other hand, as $\SM_L^{\Dol}(r,d)$ is smooth of dimension $1+r^2\deg(L)$, then for each $(\overline{r},\overline{d})\in \Delta_L=\Delta_{L'}$, we have, using Lemma \ref{lemma:morseIndex},
\begin{equation}\label{eq:formula-N+}
N_{L,\overline{r},\overline{d}}^+ = 1+r^2\deg(L) - \dim(\VHS_L(\overline{r},\overline{d}))- M_{L,\overline{r},\overline{d}}/2.
\end{equation}
By Lemma \ref{lemma:equalMorseIndex} and \eqref{eq:isomVHS-LL'} we have $N_{L,\overline{r},\overline{d}}^+ = N_{L',\overline{r},\overline{d}}^+$. 

Therefore, using \eqref{eq:EPolHiggs}, we conclude that
$$
[\SM_L^{\Dol}(r,d)] = \sum_{k=1}^r\sum_{\begin{array}{c}
{\scriptstyle (\overline{r},\overline{d})\in\Delta_L}\\
{\scriptstyle |\overline{r}|=r,\,|\overline{d}|=d}
\end{array}} \LL^{N_{L,\overline{r},\overline{d}}^+} [\VHS_L(\overline{r},\overline{d})]\\
= \sum_{k=1}^r\sum_{\begin{array}{c}
{\scriptstyle (\overline{r},\overline{d})\in\Delta_{L'}}\\
{\scriptstyle |\overline{r}|=r,\,|\overline{d}|=d}
\end{array}} \LL^{N_{L',\overline{r},\overline{d}}^+} [\VHS_{L'}(\overline{r},\overline{d})]=[\SM_{L'}^{\Dol}(r,d)].$$

Then 
\begin{equation}\label{eq:equalEpol-L-twistedsamedegreed}
E(\SM_L^{\Dol}(r,d))=E(\SM_{L'}^{\Dol}(r,d))
\end{equation}
 is direct from \eqref{E-poly}.

Take now any $d'$ also coprime with $r$. To prove that $\SM_L^{\Dol}(r,d)$ and $\SM_{L'}^{\Dol}(r,d')$ have the same $E$-polynomial, we proceed as follows. By \cite[Theorem 0.1]{MS20-chiindependence} (see also \cite[Theorem 7.15]{GWZ20}), 
\begin{equation}\label{eq:equalEpol-K(D)twisted}
E(\SM_{K_X(D)}^{\Dol}(r,d))=E(\SM_{K_X(D)}^{\Dol}(r,d')),
\end{equation}
where $K_X$ denotes the canonical bundle of $X$ and $D$ is an effective divisor.
To prove it for any twisting line bundles $L,L'$ of the same degree (greater than $2g-2$), take any point $x_0\in X$ and let $m=\deg(L)+2-2g>0$. Then $\deg(K_X(mx_0))=\deg(L)= \deg(L')$ so, applying \eqref{eq:equalEpol-L-twistedsamedegreed} and \eqref{eq:equalEpol-K(D)twisted}, yields
$$E(\SM_{L}^{\Dol}(r,d))=E(\SM_{K_X(mx_0)}^{\Dol}(r,d))=E(\SM_{K_X(mx_0)}^{\Dol}(r,d'))=E(\SM_{L'}^{\Dol}(r,d')),$$
as claimed.

Finally, if $L=L'=K_X(D)$ with $D$ effective, the isomorphism $H^\bullet(\SM_L^{\Dol}(r,d)) \cong  H^\bullet(\SM_{L'}^{\Dol}(r,d'))$ of Hodge structures follows immediately from \cite[Theorem 0.1]{MS20-chiindependence}.
\end{proof}

\section{Proof of the main results}
\label{section:mainResults}
\subsection{The complete motivic invariance theorem}

We can combine all the previous invariance results to prove one of our main theorems. Given a Lie algebroid, $\SL$ on the curve $X$, recall that $\SM_{\SL}(r,d)$ denotes the moduli space of semistable integrable $\SL$-connections of rank $r$ and degree $d$ or, equivalently, of semistable $\Lambda_\SL$-modules, where $\Lambda_\SL$ is the split almost polynomial sheaf of rings of differential operators associated to $\SL$, under the equivalence provided by Theorem \ref{thm:equiv:Liealge-sapsrdo}. If $\rk(\SL)=1$, then every $\SL$-connection is automatically integrable, so in that case $\SM_{\SL}(r,d)$ is the moduli space of all semistable $\SL$-connections of rank $r$ and degree $d$.  

\begin{theorem}
\label{thm:equalMotive}
Let $X$ be a smooth projective curve of genus $g\ge 2$ and let $\SL$ and $\SL'$ be any two Lie algebroids on $X$ such that $\rk(\SL)=\rk(\SL')=1$ and $\deg(\SL)=\deg(\SL')<2-2g$. Suppose that $r$ and $d$ are coprime. Then $[\SM_{\SL}(r,d)] = [\SM_{\SL'}(r,d)]$ in $\hat{K}(\VarC)$. Moreover, if $d'$ is any integer coprime with $r$, then $E(\SM_{\SL}(r,d)) = E(\SM_{\SL'}(r,d'))$. Finally, if $L=L'=T_X(-D)$ for some effective divisor $D$, then there is an actual isomorphism of pure mixed Hodge structures $H^\bullet(\SM_{\SL}(r,d)) \cong  H^\bullet(\SM_{\SL'}(r,d'))$.
\end{theorem}

\begin{proof}
As announced at the start of this section, this theorem follows directly from Theorems \ref{thm:equalMotiveHodge} and \ref{thm:equalMotiveHiggs}, and by \eqref{E-poly}.

By Theorem \ref{thm:equalMotiveHodge}, the motive of the moduli space is independent of the Lie algebroid structure. Thus, if $L$ and $L'$ are the underlying line bundles to $\SL$ and $\SL'$ respectively and $d$ and $d'$ are coprime with $r$, then
$$[\SM_{\SL}(r,d)]=[\SM_{L^{-1}}^{\Dol}(r,d)] \quad \quad \text{and} \quad \quad  [\SM_{\SL'}(r,d')]=[\SM_{(L')^{-1}}^{\Dol}(r,d')].$$
On the other hand, Theorem \ref{thm:equalMotiveHiggs} implies that the motive of a twisted Higgs moduli space does not depend on the twisting bundle, but rather on the underlying degree. Thus, if $d=d'$ then
$$[\SM_{\SL}(r,d)]=[\SM_{L^{-1}}^{\Dol}(r,d)]=[\SM_{(L')^{-1}}^{\Dol}(r,d)]=[\SM_{\SL'}(r,d)]$$
If we restrict ourselves to $E$ polynomials, then we can also vary the degree $d$, as given $d$ and $d'$ coprime with $r$, Theorem \ref{thm:equalMotiveHodge} implies
$$E\left(\SM_{\SL}(r,d)\right)=E\left(\SM_{L^{-1}}^{\Dol}(r,d)\right)=E\left(\SM_{(L')^{-1}}^{\Dol}(r,d')\right)=E\left(\SM_{\SL'}(r,d')\right)$$
Finally, if $L=L'=T_X(-D)$ for some effective divisor $D$, then Theorems \ref{thm:equalMotiveHodge} and \ref{thm:equalMotiveHiggs} actually give isomorphisms between the following pure mixed Hodge structures:
$$H^\bullet\left(\SM_{\SL}(r,d)\right) \cong H^\bullet\left(\SM_{K_X(D)}^{\Dol}(r,d)\right)\cong H^\bullet\left(\SM_{K_X(D)}^{\Dol}(r,d')\right) \cong  H^\bullet\left(\SM_{\SL'}(r,d')\right)$$
\end{proof}

Hence to compute the motivic class or the $E$-polynomial of $\SM_{\SL}(r,d)$, it is enough to do it for the moduli space $\SM_{K_X(D)}^{\Dol}(r,1)$ of $K_X(D)$-twisted Higgs bundles of rank $r$ and degree $1$, for some divisor $D$ of the appropriate (positive) degree.

\begin{remark}
Notice that, as a consequence of the discussion held in Section \ref{subsection:classification} and, in particular, in the proof of Corollary \ref{cor:classificationLieAlgebroidsRank1Meromorphic}, if either $\SL$ or $\SL'$ are nontrivial Lie algebroids on a line bundle $L$, then one of their anchors $\delta,\delta'\in H^0(L^{-1}\otimes T_X)$ is nonzero and, therefore, that anchor induces an isomorphism $L\cong T_X(-D)$ for an appropriate effective divisor $D$. Thus, the condition $L=L'=T(-D)$ in the Theorem is always satisfied if $\SL$ and $\SL'$ share the same line bundle $L$ and one of them is a nontrivial Lie algebroid structure.
\end{remark}

\subsection{Results on Chow motives and Voevodsky motives}
\label{sec:ChowVoev}

Theorem \ref{thm:equalMotive} shows that, under the stated conditions, there is an equality of motives 
$$[\SM_{\SL}(r,d)] = [\SM_{\SL'}(r,d)]\in \hat{K} (\VarC)$$
in the (completed) Grothendieck ring of varieties.

Nevertheless, the techniques that we use to prove this equality (namely, semiprojectivity of the $\SL$-Hodge moduli space and the exposed relations between the Bialynicki-Birula decompositions of the corresponding moduli spaces of twisted Higgs bundles) also allow us to obtain isomorphisms for other types of invariants. Given a complex scheme $X$ and a ring $R$, let us consider the following.
\begin{itemize}
\item Let $M(X)\in \mathrm{DM}^{\mathrm{eff}}(\CC,R)$ denote the Voevodsky motive of $X$, where $\mathrm{DM}^{\mathrm{eff}}(\CC,R)$ is the category of effective geometric motives as defined by Voevodsky in \cite{V00}.
\item Let $h(X)\in \mathrm{Chow}^{\mathrm{eff}}(\CC,R)$ be the Chow motive of $X$, where $\mathrm{Chow}^{\mathrm{eff}}(\CC,R)$ is the category of effective Chow motives; see for example \cite{M68,S94,dB01}.
\item Let $CH^\bullet(X,R)$ denote the Chow ring of $X$ with coefficients in $R$.
\end{itemize}
Moreover, recall that we say that $X$ has a pure Voevodsky motive if $M(X)$ belongs to the heart of $\mathrm{DM}^{\mathrm{eff}}(\CC,R)$ which is equivalent to $\mathrm{Chow}^{\mathrm{eff}}(\CC,R)$ through Voevodsky's embedding (c.f. \cite[Section 6.3]{HL21}).

\begin{theorem}
\label{thm:equalMotive-Chow-Voevodsky}
Let $X$ be a smooth projective curve of genus $g\ge 2$ and let $\SL$ and $\SL'$ be any two Lie algebroids on $X$ such that $\rk(\SL)=\rk(\SL')=1$ and $\deg(\SL)=\deg(\SL')<2-2g$. Suppose that $r$ and $d$ are coprime. Then, for every ring $R$, the Voevodsky motive of the moduli space $\SM_{\SL}(r,d)$ is pure and we have
\begin{equation*}
\begin{split}
M(\SM_{\SL}(r,d)) &\cong M(\SM_{\SL'}(r,d))\in \mathrm{DM}^{\mathrm{eff}}(\CC,R),\\
h(\SM_{\SL}(r,d)) &\cong h(\SM_{\SL'}(r,d))\in \mathrm{Chow}^{\mathrm{eff}}(\CC,R),\\
\op{CH}^\bullet(\SM_{\SL}(r,d))& \cong \op{CH}^\bullet(\SM_{\SL'}(r,d)).
\end{split}
\end{equation*}
\end{theorem}

\begin{proof}
The proof is analogous to the one in Theorem \ref{thm:equalMotiveHodge} and Theorem \ref{thm:equalMotive}, but we now use the technical theorems from Appendices A and B of \cite{HL21} to perform the necessary computations in $ \mathrm{DM}^{\mathrm{eff}}(\CC,R)$ instead of $\hat{K}(\VarC)$. By Theorem \ref{thm:semiprojectiveHodge}, for every rank one algebroid $\SL=(L,[\cdot\,,\cdot],\delta)$ satisfying the hypothesis of the theorem the moduli space $\SM_{\SL}^{\Hod}(r,d)$ is a smooth quasiprojective semiprojective variety with a $\CC^*$-equivariant submersion $\pi:\SM_{\SL}^{\Hod}(r,d)\to \CC$ such that $\pi^{-1}(0)= \SM_{L^*}^{\Dol}(r,d)$ and $\pi^{-1}(1)=\SM_{\SL}(r,d)$. Then \cite[Theorem B.1]{HL21} and \cite[Corollary B.2]{HL21} yield isomorphisms
$$M( \SM_{L^{-1}}^{\Dol}(r,d))= M(\pi^{-1}(0)) \cong M(\SM_{\SL}^{\Hod}(r,d)) \cong M(\pi^{-1}(1)) =M(\SM_{\SL}(r,d))$$
$$\op{CH}^\bullet( \SM_{L^{-1}}^{\Dol}(r,d))= \op{CH}^\bullet(\pi^{-1}(0)) \cong \op{CH}^\bullet(\SM_{\SL}^{\Hod}(r,d)) \cong \op{CH}^\bullet(\pi^{-1}(1)) =\op{CH}^\bullet(\SM_{\SL}(r,d))$$
As $\SM_{L^{-1}}^{\Dol}(r,d)$ is also smooth and semiprojective by Proposition \ref{prop:semiprojectiveHiggs}, its motive is pure by \cite[Corollary A.5]{HL21}, so the motive of $\SM_{\SL}(r,d)$ is also pure and, therefore, the above isomorphism of Voevodsky motives induces an isomorphism of the Chow motives.
Thus, we can assume without loss of generality that the Lie algebroid structures for $\SL$ and $\SL'$ are trivial (i.e., that we have moduli spaces of twisted Higgs bundles). Furthermore, purity of the Voevodsky motives and representability of the Chow groups as certain spaces of morphisms in $\mathrm{DM}^{\mathrm{eff}}(\CC,R)$ (c.f. \cite[Corollary B.2]{HL21}) imply that it is enough to prove that the Voevodsky motives are isomorphic to conclude the desired isomorphisms between Chow motives or Chow rings. The Bialynicki-Birula decomposition of the moduli spaces of Higgs bundles yield the following motivic decompositions \cite[Theorem A.4]{HL21}, in which we follow the notation from Section \ref{section:HiggsPolynomials}.
$$M(\SM_{L^{-1}}^{\Dol}(r,d))\cong \bigoplus_{k=1}^r \bigoplus_{\begin{array}{c}
{\scriptstyle (\overline{r},\overline{d})\in\Delta_{L^{-1}}}\\
{\scriptstyle |\overline{r}|=r,\,|\overline{d}|=d}
\end{array}} M(\VHS_{L^{-1}}(\overline{r},\overline{d}))\{N_{L^{-1},\overline{r},\overline{d}}^-\}$$
$$M(\SM_{(L')^{-1}}^{\Dol}(r,d))\cong \bigoplus_{k=1}^r \bigoplus_{\begin{array}{c}
{\scriptstyle (\overline{r},\overline{d})\in\Delta_{(L')^{-1}}}\\
{\scriptstyle |\overline{r}|=r,\,|\overline{d}|=d}
\end{array}} M(\VHS_{(L')^{-1}}(\overline{r},\overline{d}))\{N_{(L')^{-1},\overline{r},\overline{d}}^-\}$$
By Corollary \ref{cor:VHSChains} we have $\Delta_{L^{-1}}=\Delta_{(L')^{-1}}$. Calling this set $\Delta$, then Corollary \ref{cor:VHSChains} and Lemma \ref{lemma:equalMorseIndex} imply that for each $(\overline{r},\overline{d})\in \Delta$, we have $\VHS_{L^{-1}}(\overline{r},\overline{d})\cong \VHS_{(L')^{-1}}(\overline{r},\overline{d})$ and $N_{L^{-1},\overline{r},\overline{d}}^-=N_{(L')^{-1},\overline{r},\overline{d}}^-$, so we obtain a term-by-term isomorphism of the previous Voevodsky motives.
\end{proof}

This result can be considered as an extension to moduli spaces of Lie algebroid connections (over $\CC$) of Hoskins and Lehalleur's ``motivic non-abelian Hodge correspondence'' \cite[Theorem 4.2]{HL21}, in which it is proved that there exists an isomorphism between the Voevodsky motives and Chow rings of the de Rham and $K$-twisted Higgs moduli spaces.

\subsection{Further topological properties of moduli spaces of Lie algebroid connections}

Similarly to how we used the smoothness of the moduli space of twisted Higgs bundles to prove the smoothness of the moduli space of $\SL$-connections back in section \ref{section:semiprojectivityHodge}, we can also use the regularity properties and the Bialynicki-Birula stratification of $\SM_{\SL}^{\Hod}(r,d)$ to transfer other known properties of the moduli spaces of twisted Higgs bundles to moduli spaces of $\SL$-connections. As an example, in this section we prove that, under certain conditions, the moduli space of $\SL$-connections is irreducible and compute some of its homotopy groups, by showing that they are isomorphic to the ones of the moduli space of vector bundles. 

Let $\mathrm{\mathbf{M}}(r,d)$ denote the moduli space of semistable vector bundles of rank $r$ and degree $d$ on the curve $X$.

\begin{lemma}
\label{lemma:codimUnstable}
Let $X$ be a smooth projective curve of genus $g\ge 2$ and let $\SL$ be a rank $1$ Lie algebroid on $X$ such that $\deg(\SL)<2-2g$. Suppose that $r\geq 2$ and $d$ are coprime. Then the loci in $\SM_{\SL}(r,d)$ corresponding to semistable $\SL$-connections whose underlying vector bundle is not stable has codimension at least $(g-1)(r-1)$. In particular, with the given bounds on $g$ and $r$ it has codimension at least $1$.
\end{lemma}

\begin{proof}
Let $\SL=(L,[\cdot\,,\cdot].\delta)$. By Proposition \ref{prop:semiprojectiveHiggs} and Theorem \ref{thm:semiprojectiveHodge} we know that the moduli spaces $\SM_{\SL}^{\Hod}(r,d)$ and $\pi^{-1}(0)=\SM_{L^{-1}}^{\Dol}(r,d)\subset\SM_{\SL}^{\Hod}(r,d)$ are smooth semiprojective varieties for the $\CC^*$-action \eqref{eq:C*-action-LHodgemod} and its restriction \eqref{eq:C*-actHiggs} to $\SM_{L^{-1}}^{\Dol}(r,d)$ respectively. Recall that here $\pi$ is the map \eqref{eq:project->lambda}.

The fixed-point locus of this action is concentrated in $\SM_{L^{-1}}^{\Dol}(r,d)$ and, by Lemma \ref{lemma:fixedPointsVHS}, it corresponds to the subset of variations of Hodge structure (recall \eqref{eq:notation}),
$$\SM_{L^{-1}}^{\Dol}(r,d)^{\CC^*}=\SM_{\SL}^{\Hod}(r,d)^{\CC^*}=\bigcup_{\begin{array}{c}
{\scriptstyle (\overline{r},\overline{d}) \in \Delta_{L^{-1}}}\\
{\scriptstyle |\overline{r}|=r, \, |\overline{d}|=d}
\end{array}} \VHS_{L^{-1}}(\overline{r},\overline{d}).$$
In this decomposition there is a distinguished component, namely the one for which $\overline{r}=\{r\}$ and $\overline{d}=\{d\}$. It parameterises points of the form $(E,0,0)$ with $E$ stable (because $(r,d)=1$), and it is therefore isomorphic to the moduli space $\mathrm{\mathbf{M}}(r,d)$.
Let
$$\SM_{L^{-1}}^{\Dol}(r,d)=\bigcup_{\begin{array}{c}
{\scriptstyle (\overline{r},\overline{d}) \in \Delta_{L^{-1}}}\\
{\scriptstyle |\overline{r}|=r, \, |\overline{d}|=d}
\end{array}} U_{L^{-1},\overline{r},\overline{d}}^+\ \ \ \text{ and }\ \ \ \SM_{\SL}^{\Hod}(r,d)=\bigcup_{\begin{array}{c}
{\scriptstyle (\overline{r},\overline{d}) \in \Delta_{L^{-1}}}\\
{\scriptstyle |\overline{r}|=r, \, |\overline{d}|=d}
\end{array}} \tilde{U}_{L^{-1},\overline{r},\overline{d}}^+$$
be the corresponding Bialynicki-Birula decompositions, hence where
$$U_{L^{-1},\overline{r},\overline{d}}^+ = \left \{ (E,\nabla_\SL,0) \in \SM_{L^{-1}}^{\Dol}(r,d) \, \middle | \, \lim_{t\to 0} (E,t\nabla_\SL)\in \VHS_{L^{-1}}(\overline{r},\overline{d}) \right\},$$ $$\tilde{U}_{L^{-1},\overline{r},\overline{d}}^+ = \left \{ (E,\nabla_\SL,\lambda) \in \SM_{\SL}^{\Hod}(r,d) \, \middle | \, \lim_{t\to 0} (E,t\nabla_\SL,t\lambda)\in \VHS_{L^{-1}}(\overline{r},\overline{d}) \right\}$$
are affine bundles over $\VHS_{L^{-1}}(\overline{r},\overline{d})$ of rank $N^+_{L^{-1},\overline{r},\overline{d}}$ and $\tilde{N}^+_{L^{-1},\overline{r},\overline{d}}$ respectively.

Let us write $U=U_{L^{-1},\{r\},\{d\}}^+$ and $\tilde{U}=\tilde{U}_{L^{-1},\{r\},\{d\}}^+$ to denote the affine bundles lying over $\mathrm{\mathbf{M}}(r,d)$. Let $S$ and $\tilde{S}$ denote the subsets of $\SM_{L^{-1}}^{\Dol}(r,d)$ and $\SM_{\SL}^{\Hod}(r,d)$ respectively corresponding to triples $(E,\nabla_\SL,\lambda)$ with $E$ not stable. If $E$ is a stable vector bundle then for every $(E,\nabla_\SL,\lambda)\in \SM_{\SL}^{\Hod}(r,d)$ we have
$$\lim_{t\to 0} (E,t\nabla_\SL,t\lambda) = (E,0,0)\in\mathrm{\mathbf{M}}(r,d)\subset \SM_{\SL}^{\Hod}(r,d)^{\CC^*},$$
so $S\subset \SM_{L^{-1}}^{\Dol}(r,d)\backslash U$ and $\tilde{S}\subset \SM_{\SL}^{\Hod}(r,d)\backslash \tilde{U}$. Actually, by \cite[Proposition 5.1]{BGL11}, $S=\SM_{L^{-1}}^{\Dol}(r,d)\backslash U$ (the proof is given for the moduli space with fixed determinant, but it also works just for fixed degree) and in \cite[Proposition 5.4]{BGL11} it is proven that $\codim(S)\ge (g-1)(r-1)$ by showing that, if $U_{L^{-1},\overline{r},\overline{d}}^+\neq U$, then
$$
\codim(U_{L^{-1},\overline{r},\overline{d}}^+)\ge (g-1)(r-1).
$$
On the other hand, we know $\tilde{N}_{L^{-1},\overline{r},\overline{d}} = N_{L^{-1},\overline{r},\overline{d}}+1$ (see \eqref{eq:tildeN=N+1}), hence $\dim \tilde{U}_{L^{-1},\overline{r},\overline{d}}^+ = \dim U_{L^{-1},\overline{r},\overline{d}}^+ +1$.
By Lemma \ref{lemma:smoothHodgeModuli}, $\dim \SM_{\SL}^{\Hod}(r,d)=\dim(\SM_{L^{-1}}^{\Dol}(r,d))+1$ so we conclude that
$$\codim(\tilde{U}_{L^{-1},\overline{r},\overline{d}}^+) = \codim(U_{L^{-1},\overline{r},\overline{d}}^+) \ge (g-1)(r-1).$$
Finally, define $S'=\tilde{S}\cap \pi^{-1}(1)$. As $\CC^*$ preserves the stability of the underlying bundle, then $\tilde{S}$ is $\CC^*$-invariant and the restriction of this action to $\tilde{S}$ gives an isomorphism
$$\tilde{S}\cap \pi^{-1}(\CC^*) \cong S'\times \CC^*.$$
Therefore,
$$\dim(S')= \dim (\tilde{S}\cap \pi^{-1}(\CC^*))-1 \le \dim(\tilde{S})-1.$$
Finally, as $\dim \SM_{\SL}(r,d) = \dim \SM_{\SL}^{\Hod}(r,d)-1$, we conclude that $$\codim(S') \ge \codim(\tilde{S}) \ge (g-1)(r-1),$$ completing the proof.
\end{proof}

Consider an $\SL$-connection $(E,\nabla_\SL)$ of rank $r$ and degree $d$ on $X$. Let $\mathbb{E}$ be the $C^\infty$ vector bundle on $X$ underlying the algebraic vector bundle $E$. Note that $\mathbb{E}$ is independent of the choice of the $\SL$-connection, as long as the rank and degree are still $r$ and $d$. Let $\mathcal{G}(\mathbb{E})$ be the unitary gauge group for a fixed Hermitian metric on $\mathbb{E}$. In other words,
$\mathcal{G}(\mathbb{E})=\Omega^0(\mathrm{U}(\mathbb{E}))$ is the space of $C^\infty$-sections of the $C^\infty$-bundle of unitary automorphisms of $\mathbb{E}$.

\begin{theorem}
\label{thm:homotopygrps}
Let $X$ be a smooth complex curve of genus $g\ge 2$ and let $\SL$ be an algebroid on $X$ such that $\rk(\SL)=1$ and $\deg(\SL)<2-2g$. Suppose that $r$ and $d$ are coprime and that $r\ge 2$. Then, $\SM_{\SL}(r,d)$ is connected, hence irreducible. If, moreover, $(r,g)\neq(2,2)$, then its higher homotopy groups are given as follows:
\begin{itemize}
\item $\pi_1(\SM_{\SL}(r,d))\cong H_1(X,\ZZ)\cong\ZZ^{2g}$;
\item $\pi_2(\SM_{\SL}(r,d))\cong \ZZ$;
\item $\pi_k(\SM_{\SL}(r,d))\cong \pi_{k-1}(\mathcal{G}(\mathbb{E}))$, for every $k=3,\ldots,2(g-1)(r-1)-2$.
\end{itemize} 
\end{theorem}

\begin{proof}
Let $\SL=(L,[\cdot\,,\cdot].\delta)$. By Lemma \ref{lemma:smoothHodgeModuli} we know that the moduli space $\SM_{\SL}(r,d)$ is a smooth variety whose components are all of the same dimension $1-r^2\deg(\SL)$. Let $S'\subset \SM_{\SL}(r,d)$ be the subspace of $\SL$-connections $(E,\nabla_\SL)$ with $E$  not stable. Define $U'=\SM_{\SL}(r,d)\backslash S'$.

Let us prove that $U'$ is a torsor over $\mathrm{\mathbf{M}}(r,d)$ with contractible fibres, so $\mathrm{\mathbf{M}}(r,d)$ is a deformation retract of $U'$. Consider the forgetful map $\pi_{\mathrm{\mathbf{M}}}:U'\to\mathrm{\mathbf{M}}(r,d)$, given by $\pi_{\mathrm{\mathbf{M}}}(E,\nabla_\SL)= E$. By Corollary \ref{cor:existsConnection} the map is surjective and we have the following explicit description of each fibre
$$\pi_{\mathrm{\mathbf{M}}}^{-1}(E) = \left \{\nabla_\SL: E\to E\otimes L^* \, \middle | \,  \nabla_\SL(fs)=f\nabla_\SL(s) + s\otimes d_\SL(f),\ \forall s\in E, \, \forall f \in \SO_X\right\}.$$
Observe that if $\nabla_\SL,\nabla_\SL' \in \pi_E^{-1}(E)$, then $\nabla_\SL-\nabla_\SL' \in H^0(\End(E)\otimes L^*)$, so $\pi_{\mathrm{\mathbf{M}}}^{-1}(E) $ is an affine space on $H^0(\End(E)\otimes L^*)$. Moreover, 
$$H^1(\End(E)\otimes L^*) \cong H^0(\End(E)\otimes K\otimes L)^* =0$$
because, since $E$ is stable, $\End(E)$ is semistable and so $\End(E)\otimes K_X\otimes L$ is a semistable vector bundle with
$\deg(\End(E)\otimes K_X\otimes L)= r^2 (\deg(K_X)+\deg(L))<0$.
Therefore, by Riemann-Roch, for each $E\in \mathrm{\mathbf{M}}(r,d)$,
$$\dim(\pi_{\mathrm{\mathbf{M}}}^{-1}(E))=\dim H^0(\End(E)\otimes L^*) =1-r^2\deg(L)-g$$
is constant and thus the map $\pi_{\mathrm{\mathbf{M}}}$ is equidimensional. Let $\SE\to X\times \mathrm{\mathbf{M}}(r,d)$ be the universal bundle over $\mathrm{\mathbf{M}}(r,d)$ (i.e., the bundle whose fibre over $X\times \{E\}$ is isomorphic to $E$); it exists since $(r,d)=1$. Let $\pi_X:X\times\mathrm{\mathbf{M}}(r,d)\to X$ be the projection. Then we conclude that $U'$ is a torsor for the vector bundle
$$R(\pi_X)_*(\End (\SE)\otimes \pi_X^* L^*) \longrightarrow \mathrm{\mathbf{M}}(r,d).$$
It follows that the homotopy groups of $U'$ verify $\pi_k(U')\cong\pi_k(\mathrm{\mathbf{M}}(r,d))$ for every $k\geq 0$.

By Lemma \ref{lemma:codimUnstable}, $\codim(S')\geq (g-1)(r-1)$. Since $\SM_{\SL}(r,d)$ is smooth, this implies that $$\pi_k(\SM_{\SL}(r,d))\cong\pi_k(U')\cong\pi_k(\mathrm{\mathbf{M}}(r,d)),$$ for every $k=0,\ldots,2(g-1)(r-1)-2$. The moduli $\mathrm{\mathbf{M}}(r,d)$ is connected, hence so is $\SM_{\SL}(r,d)$, and thus irreducible because it is smooth.
As for the higher homotopy groups, the results follow from \cite[Theorem 3.1]{DU95}.
\end{proof}

\begin{remark}
By taking the trivial Lie algebroid $\SL=(L,0,0)$ in the above theorem, one gets the results for the moduli space of $\SM_{L^{-1}}^{\Dol}(r,d)$ of $L^{-1}$-twisted Higgs bundles. The irreducibility conclusion was proved in \cite{BGL11} by the same arguments, and actually the higher homotopy groups of $\SM_{L^{-1}}^{\Dol}(r,d)$ would also follow directly from \cite{BGL11} by the same argument as above.

On the other hand, improvements on the bound of $k$ for which the isomorphism  $\pi_k(\SM_{\SL}(r,d))\cong \pi_{k-1}(\mathcal{G}(\mathbb{E}))$ holds have been achieved, for twisted Higgs bundles, in certain particular situations (cf.~\cite{H98} and \cite{ZR18}), hence we might expect that such isomorphism also holds, in this generality, for higher values of $k$. 
\end{remark}

\subsection{Consequences for the geometry of moduli spaces of meromorphic or logarithmic connections}\label{irregular connections}

In Section \ref{subsection:classification}, we discussed that the moduli spaces of $\SL$-connections considered in our main Theorem \ref{thm:equalMotive} are either isomorphic to moduli spaces of twisted Higgs bundles, moduli spaces of logarithmic connections and moduli spaces of meromorphic connections. Among these types of moduli spaces, the moduli spaces of meromorphic connections being considered, where no information about its residues or Stokes data is fixed except for the maximum order of the poles at the punctures, stand as the ones whose geometry is less understood (see \cite{Bo13} for studies of the subvarieties corresponding to connections with given fixed Stokes data and monodromies). The new results then shed some light on the geometric properties of these moduli spaces.

Recall that given an effective divisor $D=\sum_{i=1}^n k_i x_i$ on $X$, with $k_i\geq 1$ we denote by $\SM_{\op{conn}}(D,r,d)$ the moduli space of rank $r$ and degree $d$ semistable \emph{singular} connections, with poles of order at most $k_i$ over each $x_i\in D$, and that this moduli space is isomorphic to the moduli space $\SM_{\ST_X(-D)}(r,d)$ of $\ST_X(-D)$-connections, where $\ST_X(-D)\subset\ST_X$ is the subalgebroid of $\ST_X$ with the induced algebroid structure with anchor $i_D:T_X(-D)\hookrightarrow T_X$.

On one hand, by Theorem \ref{thm:homotopygrps}, we know that $\SM_{\op{conn}}(D,r,d)$ is a smooth irreducible variety of dimension $1+r^2(2g-2+\deg(D))$ whose first homotopy groups are the ones given in loc.~cit. Moreover, we have the following direct corollary of Theorem \ref{thm:equalMotive} regarding their motives and $E$-polynomials.

\begin{corollary}
If $D$ and $D'$ are any two effective divisors on $X$ with $\deg(D)=\deg(D')$ and $r$ and $d$ are coprime, then
$[\SM_{\op{conn}}(D,r,d)]=[\SM_{\op{conn}}(D',r,d)]\in \hat{K}(\VarC)$ and $E(\SM_{\op{conn}}(D,r,d))=E(\SM_{\op{conn}}(D',r,d))$.
\end{corollary}

In particular, by taking $D'$ to be a simple divisor, we conclude the following.

\begin{corollary}\label{cor:motivlog=irreg}
If $r$ and $d$ are coprime, then the motivic class and $E$-polynomial of any moduli space of rank $r$ degree $d$ \emph{meromorphic} connections on a smooth projective curve $X$ of genus at least $2$ equals that of any moduli space of rank $r$ degree $d$ \emph{logarithmic} connections on $X$, with singular divisor of the same degree.
\end{corollary}

\section{Explicit computation of motives and $E$-polynomials for rank $2$ and $3$}\label{classes and polynomials}
\label{section:motiveComputations}

Fix a rank $1$ Lie algebroid $\SL=(L,[\cdot\,,\cdot],\delta)$ on the curve $X$, such that $\deg(L)<2-2g$. In this section, we provide explicit formulae for the motivic classes and $E$-polynomials of the moduli spaces of $\SL$-connections of rank $2$ and $3$ and coprime degree. Theorem \ref{thm:equalMotive} allows us to perform all computations by just considering the trivial Lie algebroid $(L,0,0)$, that is, the moduli space of $L^{-1}$-twisted Higgs bundles of corresponding rank and degree.

\subsection{Recollection of properties of motives}
We first need to introduce some notation and recall, without proof, some facts on the theory of motivic classes in $\hat{K} (\VarC)$. For details, see for example \cite{H07,Kapranov00}.

The symmetric product of a variety gives rise to the \emph{$\lambda$-operator} defined, for each $n\geq 0$, as
\begin{equation}\label{lambda-oper}
\lambda^n: \hat{K} (\VarC)\to \hat{K} (\VarC),\ \ \ \lambda^n([Y])=[\Sym^n(Y)].
\end{equation}
For example $\lambda^n(\LL^k)=\LL^{nk}$.
With these operators, $\hat{K} (\VarC)$ acquires the structure of a 
\emph{$\lambda$-ring}. In particular, the relation
\begin{equation}
\lambda^n([Y]+[Z])=\sum_{i+j=n}\lambda^i([Y])\lambda^j([Z]),
\end{equation} 
holds. 

The motive of our fixed genus $g\geq 2$ curve $X$ splits as $[X]=1+h^1(X)+\LL$, where $h^1(X)\in\hat{K}(\VarC)$ is such that the motive of the Jacobian of $X$ is given by 
\begin{equation}\label{eq:motJac}
[\Jac(X)]=\sum_{i=0}^{2g}\lambda^i(h^1(X)).
\end{equation} 
Define the $\hat{K}(\VarC)$-valued polynomial
\begin{equation}\label{polynomial P}
P_X(x)=\sum_{i=0}^{2g}\lambda^i(h^1(X))x^i\in \hat{K}(\VarC)[\hspace{-.03cm}[x]\hspace{-.03cm}]
\end{equation}
and note that $P_X(1)=[\Jac(X)]$.
Consider also the \emph{zeta function} of $X$, defined as
\[Z(X,x)=\sum_{k\geq 0}\lambda^k([X])x^k\in \hat{K}(\VarC)[\hspace{-.03cm}[x]\hspace{-.03cm}].\]

Using that $\lambda^n(h^1(X))=\LL^{n-g}\lambda^{2g-n}(h^1(X))$, if $n=0,\ldots,2g$, and that $\lambda^n(h^1(X))=0$ if $n>2g$, it follows that
\[\lambda^n([X])=\coeff_{x^0}\frac{Z(X,x)}{x^n}=\coeff_{x^0}\frac{P_X(x)}{(1-x)(1-\LL x)x^n}.\]

\subsection{Motives of $\SM_{\SL}(r,d)$ for $r=2,3$}

Now we move on to the motives of moduli spaces. We want to compute the motive of $\SM_{L^{-1}}^{\Dol}(r,d)$, for $r=2,3$ and $d$ coprime with $r$. This will be done by using the formula \eqref{eq:EPolHiggs}, and so we will need to consider the moduli space of rank $r$ and degree $d$ vector bundles (which we think of consisting of Higgs bundles which are variations of Hodge structure of type $(r)$) and then variations of Hodge structure of type $(1,1)$ for $r=2$ and type $(1,2)$, $(2,1)$ and $(1,1,1)$ for $r=3$. We will not fill the full details of the computations, and leave them to the reader. 

In this section, we use the notation $d_L$ for the degree of the line bundle $L$, so that $d_L<2-2g$.

Recall that $\mathrm{\mathbf{M}}(r,d)$ denotes the moduli space of stable vector bundles of rank $r$ and degree $d$ over the curve $X$.

Let us start with rank $2$ case. Let $d$ be odd. By Example 3.4 of \cite{GPHS11} or equation (3.9), page 41 of \cite{S14}, the motivic class of $\mathrm{\mathbf{M}}(2,d)$ is given by 
\begin{equation}\label{eq:M(2,1)}
[\mathrm{\mathbf{M}}(2,d)]=\frac{[\Jac(X)]P_X(\LL)-\LL^g[\Jac(X)]^2}{(\LL-1)(\LL^2-1)},
\end{equation}
where $P_X$ is the polynomial given in \eqref{polynomial P}. 
Notice that this formula is obtained by the one in \cite{GPHS11}  by multiplying by $\LL-1$ because in loc.\ cit., the stated formula stands for the stack of stable vector bundles, which is a $\CC^*$-gerbe over $\mathrm{\mathbf{M}}(2,d)$.

We move on the motivic class of subvarieties of $\SM_{L^{-1}}^{\Dol}(2,d)$ corresponding to variations of Hodge structure of type $(1,1)$. An $L^{-1}$-twisted-Higgs bundle $(E,\varphi)$ lies in $\VHS_{L^{-1}}((1,1),(d_1,d-d_1))$ if it is stable and
\begin{equation}\label{eq:vhs11}
E=E_1\oplus E_2,\ \  \varphi=\smtrx{0&0\\ \varphi_1&0},
\end{equation} with $E_1,E_2$ line bundles of degree $d_1$ and $d-d_1$ respectively and $\varphi_1:E_1\to E_2\otimes L^{-1}$ nonzero. The fact that $\varphi_1\neq 0$ and stability ($E_2$ is $\varphi$-invariant) impose conditions on the degree $d_1$ and, indeed,
\[
\VHS_{L^{-1}}((1,1),(d_1,d-d_1))\neq\emptyset\Longleftrightarrow d/2<d_1\leq (d-d_L)/2.
\]
In such a case, the map $\left(E_1\oplus E_2,\smtrx{0&0\\ \varphi_1&0}\right)\mapsto(\mathrm{div}(\varphi_1), E_2)$, where $\mathrm{div}(\varphi_1)$ denotes the divisor of the section $\varphi_1\in H^0(E_1^{-1}E_2L^{-1})$,  yields the isomorphism
\begin{equation}\label{eq:VHS11}
\VHS_{L^{-1}}((1,1),(d_1,d-d_1))\cong\Sym^{d-2d_1-d_L}(X)\times\Jac^{d-d_1}(X).
\end{equation} 
Hence, since the ``Jacobian'' of degree $d-d_1$ line bundles on $X$ is isomorphic to the Jacobian $\Jac(X)$ of (degree $0$ line bundles on) $X$, 
\begin{equation}\label{eq:VHS(1,1)}
[\VHS_{L^{-1}}((1,1),(d_1,d-d_1))]=\lambda^{d-2d_1-d_L}([X])[\Jac(X)],
\end{equation}
where we are using the $\lambda$-operations defined in \eqref{lambda-oper}.

Now we consider the rank $3$ case. Fix $d$ coprime with $3$, so that every semistable $L^{-1}$-twisted Higgs bundle is stable. We have
\begin{equation}\label{eq:M(3,1)}
\begin{split}
[\mathrm{\mathbf{M}}(3,d)]=&\frac{[\Jac(X)]}{(\LL-1)(\LL^2-1)^2(\LL^3-1)}\Big(\LL^{3g-1}(1+\LL+\LL^2)[\Jac(X)]^2\\
&-\LL^{2g-1}(1+\LL)^2[\Jac(X)]P_X(\LL)+P_X(\LL)P_X(\LL^2)\Big).
\end{split}
\end{equation}
by \cite[Remark 3.5]{GPHS11} or \cite[Theorem 4.7]{S14}. As in the $r=2$ case, this is obtained by the stated formula in \cite{GPHS11} by multiplying by $\LL-1$.

Consider now variations of Hodge structure of type $(1,2)$ in $\SM_{L^{-1}}^{\Dol}(3,d)$. A stable $L^{-1}$-twisted Higgs bundle $(E,\varphi)$ lies in $\VHS_{L^{-1}}((1,2),(d_1,d-d_1))$ if it is of the form \eqref{eq:vhs11}, with the only difference that now $E_2$ has rank $2$.  Let $I\subset E_2$ be the line bundle such that the saturation of the image of $\varphi_1$ equals $IL^{-1}$. Then both $E_2$ and $E_1\oplus I$ are $\varphi$-invariant subbundles of $E$. Checking the stability for them imposes conditions of $d_1$, and actually we have that 
\begin{equation}\label{eq:non-empty(1,2)}
\VHS_{L^{-1}}((1,2),(d_1,d-d_1))\neq\emptyset\Longleftrightarrow d/3<d_1<d/3-d_L/2.
\end{equation}
Moreover, by Corollary \ref{cor:VHSChains}, there is an isomorphism with the moduli space of $\alpha_{L^{-1}}=(-d_L,0)$-stable chains
\begin{equation}\label{eq:VHS12}
\VHS_{L^{-1}}((1,2),(d_1,d-d_1))\cong \HC^{\alpha_{L^{-1}}}((1,2),(d_1+d_L,d-d_1)).
\end{equation}
Since $d$ is coprime with $3$, then $\alpha_{L^{-1}}$ is not a critical value (i.e.~a value of the stability parameter where semistability changes), so the motive of $\HC^{\alpha_{L^{-1}}}((1,2),(d_1+d_L,d-d_1))$ can be read off from \cite[Example 6.4]{GPHS11}, by adapting the computation to the $L^{-1}$-twisting setting, or, perhaps more directly, from Theorem 3.2 of \cite{S14}, where the author considers the moduli space of $-d_L$-stable triples of type $((2,1)(d-d_1-2d_L,d_1))$ (cf.~\cite{BGG04}), which is isomorphic to the moduli space of $\alpha_{L^{-1}}$-stable chains. From this, we conclude that
\begin{equation}\label{eq:K-class of VHS(1,2)}
\begin{split}
[\VHS_{L^{-1}}((1,2),(d_1,d-d_1))]&=\frac{[\Jac(X)]^2}{\LL-1}\bigg(\LL^{2\lfloor d/3\rfloor-d+d_1+g+1}\lambda^{d-\lfloor d/3\rfloor-2d_1-d_L-1}([X]+\LL^2)\\
&\ \ \ -\lambda^{d-\lfloor d/3\rfloor-2d_1-d_L-1}([X]\LL+1)\bigg).
\end{split}
\end{equation}

The motives of the subvarieties of $\SM_{L^{-1}}^{\Dol}(3,d)$ corresponding to variations of Hodge structure of type $(2,1)$ are directly obtained from the ones of type $(1,2)$ by making use of the isomorphism
\begin{equation}\label{eq:VHS21isomVHS12}
\VHS_{L^{-1}}((2,1),(d_1,d-d_1))\cong \VHS_{L^{-1}}((1,2),(d_1-d,-d_1))
\end{equation}
arising from duality.

Finally, we deal with variations of Hodge structure of type $(1,1,1)$ in $\SM_{L^{-1}}^{\Dol}(3,d)$. Similarly to the previous cases, it follows that
\[
\VHS_{L^{-1}}((1,1,1),(d_1,d_2,d-d_1-d_2))\neq\emptyset\Longleftrightarrow (d_1,d_2)\in\Delta_{-d_L}(d),
\]
where 
\begin{multline}\label{eq:Delta_dl}
\Delta_{-d_L}(d)=\left\{(a,b)\in\ZZ^2\,|\,a-b\leq-d_L,\, a+2b-d\leq-d_L,\, a>d/3,\, a+b>2d/3\right\}\\
=\left\{(a,b)\in\ZZ^2\,|\, \lceil d/3 \rceil \le a \le -d_L+ \lfloor d/3 \rfloor , \max\{d_L+a, \lceil 2d/3 \rceil -a\}\le b\le \lfloor (d-d_L-a)/2\rfloor \}\right\}, 
\end{multline}
and in that case, we have the following isomorphism
\begin{equation}\label{eq:VHS111}
\VHS_{L^{-1}}((1,1,1),(d_1,d_2,d-d_1-d_2))\cong\Sym^{-d_1+d_2-d_L}(X)\times\Sym^{d-d_1-2d_2-d_L}(X)\times\Jac^{d-d_1-d_2}(X).
\end{equation}
Thus, 
\begin{equation}\label{eq:VHS(1,1,1)}
[\VHS_{L^{-1}}((1,1,1),(d_1,d_2,d-d_1-d_2))]=\lambda^{-d_1+d_2-d_L}([X])\lambda^{d-d_1-2d_2-d_L}([X])[\Jac(X)].
\end{equation}

Now we have the promised corollary of Theorem \ref{thm:equalMotive}.

\begin{corollary}\label{cor:K-class moduli rk2,3}
Let $X$ be a smooth projective curve of genus $g\geq 2$ and let $\SL$ be a rank $1$ Lie algebroid on $X$. Write $d_L=\deg(\SL)$ and suppose that $d_L<2-2g$. Then, the following formulas hold in $\hat{K}(\VarC)$.
\begin{enumerate}
\item If $(2,d)=1$,
\begin{equation*}
\begin{split}
[\SM_{\SL}(2,d)] &= \frac{\LL^{-4d_L+4-4g}\Big([\Jac(X)]P_X(\LL)-\LL^g[\Jac(X)]^2\Big)}{(\LL-1)(\LL^2-1)}\\
& \ \ \ + \LL^{-3d_L+2-2g}[\Jac(X)]\sum_{d_1=\lfloor d/2\rfloor+1}^{\lfloor\frac{d-d_L}{2}\rfloor}\lambda^{d-2d_1-d_L}([X]).
\end{split}
\end{equation*}
\item If $(3,d)=1$,
\begin{equation*}
\begin{split}
[\SM_{\SL}(3,d)] &= \frac{\LL^{-9d_L+9-9g}[\Jac(X)]}{(\LL-1)(\LL^2-1)^2(\LL^3-1)}\Big(\LL^{3g-1}(1+\LL+\LL^2)[\Jac(X)]^2\\
& \ \ \ -\LL^{2g-1}(1+\LL)^2[\Jac(X)]P_X(\LL)+P_X(\LL)P_X(\LL^2)\Big)\\
& \ \ \ +\frac{\LL^{-7d_L+5-5g}[\Jac(X)]^2}{\LL-1}\sum_{d_1=\lfloor d/3\rfloor+1}^{\lfloor\frac{d}{3}-\frac{d_L}{2}\rfloor}\bigg(\LL^{2\lfloor d/3\rfloor-d+d_1+g+1}\lambda^{d-\lfloor d/3\rfloor-2d_1-d_L-1}([X]+\LL^2)\\
& \ \ \ -\lambda^{d-\lfloor d/3\rfloor-2d_1-d_L-1}([X]\LL+1)\bigg)\\
& \ \ \ +\frac{\LL^{-7d_L+5-5g}[\Jac(X)]^2}{\LL-1}\sum_{d_1=\lfloor 2d/3\rfloor+1}^{\lfloor\frac{2d}{3}-\frac{d_L}{2}\rfloor}\bigg(\LL^{2\lfloor-d/3\rfloor+d_1+g+1}\lambda^{d-\lfloor-d/3\rfloor-2d_1-d_L-1}([X]+\LL^2)\\
& \ \ \ -\lambda^{d-\lfloor-d/3\rfloor-2d_1-d_L-1}([X]\LL+1)\bigg)\\
& \ \ \ +\LL^{-6d_L+3-3g}[\Jac(X)]\sum_{(d_1,d_2)\in\Delta_{-d_L}(d)}\lambda^{-d_1+d_2-d_L}([X])\lambda^{d-d_1-2d_2-d_L}([X]),
\end{split}
\end{equation*}
with $\Delta_{-d_L}(d)$ defined in \eqref{eq:Delta_dl}.
\end{enumerate}
\end{corollary}

\begin{remark}
It is easy to see that the motivic class $[\SM_{\SL}(3,d)]$ is the same by replacing $d$ by $-d$ (for the last sum, use the bijection between $\Delta_{-d_L}(d)$ and $\Delta_{-d_L}(-d)$ given by $(a,b)\mapsto(-d+a+b,-b)$). This is, of course, consequence of the fact that the duality map $(E,\nabla_\SL)\mapsto(E^*,\nabla_\SL^t\otimes \Id_{L^{-1}})$, with $\nabla_\SL^t:E^*\otimes L\to E^*$ the dual map of $\nabla_\SL$, yields an isomorphism between the moduli spaces $\SM_{\SL}(3,d)$ and $\SM_{\SL}(3,-d)$ (and, of course, the mentioned bijection $\Delta_{-d_L}(d)\simeq\Delta_{-d_L}(-d)$ is provided by this duality map).
\end{remark}
\begin{proof}
As we are considering cases in which the rank and the degree are coprime, we can apply Theorem \ref{thm:equalMotive}, and assume without loss of generality that $\SL$ has the trivial Lie algebroid structure $(L,0,0)$. Therefore, $\SM_{\SL}(r,d)$ corresponds to the moduli space of $L^{-1}$-twisted Higgs bundles of rank $r$ and degree $d$.

Then, everything follows from the decomposition \eqref{eq:EPolHiggs}, using the formula \eqref{eq:formula-N+} (with $L$ replaced by $L^{-1}$),
\[N_{L^{-1},\overline{r},\overline{d}}^+ = 1-r^2\deg(L) - \dim(\VHS_{L^{-1}}(\overline{r},\overline{d}))- M_{L^{-1},\overline{r},\overline{d}}/2\]
for the exponents of $\LL$ in each summand. All of them are straightforward, using \eqref{eq:Morseindex} and \eqref{eq:Eulerchar}. We leave the details of the computations for the reader. In rank $2$, use \eqref{eq:VHS11} followed by \eqref{eq:M(2,1)} and \eqref{eq:VHS(1,1)}. 

In rank $3$, use \eqref{eq:VHS12} and \eqref{eq:VHS111},  knowing that, in the $(1,2)$-type, 
\begin{equation*}
\begin{split}
\dim(\VHS_{L^{-1}}(X,(1,2),(d_1,d-d_1)))&=\dim(\HC^{\alpha_{L^{-1}}-ss}((1,2),(d_1+d_L,d-d_1)))\\
&=3g-2-3d_1+d-2d_L
\end{split}
\end{equation*} by Theorem A (2) of \cite{BGG04} and that the dimension of $\dim(\VHS_{L^{-1}}(X,(1,2),(d_1,d-d_1)))$ is computed similarly using \eqref{eq:VHS21isomVHS12}. Then, \eqref{eq:EPolHiggs} becomes 
\begin{equation*}\label{eq:Epolrk3-step1}
\begin{split}
[\SM_{\SL}(3,d)]&=\LL^{-9d_L+9-9g}[\mathrm{\mathbf{M}}(3,d)]\\
& \ \ \ +\LL^{-7d_L+5-5g}\sum_{d_1=\lfloor\frac{d}{3}\rfloor+1}^{\lfloor\frac{d}{3}-\frac{d_L}{2}\rfloor}[\VHS_{L^{-1}}((1,2),(d_1,d-d_1))]\\
& \ \ \ +\LL^{-7d_L+5-5g}\sum_{d_1=\lfloor\frac{2d}{3}\rfloor+1}^{\lfloor\frac{2d}{3}-\frac{d_L}{2}\rfloor}[\VHS_{L^{-1}}((2,1),(d_1,d-d_1))]\\
& \ \ \ +\LL^{-6d_L+3-3g}\sum_{(d_1,d_2)\in\Delta_{-d_L}(d)}[\VHS_{L^{-1}}((1,1,1),(d_1,d_2,d-d_1-d_2))],\\
\end{split}
\end{equation*}
and we obtain the result from \eqref{eq:M(3,1)}--\eqref{eq:VHS(1,1,1)}.
\end{proof}

\begin{remark}
Explicit motivic formulas in the classical case, namely for the moduli spaces of $K$-twisted Higgs bundles (for coprime rank and degree), are known from \cite{GPHS11} (cf.~Theorem $3$ for rank $2,3$ and Theorem $2$ for rank $4$).
\end{remark}

\subsection{$E$-polynomials of $\SM_{\SL}(r,d)$ for $r=2,3$}

We now apply the map \eqref{E-poly} to the motivic formulas of the preceding result. This will give the $E$-polynomials of the corresponding moduli spaces of rank $1$ Lie algebroid connections.

We will skip the details because these formulas are actually already known for the twisted Higgs bundle moduli spaces (hence also for the corresponding Lie algebroid connections moduli as a consequence of Theorem \ref{thm:equalMotiveHodge}) by \cite{CDP11} and the ones we obtain match them.

If $X$ is again our smooth projective curve of genus $g$, it is well-known that  $E(\Jac(X))=(1-u)^g(1-v)^g$, thus 
\[E(\lambda^n([X]))=\coeff_{x^0}\frac{(1-ux)^g(1-vx)^g}{(1-x)(1-uvx)x^n}.\]

In addition, from \eqref{eq:motJac} and \eqref{polynomial P}, we have that, for any $k$,
\begin{equation*}
\begin{split}
E(P_X(\LL^k))&=\sum_{i=0}^{2g}E(\lambda^i(h^1(X)))(u^kv^k)^i=\sum_{i=0}^{2g}\sum_{p+q=i}(-1)^ih^{p,q}(\Jac(X))u^pv^q(u^kv^k)^i\\
&=\sum_{i=0}^{2g}\sum_{p+q=i}(-1)^{p+q}h^{p,q}(\Jac(X))(u^{k+1}v^k)^p(u^kv^{k+1})^q\\
&=E(\Jac(X))(u^{k+1}v^k,u^kv^{k+1})=(1-u^{k+1}v^k)^g(1-u^kv^{k+1})^g.
\end{split}
\end{equation*}
In particular,
\begin{equation}\label{eq:poly PL PL2}
E(P_X(\LL))=(1-u^2v)^g(1-uv^2)^g\ \ \text{ and }\ \ E(P_X(\LL^2))=(1-u^3v^2)^g(1-u^2v^3)^g.
\end{equation}

\begin{corollary}\label{cor:E-poly_r=2,3}
Let $X$ be smooth projective curve of genus $g\ge 2$ and let $\SL$ be a rank $1$ Lie algebroid on $X$. Write $d_L=\deg(\SL)$ and suppose that $d_L<2-2g$.
\begin{enumerate}
\item Let $d$ be odd. Then,
\begin{equation*}
\begin{split}
E(\SM_{\SL}(2,d))& = (uv)^{-4d_L+4-4g} E(\mathrm{\mathbf{M}}(2,d))\\
&\ \ \ + (uv)^{-3d_L+2-2g}(1-u)^g(1-v)^g \coeff_{x^0}\left(\frac{(1-ux)^g(1-vx)^gx^{d_L+1}}{(1-x^2)(1-x)(1-uvx)}\right),
\end{split}
\end{equation*}
where
$$E(\mathrm{\mathbf{M}}(2,d))=\frac{(1-u)^g(1-v)^g(1-u^2v)^g(1-uv^2)^g-(uv)^g(1-u)^{2g}(1-v)^{2g}}{(uv-1)((uv)^2-1)}.$$
\item Let $d$ be coprime with $3$. Then,
\begin{equation*}
\begin{split}
E(\SM_{\SL}(3,d))&= (uv)^{-9d_L+9-9g} E(\mathrm{\mathbf{M}}(3,d))\\
&\ \  +\frac{(1-u)^{2g}(1-v)^{2g}(uv)^{-7d_L+6-4g}}{uv-1}\cdot\coeff_{x^0}\frac{(1-ux)^g(1-vx)^gx^{d_L+2}}{(1-x)(1-uvx)(1-(uv)^2x)(1-uvx^2)}\\
&\ \  -\frac{(1-u)^{2g}(1-v)^{2g}(uv)^{-8d_L+6-5g}}{uv-1}\cdot\coeff_{x^0}\frac{(1-ux)^g(1-vx)^gx^{d_L+2}}{(1-x)(1-uvx)(uv-x)((uv)^2-x^2)}\\
&\ \  +\frac{(1-u)^{2g}(1-v)^{2g}(uv)^{-7d_L+5-4g}}{uv-1}\cdot\coeff_{x^0}\frac{(1-ux)^g(1-vx)^gx^{d_L+1}}{(1-x)(1-uvx)(1-(uv)^2x)(1-uvx^2)}\\
&\ \  -\frac{(1-u)^{2g}(1-v)^{2g}(uv)^{-8d_L+7-5g}}{uv-1}\cdot\coeff_{x^0}\frac{(1-ux)^g(1-vx)^gx^{d_L+1}}{(1-x)(1-uvx)(uv-x)((uv)^2-x^2)}\\
&\ \  +(1-u)^g(1-v)^g(uv)^{-6d_L+3-3g}\cdot\\
&\ \  \cdot\coeff_{x^0y^0}\frac{(1-ux)^g(1-vx)^g(1-uy)^g(1-vy)^gx^{2d_L+2}y^{2d_L+1}(x^{-2d_L}-y^{-d_L})(y^{-2d_L}-x^{-d_L})}{(1-x)(1-uvx)(1-y)(1-uvy)(x-y^2)(y-x^2)},
\end{split}
\end{equation*}
where
\begin{equation*}
\begin{split}
E(\mathrm{\mathbf{M}}(3,d))&=\frac{(1-u)^g(1-v)^g}{(uv-1)((uv)^2-1)^2((uv)^3-1)} \bigg((uv)^{3g-1}(1+uv+(uv)^2)(1-u)^{2g}(1-v)^{2g}  \\
&- (uv)^{2g-1}(1+uv)^2(1-u)^g(1-v)^g(1-u^2v)^g (1-uv^2)^g \\
&+ (1-u^2v)^g(1-uv^2)^g(1-u^3v^2)^g(1-u^2v^3)^g\bigg).
\end{split}
\end{equation*}
\end{enumerate}
\end{corollary}
\begin{proof}
This follows from Corollary \ref{cor:K-class moduli rk2,3} and from computations which are now standard. We leave the details for the reader, who may see for example \cite{G95}, \cite{GPHS11},  \cite{SB10} (especially Theorems 3.1.4 and 3.5.7) and \cite{S14} for techniques on these computations. Note however that the ones in \cite{SB10} contain some slight inaccuracies (so that the final results stated there, in Theorems 3.1.4 and 3.5.7, are not correct). 
The formulas for the $E$-polynomials of the appropriate powers of the $\lambda$-operations of $[X]+\LL^2$ and of $[X]\LL+1$, which appear in Corollary \ref{cor:K-class moduli rk2,3}, may be found in equation (3.3) of page 39 of \cite{S14} and the ones concerning $P_X(\LL)$ and $P_X(\LL^2)$ are given in \eqref{eq:poly PL PL2}.
\end{proof}

\begin{remark}
The $E$-polynomials of the moduli space of vector bundles of rank $2$ and $3$ and coprime degree $d$ were first computed recursively in \cite[Theorem 1]{RK00}. The one for rank $3$ was also explicitly obtained with different techniques, in Theorem 1.2 of \cite{Mu08} (even though the formula there has a minor inaccuracy on a sign and the one on Theorem 7.1 -- not in Theorem 1.2 -- has an extra $(1+u)^g(1+v)^g$ term which should not be there). The difference in the signs to the above formulas is due to different sign conventions on the definition of the $E$-polynomial, namely to the factor $(-1)^{p+q}$ that we use in \eqref{def:E-pol}.
\end{remark}

\begin{remark}
From Corollary \ref{cor:motivlog=irreg}, we know that the explicit formulas for the motives and $E$-polynomials obtained in Corollaries \ref{cor:K-class moduli rk2,3} and \ref{cor:E-poly_r=2,3} give simultaneously the motives and $E$-polynomials of the moduli spaces of logarithmic and of meromorphic connections.
\end{remark}

\subsection{Motives of maximal components of $\U(2,2)$-Higgs bundles and $\U(3,3)$-Higgs bundles}\label{sec:U(n,n)}

Given a line bundle $L$ over $X$, there is a generalized notion of $L$-twisted $G_\RR$-Higgs bundle over $X$ for any reductive real Lie group $G_\RR$; cf.~\cite{BGG06}. In this generality, $L$-twisted Higgs bundles as in Definition \ref{def:Higgsbundle} are exactly $L$-twisted $\GL(n,\CC)$-Higgs bundles (with $\GL(n,\CC)$ thought of as a real group). These are still algebraic objects over $X$, and there is a corresponding moduli space $\SM_L(G_\RR)$. The name comes from the fact that when $L=K$ and $G_\RR$ is semisimple then $\SM_K(G_\RR)$ is homeomorphic to the character variety of $X$, that is, the space of $G_\RR$-conjugacy classes of reductive representations of $\pi_1(X)$ in $G_\RR$. This homeomorphism is part of the non-abelian Hodge correspondence, which also holds for real groups; cf.~\cite{GGM12}. If $G_\RR$ is not semisimple, then there is a similar homeomorphism by considering instead an appropriate central extension of $\pi_1(X)$.
 
We are interested here in the case of $G_\RR=\U(n,n)$, the group of linear transformations of $\CC^{2n}$ preserving a Hermitian form of signature $(n,n)$.
A ($K$-twisted) $\U(n,n)$-Higgs bundle over $X$ is a tuple $(V,W,\beta,\gamma)$ where $V$ and $W$ are rank $n$ vector bundles and $\beta$ and $\gamma$ are algebraic maps $\beta:W\to V\otimes K$ and $\gamma:V\to W\otimes K$ (see \cite{BGG03}). This can be seen as a usual Higgs bundle $(E,\varphi)$ by taking $E=V\oplus W$ and $\varphi=\smtrx{0&\beta\\ \gamma&0}$. The degree of $V$ and of $W$ fix the topological types of such objects. Moreover, given $d',d''\in\ZZ$, there is a suitable semistability condition allowing the construction of the moduli space $\SM^{d',d''}(\U(n,n))$ of $\U(n,n)$-Higgs bundles over $X$ where $\deg(V)=d'$ and $\deg(W)=d''$. It turns out that $\SM^{d',d''}(\U(n,n))$ is non-empty if and only if $|d'-d''|\leq n(2g-2)$. Here, the difference $d'-d''$ is the so-called \emph{Toledo invariant}. 

The maximal situation $|d'-d''|=n(2g-2)$ is special. Indeed, let $\SM^{\max}(\U(n,n))$ denote the moduli $\SM^{d',d''}(\U(n,n))$ when $d'-d''=n(2g-2)$. Duality provides an isomorphism between $\SM^{\max}(\U(n,n))$ and the minimal Toledo moduli space, namely $d''-d'=n(2g-2)$, hence we can restrict to the maximal case. Now, if $(V,W,\beta,\gamma)\in \SM^{\max}(\U(n,n))$ then, up to isomorphism, it is of the form $(V,V\otimes K^{-1},\beta,\Id_V)$ and hence $\beta\otimes\Id_K:V\to V\otimes K^2$. So $(V,\beta\otimes\Id_K)$ is a $K^2$-twisted ($\GL(n,\CC)$-)Higgs bundle. It is shown in \cite{BGG03} that this correspondence respects (semi)stability conditions and yields an isomorphism 
\begin{equation}\label{eq:CayleyUnn}
\SM^{\max}(\U(n,n))\cong\SM^{\Dol}_{K^2}(n,d')
\end{equation}
where we recall that $\SM^{\Dol}_{K^2}(n,d')$ denotes the moduli space of $K^2$-twisted Higgs bundles of rank $n$ and degree $d'$. Since $\SM^{\Dol}_{K^2}(n,d')$ is connected, then $\SM^{\max}(\U(n,n))$ is also known as the \emph{maximal component} of the moduli space of $\U(n,n)$-Higgs bundles over $X$.

The group $\U(n,n)$ is of Hermitian type and the above correspondence is a particular instance of a so-called Cayley correspondence which holds for any such groups \cite{BGR17}. In fact, there is a  further generalized Cayley correspondence, holding for certain connected components of $G_\RR$-Higgs bundles moduli spaces, for a wider class of real groups $G_\RR$ (see \cite{BCGGO21}), and having very deep connections with higher Teichm{\"u}ller theory \cite{W18}, via the non-abelian Hodge correspondence. It is hence of great interest to study the geometry of such components, which is generally unknown, except when $G_\RR$ is a split real form for which such components are contractible spaces; see \cite{H92}. It turns out that the study we have done so far, immediately gives substantial information of the maximal components $\SM^{\max}(\U(n,n))$ for low rank and certain appropriate invariant. Namely, we obtain the motive and $E$-polynomial of $\SM^{\max}(\U(2,2))$ for $d'$ odd and $\SM^{\max}(\U(3,3))$ for $d'$ coprime to $3$.

\begin{theorem}
\label{thm:UnnMotive}
Let $\SM^{\max}(\U(n,n))$ be the moduli space of $\U(n,n)$-Higgs bundles with maximal Toledo invariant, hence of topological type $(d',d'-n(2g-2))\in\ZZ\times\ZZ$.
 Suppose $n=2$ and $d'$ is odd or $n=3$ and $d'$ coprime to $3$. Then the motive and the $E$-polynomial of $\SM^{\max}(\U(2,2))$ and $\SM^{\max}(\U(3,3))$ are, respectively, given by the corresponding formulas in Corollaries \ref{cor:K-class moduli rk2,3} and \ref{cor:E-poly_r=2,3} with $d_L=4-4g$.
\end{theorem}
\begin{proof}
Apply Corollaries \ref{cor:K-class moduli rk2,3} and \ref{cor:E-poly_r=2,3} to the trivial algebroid on $T_X^2=K^{-2}$, and use \eqref{eq:CayleyUnn}.
\end{proof}

\subsection{Computational verification of Mozgovoy's Conjecture on the motivic ADHM recursion formula in low rank}\label{sec:Mozgovoy}

In \cite[Conjecture 3]{Moz12}, Mozgovoy stated a conjectural explicit formula for the motivic classes of the moduli spaces of twisted Higgs bundles in the completion of the Gothendieck ring of Chow motives over $\CC$ with rational coefficients $\hat{K_0}(\SC\SM_{\CC})$. It is obtained by explicitly solving a motivic version to the ADHM recursion formula conjectured by Chuang, Diaconescu and Pan \cite{CDP11}. In particular, it also implies a conjectural explicit formula for the corresponding $E$-polynomials. More recently, in \cite[Theorem 1.1 and Theorem 4.6]{MG19}, Mozgovoy and O'Gorman proved an analogous formula for the number of rational points of the moduli space of twisted Higgs bundles which agrees with the predictions given by \cite[Conjecture 3]{Moz12}, up to a minor inaccuracy in one of the exponents in loc. cit. which was corrected in \cite[Theorem 1.1]{MG19}.

Our explicit formulas for the motives of the moduli spaces of twisted Higgs bundles described in Corollary \ref{cor:K-class moduli rk2,3} could then be used to verify that \cite[Conjecture 3]{Moz12} holds for motives in rank 2 and 3. Moreover, Theorem \ref{thm:equalMotiveHodge} implies that if Mozgovoy's conjecture is true in general, then his formula also gives the motives of the moduli spaces of $\SL$-connections. The first author developed a MATLAB code capable of handling and simplifying motivic expressions in the sub-$\lambda$-ring of $\hat{K_0}(\SC\SM_{\CC})$ generated by curves (see \cite{A22} for the complete description of the algorithm, the details about the computational simplification of such expressions and some explicit polynomial formulas in low rank for the motives of the moduli spaces described in this article). This has allowed us to compute explicitly and simplify the motivic expressions obtained in Corollary \ref{cor:K-class moduli rk2,3} and the ones obtained from Mozgovoy's formulas (as well as the corresponding $E$-polynomials and Betti numbers), allowing us to verify computationally that the following Conjecture holds when
\begin{itemize}
\item $2\le g\le 11$,
\item $1\le r\le 3$,
\item $d$ is coprime with $r$ and 
\item $2g-1\le -d_\SL \le 2g+18$ (i.e., $1\le p\le 20$ in Mozgovoy's notation, where $-d_\SL=2g-2+p$).
\end{itemize}
Here is the Conjecture. The corresponding notations are explained after the statement.

\begin{conjecture}{\cite[Conjecture 3]{Moz12}}
\label{conj:motivicADHM}
Let $X$ be a smooth complex algebraic curve of genus $g$ and let $[X]=1+h^1(X)+\LL\in \hat{K_0}(\SC\SM_{\CC})$. Let $\SL$ be a rank $1$ Lie algebroid of degree $-(2g-2+p)$ with $p>0$. For each integer $n\ge 1$, let
$$\SH_n(t)=\sum_{\lambda\in \SP_n} \prod_{s\in d(\lambda)} (-t^{a(s)-l(s)} \LL^{a(s)})^p t^{(1-g)(2l(s)+1)}Z(X,t^{h(s)}\LL^{a(s)}).$$
Define $H_r(t)$ for $r\ge 1$ as follows
\begin{equation*}
\begin{split}
\sum_{r\ge 1} H_r(t) T^r &= (1-t)(1-\LL t) \op{PLog} \Bigg (\sum_{n\ge 0} \SH_n(t) T^n \Bigg) \\
&=(1-t)(1-\LL t) \sum_{j\ge 1} \frac{\mu(j)}{j} \psi_j \bigg [\log \bigg ( 1+\sum_{n\ge 1}\SH_n(t) T^n \bigg) \bigg]\\
&=(1-t)(1-\LL t) \sum_{j\ge 1} \sum_{k\ge 1} \frac{(-1)^{k+1}\mu(j)}{jk}  \Bigg ( \sum_{n\ge 1} \psi_j[\SH_n(t)] T^{jn} \Bigg)^k.
\end{split}
\end{equation*}
Then $H_r(t)\in \ZZ[\LL,\lambda^1(h^1(X)),\ldots,\lambda^g(h^1(X)),t^{\pm 1}]$ and
$$[\SM_{\SL}(r,d)] = (-1)^{pr} \LL^{r^2(g-1)+p\frac{r(r+1)}{2}} H_r(1).$$
\end{conjecture}

\begin{remark}\mbox{}
\begin{enumerate}
\item The formulas from Corollary \ref{cor:K-class moduli rk2,3} hold in $\hat{K}(\VarC)$, but Mozgovoy's formulas were originally stated in $\hat{K_0}(\SC\SM_{\CC})$. The polynomial expressions shown in \cite{A22} were obtained from the simplification of the equations in Corollary \ref{cor:K-class moduli rk2,3} and, therefore, still hold in $\hat{K}(\VarC)$, but in order to simplify Mozgovoy's formulas, the algorithm used in \cite{A22} assumes that the motivic ring admits division by integers (i.e., that the additive group of motives is torsion free). This is why we can only state that Conjecture \ref{conj:motivicADHM} has been verified in $\hat{K_0}(\SC\SM_{\CC})$ and not in $\hat{K}(\VarC)$.
\item This is a slightly modified version of \cite[Conjecture 3]{Moz12} which includes an exponent correction and the explicit logarithmic form from \cite[Theorem 1.1]{MG19} as well as a motivic version of the polynomial dependency on the generators asserted at \cite[Theorem 1.1 and Theorem 4.6]{MG19} ($H_r(t)\in \ZZ[\LL,\lambda^1(h^1(X)),\ldots,\lambda^g(h^1(X)),t^{\pm 1}]$) which was not originally stated for motives at \cite[Conjecture 3]{Moz12}.
\item In virtue of Theorem \ref{thm:equalMotiveHodge}, the Conjecture is stated for Lie algebroid connections instead of twisted Higgs bundles (as it is stated in \cite{Moz12}).
\end{enumerate}
\end{remark}

In the previous Conjecture, $\SP_n$ denotes the set of partitions of $n$. A partition $\lambda\in \SP_n$ is considered as a decreasing sequence of integers $\lambda=(\lambda_1\ge \lambda_2 \ge \cdots \ge \lambda_k>0)$ with $\sum_{i=1}^k \lambda_i=n$ and
$$d(\lambda)=\{(i,j)\in \ZZ^2 | 1\le i, \, 1\le j\le \lambda_i\},$$
$$a(i,j)=\lambda_i-j, \quad \quad l(i,j)=\lambda_j'-i=\max\{l|\lambda_l\ge j\}-i, \quad \quad h(i,j)=a(i,j)+l(i,j)+1.$$
Then, $\psi_j:\hat{K_0}(\SC\SM_{\CC})\to \hat{K_0}(\SC\SM_{\CC})$ denotes the $j$-th Adams operation on $\hat{K_0}(\SC\SM_{\CC})$. At this point, we would like to make a clarification about the $\lambda$-ring structures on $\hat{K_0}(\SC\SM_{\CC})$, as there are two different opposite $\lambda$-ring structures (c.f. \cite{H07}) involved in the formula. On one hand, the space of motives $\hat{K_0}(\SC\SM_{\CC})$ admits the natural $\lambda$-ring structure that we used at the beginning of the section given by
$$\lambda^j([Y])=[\Sym^j([Y]).$$
This is the one used to define the motivic zeta function $Z(X,t)=\sum_{j\ge 0} \lambda^j([X]) t^j$.
However, we can also define an alternative $\lambda$-ring structure $\sigma^j$ on motives as follows,
$$\sum_{j\ge 0} \sigma^j([Y])t^j = \Bigg( \sum_{j\ge 0} \lambda^j([Y]) (-t)^j \Bigg)^{-1}.$$
In \cite{H07} (see also \cite[\S 8]{LL04}) it is proved that $\sigma$ (contrary to $\lambda$) makes $\hat{K_0}(\SC\SM_{\CC})$ a special $\lambda$-ring, meaning that there are universal polynomials $P_j\in \ZZ(X_1,\ldots,X_j,Y_1,\ldots,Y_j)$ and $P_{n,m}\in \ZZ(X_1,\ldots,X_{nm})$ such that
\begin{equation*}
\begin{split}
\sigma^j([Y][Z])&=P_j\left (\sigma^1([Y]),\ldots,\sigma^j([Y]),\sigma^1([Z]),\ldots,\sigma^j([Z]) \right)\\
\sigma^n(\sigma^m([Y]))&=P_{n,m} \left (\sigma^1([Y]),\ldots,\sigma^{nm}([Y]) \right).
\end{split}
\end{equation*}
Then, we can compute recursively the Adams operations $\psi_n: \hat{K_0}(\SC\SM_{\CC}) \to \hat{K_0}(\SC\SM_{\CC})$ for $j\ge 1$ as follows (c.f. \cite{Hop06}),
$$\psi_1(X)=X$$
$$\psi_n(X)=(-1)^{n-1}n\sigma^n([Y])-\sum_{j=1}^{n-1} (-1)^{n-j} \sigma^{n-j}([Y]) \psi_j([Y]) .$$
As $\sigma$ is special, these Adams operations become ring homomorphisms and, moreover, $\psi_n\circ \psi_m=\psi_{nm}$. The Adams operations are then extended to $\hat{K}(\VarC)((t))$ as ring homomorphisms with $\psi_j(t)=t^j$. Finally, the operator $\op{PLog}$ in the formula is the Plethystic logarithm, for such Adams operations (for the opposite structure $\sigma$), defined as
$$\op{PLog}(A)=\sum_{j\ge 1} \frac{\mu(j)}{j} \psi_k[\log(A)].$$

\section{Motives of moduli spaces of $\SL$-connections with fixed determinant}
\label{section:fixedDeterminant}

So far we have considered moduli spaces of flat $\SL$-connections (or $\Lambda$-modules) with fixed rank and degree, but the techniques presented in the previous sections also allow us to obtain analogues of the previous results for moduli spaces of flat $\SL$-connections with fixed determinant. In this section we will define the moduli space of flat $\SL$-connections with fixed determinant and show several results proving the invariance of its motivic class and $E$-polynomial regarding the Lie algebroid structure, providing the necessary changes in the previously exposed arguments to treat the fixed determinant scenario.

\subsection{Twisted Higgs bundles with fixed determinant}

Let $\xi$ be an algebraic line bundle over $X$ of degree $d$ coprime with $r$. Let $\SM_L^{\Dol}(r,\xi)\subset \SM_L(r,d)$ be the moduli space of traceless $L$-twisted Higgs bundles with fixed determinant $\xi$, i.e., the moduli space of pairs $(E,\varphi)$ with $\deg(E)\cong \xi$ and $\varphi\in H^0(\End_0(E)\otimes L)$, where $\End_0(E)$ denotes the endomorphisms of $E$ with trace $0$.

By \cite[Theorem 1.2]{BGL11}, if $g\ge 2$ and $\deg(L)>2g-2$,  $\SM_L^{\Dol}(r,\xi)$ is a smooth irreducible $\CC^*$-invariant closed subvariety of $\SM_L^{\Dol}(r,d)$, so it is a smooth semiprojective variety. Hence the digression held about the structure of the Bialynicki-Birula stratification of $\SM_L^{\Dol}(r,d)$ applies to $\SM_L^{\Dol}(r,\xi)$ as well. The fixed-point locus of the $\CC^*$-action on $\SM_L^{\Dol}(r,\xi)$ clearly corresponds to the intersection of the fixed-point locus of the $\CC^*$-action on $\SM_L^{\Dol}(r,d)$ with $\SM_L^{\Dol}(r,\xi)$. In section \ref{section:HiggsPolynomials} we described a decomposition of the fixed-point locus as
$$\SM_L^{\Dol}(r,d)^{\CC^*}=\bigcup_{\begin{array}{c}
{\scriptstyle (\overline{r},\overline{d})\in\Delta_L}\\
{\scriptstyle |\overline{r}|=r,\,|\overline{d}|=d}
\end{array}} \VHS_L(\overline{r},\overline{d}),$$
thus we have a decomposition
$$\SM_L^{\Dol}(r,\xi)^{\CC^*}=\bigcup_{\begin{array}{c}
{\scriptstyle (\overline{r},\overline{d})\in\Delta_L}\\
{\scriptstyle |\overline{r}|=r,\,|\overline{d}|=d}
\end{array}} \VHS_L(\overline{r},\overline{d}, \xi),$$
where $\VHS_L(\overline{r},\overline{d},\xi)=\VHS_L(\overline{r},\overline{d}) \cap \SM_L^{\Dol}(r,\xi)$.
By construction all variations of Hodge structure have traceless Higgs fields, so 
$$\VHS_L(\overline{r},\overline{d},\xi)=\left \{(E_\bullet,\varphi_\bullet)\in \VHS_L(\overline{r},\overline{d}) \,  \middle | \,  \bigotimes_{i=1}^k \det(E_i) \cong \xi \right\}.$$
On the other hand, we can consider algebraic chains with fixed ``total determinant'' in the following sense. For $\overline{r}$ and $\overline{d}$ such that $|\overline{r}|=r$ and $|\overline{d}|=d$, define
$$\HC^\alpha(\overline{r},\overline{d},\xi) = \left \{(E_\bullet,\varphi_\bullet)\in \HC^\alpha(\overline{r},\overline{d})\,  \middle | \,  \bigotimes_{i=1}^k \det(E_i) \cong \xi \right\}.$$

\begin{lemma}
\label{lemma:VHSChainsFixedDet}
Given $\overline{r}=(r_1,\ldots,r_k)$ and a degree $d$ line bundle $\xi$, consider the line bundle
$$\xi_L= \xi \otimes L^{\otimes ( \sum_{i=1}^k (i-k)r_i)}.$$
Then the isomorphism described in Corollary \ref{cor:VHSChains} induces an isomorphism
$$\VHS_L(\overline{r},\overline{d},\xi) \cong \HC^{\alpha_L}(\overline{r}, \overline{d}_L,\xi_L).$$
\end{lemma}

\begin{proof}
Given $(E_\bullet,\varphi_\bullet)\in \VHS_L(\overline{r},\overline{d})$, the underlying bundles of its corresponding algebraic chain are $\tilde{E}_i=E_i\otimes L^{i-k}$.
So if $(E_\bullet,\varphi_\bullet)\in \VHS_L(\overline{r},\overline{d},\xi)$, then
$$\bigotimes_{i=1}^k \det(\tilde{E}_i) =  \bigotimes_{i=1}^k \left (\det(E_i) \otimes L^{(i-k)r_i} \right)= \left ( \bigotimes_{i=1}^k \det(E_i) \right) \otimes L^{\otimes ( \sum_{i=1}^k (i-k)r_i)} \cong \xi_L.$$
The converse is analogous.
\end{proof}

\begin{lemma}
\label{lemma:ChainsFixedDetIso}
Fix $\overline{r}=(r_1,\ldots,r_k)$ and $\overline{d}=(d_1,\ldots,d_k)$. Let $\xi$ and $\xi'$ be two line bundles over $X$ of degree $\sum_{i=1}^k d_i$. Then
$$\HC^\alpha(\overline{r},\overline{d},\xi) \cong \HC^\alpha(\overline{r},\overline{d},\xi').$$
\end{lemma}

\begin{proof}
Let $r=\sum_{i=1}^k r_i$. As $\xi$ an $\xi'$ have the same degree, $\xi'\otimes \xi^{-1}$ has degree zero, so there exists a line bundle $\psi$ such that $\psi^{\otimes r}\cong \xi'\otimes \xi^{-1}$. Given $(E_\bullet,\varphi_\bullet)\in \HC^\alpha(\overline{r},\overline{d},\xi)$, consider the algebraic chain
$$(E_\bullet\otimes \psi, \varphi_\bullet \otimes \id_{\psi}).$$
As $\psi$ has degree zero, tensoring by $\psi$ gives an $\alpha$-slope-preserving correspondence between subchains of $(E_\bullet,\varphi_\bullet)$ and those of $(E_\bullet\otimes \psi,\varphi_\bullet\otimes \id_\psi)$. Thus, $(E_\bullet\otimes \psi,\varphi_\bullet\otimes \id_\psi)$ is $\alpha$-(semi)stable and clearly
$$\bigotimes_{i=1}^k\det(E_\bullet\otimes \psi) = \bigotimes_{i=1}^k ( \det(E_i) \otimes \psi^{r_i}) = \left( \bigotimes_{i=1}^k \det(E_i) \right) \otimes \psi^{\sum_{i=1}^r r_i} \cong \xi \otimes \psi^r \cong \xi',$$
and then tensorization by $\psi$ yields the desired isomorphism.
\end{proof}

\begin{corollary}
\label{cor:isomorphicVHSSL}
Let $\xi$ and $\xi'$ be line bundles on $X$ of degree $d$. Let $L$ and $L'$ be line bundles with $\deg(L)=\deg(L')$. Then
$$\VHS_L(\overline{r},\overline{d},\xi) \cong \VHS_{L'}(\overline{r},\overline{d},\xi')$$
\end{corollary}

\begin{proof}
This follows from Lemmas \ref{lemma:VHSChainsFixedDet} and \ref{lemma:ChainsFixedDetIso}, using the fact that $\alpha_L=\alpha_{L'}$, $\overline{d}_L=\overline{d}_{L'}$.
\end{proof}

Finally, consider the Bialynicki-Birula decomposition of $\SM_L^{\Dol}(r,\xi)$
$$\SM_L^{\Dol}(r,\xi)=\bigcup_{\begin{array}{c}
{\scriptstyle (\overline{r},\overline{d})\in\Delta_L}\\
{\scriptstyle |\overline{r}|=r,\,|\overline{d}|=d}
\end{array}} U_{\overline{r},\overline{d},\xi}^+$$
where, clearly, $U_{\overline{r},\overline{d},\xi}^+=U_{\overline{r},\overline{d},\xi}\cap \SM_L^{\Dol}(r,\xi)$.
We know that $U_{\overline{r},\overline{d},\xi}^+\to \VHS_L(\overline{r},\overline{d},\xi)$ is an affine bundle of rank $N_{L,\overline{r},\overline{d},\xi}^+$ and, therefore, we have the analogue of equation \eqref{eq:EPolHiggs},
\begin{equation}
\label{eq:EPolHiggsSL}
[\SM_L^{\Dol}(r,\xi)] = \sum_{\begin{array}{c}
{\scriptstyle (\overline{r},\overline{d})\in\Delta_L}\\
{\scriptstyle |\overline{r}|=r,\,|\overline{d}|=d}
\end{array}} \LL^{N_{L,\overline{r},\overline{d},\xi}^+} [\VHS_L(\overline{r},\overline{d},\xi)],
\end{equation}
which can be used in an analogous way to prove the following invariance property of the motivic class of the moduli space of $L$-twisted Higgs bundles with fixed determinant.

\begin{theorem}
\label{thm:equalMotiveHiggsSL}
Let $X$ be a smooth projective curve of genus $g\geq 2$. Let $L$ and $L'$ be line bundles over $X$ such that $\deg(L)=\deg(L')>2g-2$. Assume that $\xi$ and $\xi'$ are line bundles of degree $d$ coprime with the rank $r$. Then the motives of the corresponding moduli spaces $[\SM_L^{\Dol}(r,\xi)]$ and $[\SM_{L'}^{\Dol}(r,\xi')]$ are equal in $K(\VarC)$. Moreover, if $d''$ is any integer coprime with $r$, and $\xi''$ is any line bundle of degree $d''$, then $E(\SM_L^{\Dol}(r,\xi))=E(\SM_{L'}^{\Dol}(r,\xi''))$. 
\end{theorem}

\begin{proof}
The argument is completely analogous to the one which lead to Theorem \ref{thm:equalMotive}. One just has to use the corresponding fixed determinant versions of the objects involved. Note that the appropriate deformation complex
is $C_0^\bullet(E,\varphi) \, : \, \End_0(E) \stackrel{[-,\varphi]}{\to} \End_0(E)\otimes L$, which decomposes as $C_0^\bullet(E,\varphi)=\bigoplus_{l=-k+1}^{k-1}C_{0,l}^\bullet(E,\varphi)$, just like $C^\bullet(E,\varphi)$ in \eqref{eq:complx-decomp}, but $C_{0,l}^\bullet(E,\varphi)=C_l^\bullet(E,\varphi)$, so the Morse index for the fixed determinant case equals the non-fixed determinant case. 
Moreover, the equality of the $E$-polynomials is also precisely the same argument, but here one has to refer to \cite[Theorem 0.5]{MS20-endoscopy} (see also \cite[Corollary 7.17]{GWZ20}), instead of the references stated in the proof of Theorem \ref{thm:equalMotive}. The details are left to the reader.

\end{proof}

\subsection{Moduli spaces of $\SL$-connections with fixed determinant}

Let $\SL$ be any Lie algebroid. Let $E$ be a rank $r$ vector bundle with determinant $\xi=\det(E)=\Lambda^rE$ and let $\nabla_\SL:E\to E\otimes \Omega_\SL^1$ be an integrable $\SL$-connection on $E$. Then $\nabla_\SL$ induces a map
$$\tr(\nabla_\SL): \xi \longrightarrow \xi\otimes \Omega_\SL^1,$$
defined as follows. For local sections  $s_1,\ldots,s_r$ of $E$,
$$\tr(\nabla_\SL)(s_1\wedge \ldots \wedge s_r)= \sum_{i=1}^r s_1\wedge \ldots \wedge \nabla_\SL(s_i) \wedge \ldots \wedge s_r.$$

Observe that if $v_1,\ldots,v_r$ is a local trivialising basis of $E$ over some open subset of $X$, then if we write $\nabla_\SL$ in that basis as $\nabla_\SL=d_\SL+G$ with $G=(g_{ij})$, we get
$$\tr(\nabla_\SL)(v_1\wedge \ldots \wedge v_r) = \sum_{i=1}^r v_1\wedge \ldots \wedge Gv_i \wedge \ldots \wedge v_r = \sum_{i=1}^r v_1\wedge \ldots \wedge g_{ii} v_i \wedge \ldots \wedge v_r = \tr(G) v_1\wedge \ldots \wedge v_r,$$
justifying the notation ``$\tr(\nabla_\SL)$''. 

\begin{lemma}
\label{lemma:traceMap}
Let $X$ be a smooth projective curve and let $\SL$ be a Lie algebroid on $X$. Let $(E,\nabla_\SL)$ be an integrable $\SL$-connection with $\det(E)\cong \xi$. Then $\tr(\nabla_\SL)$ is an integrable $\SL$-connection on $\xi$. 
\end{lemma}

\begin{proof}
It is clear by construction that $\tr(\nabla_\SL)$ is $\CC$-linear, so we need to prove that it satisfies the Leibniz rule and that it is integrable. Let $s_1,\ldots s_r$ be local sections $E$ and $f$ a local algebraic function on $X$. Then, for each $j=1,\ldots,r$, we have
\begin{equation*}
\begin{split}
\tr(\nabla_\SL) (s_1\wedge \ldots \wedge f s_j \wedge \ldots \wedge s_r) &= \sum_{i\ne j} s_1\wedge \ldots \wedge \nabla_\SL(s_i)\wedge \ldots \wedge fs_j \wedge \ldots \wedge s_r\\
& \ \ \ + s_1\wedge \ldots \wedge f \nabla_\SL(s_j) \wedge \ldots \wedge s_r + s_1 \wedge \ldots \wedge s_j \otimes d_\SL(f) \wedge \ldots \wedge s_r\\
&=f \sum_{i=1}^r s_1\wedge \ldots \wedge \nabla_\SL(s_i) \wedge \ldots \wedge s_r + s_1\wedge \ldots \wedge s_r \otimes d_\SL(f) \\
&= f\tr(\nabla_\SL)(s_1\wedge \ldots \wedge s_r) + s_1\wedge \ldots \wedge s_r \otimes d_\SL(f),
\end{split}
\end{equation*}
so $(\xi,\tr(\nabla_\SL))$ is an $\SL$-connection. 

Let us now prove that it is integrable. Suppose $\SL=(V,[\cdot\,,\cdot],\delta)$, with $\rk(V)=k$. We will prove it via a local representation of $\tr(\nabla_\SL)$. Let $U\subset X$ be an open subset such that $E$ and $V$ are trivial bundles over $U$. Write $\nabla_\SL$ locally over $U$ as $\nabla_\SL=d_\SL+G$, where $G$ is an $V^*$-valued $r\times r$ matrix. Let $w_1,\ldots,w_k$ be a trivialising basis of $V^*$ over $U$. Then we can write
$$G=\sum_{i=1}^k G_i \otimes w_i$$
where $G_i$ is an $\SO_X(U)$-valued matrix. Now, we have that, over $U$, the curvature of $\nabla_\SL$ translates into
\begin{equation}
\label{eq:traceMap1}
\nabla_\SL^2= d_\SL+G\wedge G  = d_\SL(G)+ \sum_{i,j=1}^k G_iG_j \otimes w_i\wedge w_j = d_\SL(G) + \sum_{i<j} [G_i,G_j] \otimes w_i\wedge w_j,
\end{equation}
and, since $\tr(\nabla_\SL)= d_\SL+\tr(G)=  d_\SL+\sum_{i=1}^k \tr(G_i) \otimes w_i$, 
$$\tr(\nabla_\SL)^2 =  d_\SL(\tr(G))+ \sum_{i<j} (\tr(G_i)\tr(G_j)-\tr(G_j)\tr(G_i)) \otimes w_i\wedge w_j=d_\SL(\tr(G)) = \tr(d_\SL(G)).$$
From the integrability of $\nabla_\SL$, it follows from \eqref{eq:traceMap1} that $d_\SL(G) = -\sum_{i<j} [G_i,G_j] \otimes w_i\wedge w_j$, hence
$$\tr(d_\SL(G)) = -\sum_{i<j} \tr([G_i,G_j]) \otimes w_i\wedge w_j=0,$$
proving that $(\xi,\tr(\nabla_\SL))$ is integrable.
\end{proof}

Let $(E,\nabla_\SL)\in \SM_{\SL}(r,d)$. As $\xi$ is a line bundle,  $(\xi,\tr(\nabla_\SL))$ is automatically stable, so $(\xi,\tr(\nabla_\SL))\in \SM_{\SL}(1,d)$. As the determinant construction can be clearly done in families, then it defines the following map,
$$\det: \SM_{\SL}(r,d)\longrightarrow \SM_{\SL}(1,d),\ \ \ \det(E,\nabla_\SL)=(\det(E),\tr(\nabla_\SL)).$$

Let $(\xi,\delta)\in \SM_{\SL}(1,d)$ be an integrable $\SL$-connection of rank $1$ and degree $d$. Define
$$\SM_{\SL}(r,\xi,\delta) = {\det}^{-1}(\xi,\delta) \subset \SM_{\SL}(r,d)$$
as the moduli space of $\SL$-connections with fixed determinant $(\xi,\delta)$.

For example, if $\SL=(L,0,0)$is trivial, with $L$ a line bundle, then
$$\SM_{(L,0,0)}(r,\xi,0) = \SM_{L^{-1}}^{\Dol}(r,\xi).$$
If $\ST_X$ is the canonical Lie algebroid on $X$, then $\SM_{\ST_X}(r,\SO_X,0)$ is the moduli space of $\SSL(r,\CC)$-connections on $X$.

The determinant map extends to the $\SL$-Hodge moduli space, obtaining a map
$$\det: \SM_{\SL}^{\Hod}(r,d) \longrightarrow \SM_{\SL}^{\Hod}(1,d),$$
over $\CC$, by taking $\det(E,\nabla_\SL,\lambda) = (\det(E),\tr(\nabla_\SL),\lambda)$. Moreover, this map is $\CC^*$-equivariant for action \eqref{eq:C*-action-LHodgemod}. For each $(\xi,\delta,\lambda)\in \SM_{\SL}^{\Hod}(1,d)$, define
$$\SM_{\SL}^{\Hod}(r,\xi,\delta) = {\det}^{-1}(\CC \cdot (\xi,\delta,\lambda))$$
as the \emph{$\SL$-Hodge moduli space with fixed determinant} $(\xi,\delta)$. Here $\CC \cdot (\xi,\delta,\lambda)$ denotes the closure of the $\CC^*$-orbit of $(\xi,\delta,\lambda)$ in $ \SM_{\SL}^{\Hod}(1,d)$ which, since $\xi$ is a line bundle, is just the set of elements of the form $(\xi,t\delta,t\lambda)$, with $t\in\CC$. Then $\SM_{\SL}^{\Hod}(r,\xi,\delta)$ is clearly a $\CC^*$-invariant closed subvariety of the $\SL$-Hodge moduli space $\SM_{\SL}^{\Hod}(r,d)$ and if $\pi: \SM_{\SL}^{\Hod}(r,\xi,\delta) \to \CC$ is the restriction of the map \eqref{eq:project->lambda}, we have
\begin{itemize}
\item $\pi^{-1}(0) \cong \SM_{L^{-1}}^{\Dol}(r,\xi)$;
\item $\pi^{-1}(1) \cong \SM_{\SL}(r,\xi,\delta)$;
\item $\pi^{-1}(\CC^*)\cong \SM_{\SL}(r,\xi,\delta)\times \CC^*$.
\end{itemize}

The deformation theory for this moduli space is very similar to the deformation for the moduli space of $\SL$-connections with fixed degree computed in \cite[Theorem 47]{Tor11}.

\begin{lemma}
\label{lemma:deformationSL}
The Zariski tangent space to the moduli space $\SM_{\SL}(r,\xi,\delta)$ at a stable point $(E,\nabla_\SL)$ is isomorphic to $\HH^1(C_0^\bullet(E,\nabla_\SL))$,  where $C_0^\bullet(E,\nabla_\SL)$ is the complex
$$C_0^\bullet(E,\nabla_\SL) \, : \, \End_0(E) \stackrel{[-,\nabla_\SL]}{\longrightarrow} \End_0(E)\otimes \Omega_\SL^1\stackrel{[-,\nabla_\SL]}{\longrightarrow} \ldots \stackrel{[-,\nabla_\SL]}{\longrightarrow} \End_0(E)\otimes \Omega_\SL^{\rk(\SL)},$$
and the obstruction for the deformation theory lies in $\HH^2(C_0^\bullet(E,\nabla_\SL))$.
\end{lemma}

\begin{proof}
The deformations of $\SM_{\SL}(r,\xi,\delta)$ are precisely the deformations of $\SM_{\SL}(r,d)$ which preserve the determinant and trace. Following the same notation as the one used in Lemma \ref{lemma:HodgeTangent}, let $\SU=\{U_\alpha\}$ be a covering of $X$ such that $E$ is trivial over $U_\alpha$. Fix a trivialization of $E$ over $\SU$ and for each $\alpha$ and $\beta$, let $g_{\alpha\beta}:U_{\alpha\beta} \to \GL(r,\CC)$ be the transition functions for $E$ and let $\nabla_{\SL,\alpha}= d_\SL + G_\alpha$ be the local representation of $\nabla_\SL$ over $U_\alpha$.

Let $(E',\nabla_\SL')$ be a deformation of $(E,\nabla_\SL)$ over $X\times \Spec(\CC[\varepsilon]/\varepsilon^2)$ such that the transition functions of $E'$ are
$$g_{\alpha\beta}'=g_{\alpha\beta}+\varepsilon g_{\alpha\beta}^1\ \ \ \text{ and }\ \ \ \nabla_{\SL,\alpha}' = \varepsilon d_\SL + G_\alpha + \varepsilon G_\alpha^1.$$
By \cite[Theorem 47]{Tor11}, $g_{\alpha\beta}'$ and $G_\alpha^1$ correspond to a deformation of $(E,\nabla_\SL)$ in $\SM_{\SL}(r,d)$ if and only if the cocycles $c\in C^1(\SU,\End(E))$ and $C\in C^0(\SU,\End(E)\otimes \Omega_\SL^1)$ defined by $c_{\alpha\beta}^{(\alpha)} = g_{\alpha\beta}^1 g_{\beta\alpha}$ and $C_\alpha^{(\alpha)} = G_\alpha^1$ satisfy $\partial c= 0$, $\partial C = \tilde{\nabla}_\SL c$ and $\tilde{\nabla}_\SL C =0$.

We will prove that $(c,C)$ defines a deformation of $(E,\nabla_\SL)$ in $\SM_{\SL}(r,\xi,\delta)$ if and only if $(c,C)\in C^1(\SU,\End_0(E)) \times C^0(\SU,\End_0(E)\otimes \Omega_\SL^1)$. 
We have $\tr(\nabla_\SL')=\delta=\tr(\nabla_\SL)$ if and only if
$$\tr(G_\alpha)=\tr(G_\alpha^1)=\tr(G_\alpha)+\varepsilon \tr(G_\alpha^1),$$
so $\tr(G_\alpha^1)=0$ and, therefore, $C\in C^1(\SU,\End_0(E)\otimes \Omega_\SL^1)$.
On the other hand, $\det(E')=\xi$ if and only if
$$\det(g_{\alpha\beta})=\det(g_{\alpha\beta}^1) = \det(g_{\alpha\beta}+\varepsilon g_{\alpha\beta}^1 )\\
=\det(g_{\alpha\beta}) + \varepsilon \sum_{i=1}^r \det \left ( (g_{\alpha\beta})_1 | \cdots | (g_{\alpha\beta}^1)_i | \cdots | (g_{\alpha\beta})_r \right),$$
since $\varepsilon^2=0$. Here, $( (g_{\alpha\beta})_1 | \cdots | (g_{\alpha\beta}^1)_i | \cdots | (g_{\alpha\beta})_r )$ denotes the $r\times r$ matrix whose  $i$-th column equals the $i$-th column of $g_{\alpha\beta}^1$ and the other columns are the corresponding ones of $g_{\alpha\beta}$. Set
$$D_i=\det \left ( (g_{\alpha\beta})_1 | \cdots | (g_{\alpha\beta}^1)_i | \cdots (g_{\alpha\beta})_r \right)$$ so that $\det(E')=\xi$ if and only if $\sum_{i=1}^r D_i=0$.
 Let $A$ be such that $g_{\alpha\beta} A = g_{\alpha\beta}^1$.
By Cramer's rule, $D_i = A_{ii} \det(g_{\alpha\beta})$, thus 
$$\sum_{i=1}^r D_i= \det(g_{\alpha\beta}) \tr(A) = \det(g_{\alpha\beta}) \tr( g_{\alpha\beta}^{-1}g_{\alpha\beta}^1)= \det(g_{\alpha\beta}) \tr(c_{\alpha\beta}^{(\alpha)}),$$
and so $\det(E')=\xi$ if and only if $\tr(c)=0$, i.e., if $c\in C^1(\SU,\End_0(E))$.

The rest of the proof is exactly the same as the one of \cite[Theorem 47]{Tor11}, noting that the Zariski tangent space is then $\HH^1(C_0^\bullet(E,\nabla_\SL))/\Aut(E,\nabla_\SL)=\HH^1(C_0^\bullet(E,\nabla_\SL))$ since $\Aut(E,\nabla_\SL)$ is trivial by stability of $(E,\nabla_\SL)$ (the scalar automorphisms that appeared at the end of the proof of Lemma \ref{lemma:HodgeTangent} do not appear in this case because the determinant of $E$ is fixed).
\end{proof}

\begin{proposition}
\label{thm:semiprojectiveHodgeSL}
Let $X$ be a smooth projective curve of genus $g\ge 2$. Let $\SL$ be a Lie algebroid with $\rk(\SL)=1$ and $\deg(\SL)<2-2g$. Take $r$ and $d$ is coprime. Then, for each $(\xi,\delta)\in \SM_{\SL}(1,d)$, the moduli space $\SM_{\SL}^{\Hod}(r,\xi,\delta)$ is a smooth semiprojective variety for the $\CC^*$-action $t\cdot (E,\nabla_\SL,\lambda)=(E,t\nabla_\SL,t\lambda)$. Furthermore, the map $\pi:\SM_{\SL}^{\Hod}(r,\xi,\delta)\to\CC$, $\pi(E,\nabla_\SL,\lambda)=\lambda$ is a surjective submersion and $\SM_{\SL}(r,\xi,\delta)$ is a smooth variety of dimension $\deg(\SL)(1-r^2)$.
\end{proposition}

\begin{proof}
The argument is exactly the same as the one carried on in section \ref{section:semiprojectivityHodge}. The only difference is that the computation of the dimension of the tangent bundle done in Lemma \ref{lemma:dimensionTangent} now becomes
$$\dim T_{(E,\nabla_\SL)} \SM_{\SL}(r,\xi,\delta) = \deg(\SL)(1-r^2)+\dim \left( \HH^2(C_0^\bullet(E,\nabla_\SL)) \right).$$
Notice that here we have to take trace-free endomorphisms, hence by point (3) of Lemma \ref{lemma:stableHom}, $\HH^0(C_0^\bullet(E,\nabla_\SL))=0$.
Taking into account Lemma \ref{lemma:deformationSL}, the deformation theory computed in Lemma \ref{lemma:HodgeTangent} becomes now
$$T_{(E,\nabla_\SL,0)}\SM_{\SL}^{\Hod}(r,\xi,\delta) \cong \frac{\left \{(c,C,\lambda_\varepsilon) \in \left( \begin{array}{l}
C^1(\SU,\End(E)) \times \\
C^0 (\SU,\End(E)\otimes \Omega_\SL) \times \CC 
\end{array} \right) \middle | \begin{array}{l}
\partial c= 0\\
\partial C = \tilde{\nabla}_\SL c + \lambda_\varepsilon \omega\\
\tilde{\nabla}_\SL C = -\lambda_\varepsilon d_\SL(\nabla_\SL)\\
\tr(c)=0\\
\tr(C)=\lambda_\varepsilon \delta
\end{array} \right \}}{ \left \{ (\partial \eta ,\tilde{\nabla}_\SL \eta, 0) \, \middle | \,  \eta\in C^0(\SU,\End_0(E)) \right \}}$$
with $d\pi([(c,C,\lambda_\varepsilon)])=\lambda_\varepsilon$. Then, clearly
$$\ker d\pi \cong  \frac{\left \{(c,C,0) \in \left( \begin{array}{l}
C^1(\SU,\End_0(E)) \times \\
C^0 (\SU,\End_0(E)\otimes \Omega_\SL) \times \CC 
\end{array} \right) \middle | \begin{array}{l}
\partial c= 0\\
\partial C = \tilde{\nabla}_\SL c + \lambda_\varepsilon \omega\\
\tilde{\nabla}_\SL C = -\lambda_\varepsilon d_\SL(\nabla_\SL)
\end{array} \right \}}{ \left \{ (\partial \eta ,\tilde{\nabla}_\SL \eta, 0) \, \middle | \,  \eta\in C^0(\SU,\End_0(E)) \right \}} \cong T_{(E,\nabla_\SL)} \SM_{L^{-1}}^{\Dol}(r,\xi)$$
and the proof proceeds exactly as in Lemma \ref{lemma:smoothHodgeModuli} and Theorem \ref{thm:semiprojectiveHodge}.
\end{proof}

\begin{theorem}
\label{thm:equalMotiveHodgeSL}
Let $\SL=(L,[\cdot\,,\cdot],\delta)$ be Lie algebroid on $X$ such that $L$ is a line bundle with $\deg(L)<2-2g$. If $r$ and $d$ are coprime, then, for each $(\xi,\delta)\in \SM_{\SL}(1,d)$, we have
$$[\SM_{\SL}(r,\xi,\delta)] = [\SM_{L^{-1}}^{\Dol}(r,\xi,\delta)],\ \ \ [\SM_{\SL}^{\Hod}(r,\xi,\delta)] = \LL[\SM_{L^{-1}}^{\Dol}(r,\xi,\delta)]$$
and we have an isomorphism of Hodge structures
$$H^\bullet(\SM_{\SL}(r,\xi,\delta)) \cong H^\bullet(\SM_{L^{-1}}^{\Dol}(r,\xi,\delta))$$
In particular,
$$E(\SM_{\SL}(r,\xi,\delta)) = E(\SM_{L^{-1}}^{\Dol}(r,\xi,\delta)),\ \ \ E(\SM_{\SL}^{\Hod}(r,\xi,\delta)) = uv E(\SM_{L^{-1}}^{\Dol}(r,\xi,\delta)).$$
Moreover, both $\SM_{\SL}(r,\xi,\delta)$ and $\SM_{\SL}^{\Hod}(r,\xi,\delta)$ have pure mixed Hodge structures.
\end{theorem}

\begin{proof}
The proof is completely analogous to that of Theorem \ref{thm:equalMotiveHodge}. The details are left to the reader.
\end{proof}

Combining this result with Theorem \ref{thm:equalMotiveHiggsSL} and working analogously to Theorem \ref{thm:equalMotive-Chow-Voevodsky}, yields the fixed-determinant version of Theorems \ref{thm:equalMotive} and \ref{thm:equalMotive-Chow-Voevodsky}.

\begin{theorem}
\label{thm:equalMotiveSL}
Let $X$ be a smooth projective curve of genus $g\ge 2$ and let $\SL$ and $\SL'$ be any Lie algebroids on $X$ such that $\rk(\SL)=\rk(\SL')=1$ and $\deg(\SL)=\deg(\SL')<2-2g$. Suppose that $r$ and $d$ are coprime. Let $(\xi,\delta)\in \SM_{\SL}(1,d)$ and $(\xi',\delta')\in \SM_{\SL'}(1,d)$. Then
$$\SI(\SM_{\SL}(r,\xi,\delta)) = \SI(\SM_{\SL'}(r,\xi',\delta'))$$
where $\SI(X)$ denotes one of the following:
\begin{enumerate}
\item The motivic class $[X]\in \hat{K}(\VarC)$;
\item The Voevodsky motive $M(X)\in \op{DM}^{\op{eff}}(\CC,R)$ for any ring $R$. In this case, moreover, the motives are pure;
\item The Chow motive $h(X)\in \op{Chow}^{\op{eff}}(\CC,R)$ for any ring $R$;
\item The Chow ring $\op{CH}^\bullet(X,R)$ for any ring $R$.
\end{enumerate}
Moreover, the mixed Hodge structures of the moduli spaces are pure and if $d''$ is any integer coprime with $r$ and $(\xi'',\delta'')\in \SM_{\SL'}(1,d'')$, then 
$$E(\SM_{\SL}(r,\xi,\delta)) = E(\SM_{\SL'}(r,\xi'',\delta'')).$$
Finally, if $L=L'=T_X(-D)$ for some effective divisor $D$, then we have an actual isomorphism of pure mixed Hodge structures
$$H^\bullet(\SM_{\SL}(r,\xi,\delta)) \cong  H^\bullet(\SM_{\SL'}(r,\xi'',\delta'')).$$
\end{theorem}

Finally, we have the analogue of Theorem \ref{thm:homotopygrps} for the fixed determinant case. The proof is exactly the same, since the fixed determinant version of Lemma \ref{lemma:codimUnstable} still holds by the precise same arguments. The only difference is that now one has to use the homotopy groups of the moduli space $\mathrm{\mathbf{M}}(r,\xi)$ of vector bundles of fixed determinant $\xi$, hence use \cite[Theorem 3.2]{DU95}.

\begin{theorem}
\label{thm:homotopygrps-fixed-det}
Let $X$ be a smooth complex curve of genus $g\ge 2$ and let $\SL$ be an algebroid on $X$ such that $\rk(\SL)=1$ and $\deg(\SL)<2-2g$. Suppose that $r$ and $d$ are coprime and that $r\ge 2$. Let $(\xi,\delta)\in \SM_{\SL}(1,d)$. Then, $\SM_{\SL}(r,\xi,\delta)$ is connected, hence irreducible. If, moreover, $(r,g)\neq(2,2)$, then its higher homotopy groups are given as follows:
\begin{itemize}
\item $\pi_1(\SM_{\SL}(r,\xi,\delta))=0$;
\item $\pi_2(\SM_{\SL}(r,\xi,\delta))\cong \ZZ$;
\item $\pi_k(\SM_{\SL}(r,\xi,\delta))\cong \pi_{k-1}(\mathcal{G}(\mathbb{E}))$, for every $k=3,\ldots,2(g-1)(r-1)-2$.
\end{itemize} 
\end{theorem}

\bibliographystyle{alpha}
\bibliography{LieAlgebroidConnectionsMotive}

\end{document}